\documentclass[a4paper,11pt,oneside,reqno]{amsart}
\usepackage{amssymb,amsmath,amsfonts,amsthm,amscd,amstext,mathtools}
\usepackage[english]{babel}
\usepackage[utf8x]{inputenc}
\usepackage[T1]{fontenc}
\usepackage[dvipsnames]{xcolor}            
\usepackage{xspace}
\usepackage[pdftex]{graphicx}
\usepackage{epstopdf}
\usepackage{mathrsfs}
\usepackage{stmaryrd}
\usepackage{enumitem}	
\usepackage{verbatim}
\usepackage{yfonts}
\usepackage{tikz}
\usepackage{tkz-tab}
\usepackage{float}
\usepackage[pdftex]{hyperref}
\makeatletter
\newcommand{\leqnos}{\tagsleft@true\let\veqno\@@leqno}
\newcommand{\reqnos}{\tagsleft@false\let\veqno\@@eqno}
\reqnos
\makeatother
\usepackage{caption}           
\usepackage{url}
\usepackage[a4paper,scale={0.72,0.74},marginratio={1:1},footskip=7mm,headsep=10mm]{geometry}
\usepackage{hyperref}
\hypersetup{					
	breaklinks=true,			
	colorlinks,
	citecolor=blue,
	filecolor=black,
	linkcolor=blue,
	urlcolor=black
}
\newcommand{\ind}{{ 1}}

\newcommand{\Cunder}{K_{\beta,\delta}}
\newcommand{\Ccrit}{K_{\beta,\delta}}
\newcommand{\Cover}{K_{\beta,\delta}}
\newcommand{\Pb}{\mathbf{P}}

\newcommand{\bP}{\mathbf{P}}

\newcommand{\bZ}{\mathbf{Z}}
\newcommand{\bE}{\mathbf{E}}

\newcommand{\Pbb}{\Pb_{\gb}}			
\newcommand{\Ebb}{\bE_{\gb}}			
\newcommand{\bbR}{\mathbb{R}}
\newcommand{\R}{\mathbb{R}}
\newcommand{\bbN}{\mathbb{N}}
\newcommand{\N}{\mathbb{N}}
\newcommand{\bbZ}{\mathbb{Z}}

\newcommand{\Var}{\mathbb{V}\mathrm{ar}}

\newcommand{\be}{\begin{equation}}
	\newcommand{\ee}{\end{equation}}		
\newcommand{\sA}{\mathscr{A}}

\newcommand{\sC}{\mathscr{C}}

\newcommand{\lmgf}{\mathcal{L}}

\newcommand{\bB}{\mathbf{B}}
\newcommand{\bh}{\mathbf{h}}

\newcommand{\Cc}{{\ensuremath{\mathcal C}}}
\newcommand{\G}{{\ensuremath{\mathcal G}}}

\newcommand{\cA}{{\ensuremath{\mathcal A}} }

\newcommand{\cE}{{\ensuremath{\mathcal E}} }
\newcommand{\cS}{{\ensuremath{\mathcal S}} }
\newcommand{\cH}{{\ensuremath{\mathcal H}} }
\newcommand{\cC}{{\ensuremath{\mathcal C}} }
\newcommand{\cN}{{\ensuremath{\mathcal N}} }
\newcommand{\cL}{{\ensuremath{\mathcal L}} }

\newcommand{\cT}{{\ensuremath{\mathcal T}} }
\newcommand{\cD}{{\ensuremath{\mathcal D}} }

\newcommand{\cV}{{\ensuremath{\mathcal V}} }
\newcommand{\cB}{{\ensuremath{\mathcal B}} }
\newcommand{\cR}{{\ensuremath{\mathcal R}} }

\newcommand{\cG}{{\ensuremath{\mathcal G}} }

\renewcommand{\epsilon}{\varepsilon}
\renewcommand{\phi}{\varphi}

\newcommand{\ga}{\alpha}
\newcommand{\gb}{\beta}
\newcommand{\gd}{\delta}
\newcommand{\gep}{\varepsilon}  

\newcommand{\gG}{\Gamma}

\newcommand{\gD}{\Delta}

\newcommand{\gl}{\lambda}

\newcommand{\cst}{({\rm{cst.}})}
\newcommand{\simL}{\underset{L\to \infty}{\sim}}
\newcounter{cst}[section]		
\newcounter{svf}[section]		
\newcommand{\beq}{\begin{equation}}
	\newcommand{\eeq}{\end{equation}}
\newtheorem{theorem}{Theorem}[section]

\newtheorem{proposition}[theorem]{Proposition}
\newtheorem{corollary}[theorem]{Corollary}
\newtheorem{lemma}[theorem]{Lemma}
\newtheorem{claim}[theorem]{Claim}
\theoremstyle{definition}
\newtheorem{definition}[theorem]{Definition}
\newtheorem{remark}[theorem]{Remark}
\numberwithin{equation}{section}			

\definecolor{darkorange}{rgb}{0.0, 0.5, 0.0}
\newcommand{\dd}{\mathrm{d}}		

\newcommand{\suptwo}[2]{\sup_{\substack{#1 \\ #2}}}

\newcommand{\sumtwo}[2]{\sum_{\substack{#1 \\ #2}}}

\renewcommand{\hat}{\widehat}
\renewcommand{\tilde}{\widetilde}

\DeclareMathOperator*{\argmax}{arg\,max}

\newcommand{\somme}[2]{\underset{#1}{\overset{#2}{\sum}}}
\newcommand{\produit}[2]{\underset{#1}{\overset{#2}{\prod}}}
\newcommand{\limite}[2]{\underset{#1 \longrightarrow #2}{\lim}}
\newcommand{\Z}[0]{\mathbb{Z}}
\newcommand{\C}[0]{\mathbb{C}}
\newcommand{\LoiMarchealeatoire}[1]{\mathbf{P}_\beta \Big( #1 \Big)}
\newcommand{\LoiMarchealeatoireSansParenthese}[0]{\mathbf{P}_\beta }
\newcommand{\EsperanceMarcheAleatoire}[1]{\mathbf{E}_\beta \Big( #1 \Big)}
\newcommand{\EsperanceMarcheAleatoiresansparenthese}[0]{\mathbf{E}_\beta}
\newcommand{\aN}{\textsl{a}_N}
\newcommand{\bN}{\textsl{b}_N}
\newcommand{\htildeq}[0]{\tilde{h}(q)}
\newcommand{\htildeqbar}[0]{\tilde{h}(\bar q)}
\newcommand{\sqdelta}[0]{s_{\delta}(q)}
\newcommand{\sqdeltaN}[0]{s_{\delta,N}(q)}
\newcommand{\sqdeltan}[0]{s_{\delta,n}(q)}

\newcommand{\qinf}[0]{q^*_\gd}
\newcommand{\qstar}{q^*_\gd}

\newcommand{\sign}{\textrm{sign}}
\newcommand{\penalite}[0]{T_\delta}
\newcommand{\deltaconcave}[0]{\check{\gd}}
\newcommand{\ProbaTiltee}[2]{\tilde{\mathbf{P}}_{#1} \Big( #2 \Big)}
\newcommand{\ProbaTilteeSansParenthese}[1]{\tilde{\mathbf{P}}_{#1}}
\newcommand{\EsperanceTiltee}[2]{\tilde{\mathbf{E}}_{#1} \Big( #2 \Big)}
\newcommand{\EsperanceTiltHomogen}[2]{\tilde{\mathbf{E}}_{#1} \Big( #2 \Big)}

\newcommand{\Proba}[2]{\mathbf{P}_{#1} \Big( #2 \Big)}
\newcommand{\ProbaSansParenthese}[1]{\mathbf{P}_{#1}}
\newcommand{\ProbaSurCritique}[1]{\mathbf{P}_{N,\rm{sup}}^{\delta,q}\left(#1 \right)}
\newcommand{\EsperanceSurCritique}[1]{\mathbf{E}_{N,\rm{sup}}^{\delta,q}\left(#1 \right)}
\newcommand{\ProbaSurCritiqueSansParenthese}[0]{\mathbf{P}_{N,\rm{sup}}^{\delta,q}}

\newcommand{\deltabar}{\underline{\delta}}
\newcommand{\cstexpo}{C_2}

\newcommand{\sfq}{\mathsf{q}}

\title[IPDSAW interacting with a vertical wall]{Phase diagram of the interacting partially directed self-avoiding walk attracted by a vertical wall}
\author{Elric Angot}
\address{Universit\'e de Nantes, Laboratoire Jean Leray, 2, rue de la Houssini\`ere, 44322 Nantes cedex 3, France}
\email{elric.angot@univ-nantes.fr}
\author{Nicolas P\'etr\'elis}
\address{Universit\'e de Nantes, Laboratoire Jean Leray, 2, rue de la Houssini\`ere, 44322 Nantes cedex 3, France}
\email{nicolas.petrelis@univ-nantes.fr}
\author{Julien Poisat}
\address{Université Paris-Dauphine, CNRS, UMR 7534, CEREMADE, PSL Research University, Place du Maréchal de Lattre de Tassigny, 75775 Paris cedex 16, France}
\email{poisat@ceremade.dauphine.fr}
\begin{document}
	\begin{abstract}
		In the present paper, we consider the interacting partially-directed self-avoiding walk (IPDSAW) attracted by a vertical wall. The IPDSAW was introduced by Zwanzig and Lauritzen (J. Chem. Phys., 1968) as a manner of investigating the collapse transition of a homopolymer dipped in a repulsive solvent. We prove in particular that a surface transition occurs inside the collapsed phase between (i) a regime where the attractive vertical wall does not influence the geometry of the polymer and (ii) a regime where the polymer is partially attached at the wall on a length that is comparable to its horizontal extension, modifying its asymptotic Wulff shape. The latter rigorously confirms the conjecture exposed by physicists in (Physica A: Stat. Mech. \& App., 2002). We push the analysis even further by providing sharp asymptotics of the partition function inside the collapsed phase.
	\end{abstract}
	\maketitle
	\let\thefootnote\relax\footnote{{\it Acknowledgements.}  The authors would like to thank the Centre Henri Lebesgue ANR-11-LABX-0020-01 for creating an attractive mathematical environment. JP acknowledges the support of ANR-22-CE40-0012 grant LOCAL.}
	\setcounter{tocdepth}{2}
	\tableofcontents
	\section*{Notation}
	Let $(a_L)_{L\geq 1}$ and $(b_L)_{L\geq 1}$ be two sequences of positive numbers. We will write 
	\beq
	a_L\underset{L\to \infty}{\sim}b_L \quad \text{if} \quad \lim_{L\to \infty} a_L/b_L=1.
	\eeq
	We will also write $\cst$  to denote generic positive constants whose value may change from line to line. We denote by $\bbN = \{1,2,3,\ldots\}$ the set of positive integers while $\bbN_0= \{0,1,2,\ldots\}$ is the set of non-negative integers.
	If $X = (X_i)_{i \in \N}$ is a random process, we note $X_I = (X_i)_{i \in I}$ for every $I\subseteq \N$ and abbreviate $\{X_I>0\}:=\{X_{i}>0,\ i \in I\}$. 
	\section{Introduction}
	The collapse transition of a polymer dipped in a repulsive solvent is a physical phenomenon that has been extensively studied in the physics literature (see e.g. \cite{refId0,flory1953principles} for theoretical background and more recently~\cite{Owczarek_2007} or \cite[Section 8]{Gut15} for computational background).
	There are so far very few mathematical models for which the collapse transition has been rigorously proven. Among this latter class of 
	models, the Interacting Partially Directed Self-Avoiding Walk  
	(IPDSAW) was initially introduced in \cite{ZL68} and investigated first with the help of combinatorial tools (see e.g. \cite{F90, OPB93}) and then, in the last 
	decade, thanks to a random walk representation of the trajectories. This probabilistic perspective allowed for a much deeper mathematical understanding of both the 
	phase transition and the geometric features of a typical trajectory sampled from the polymer measure, in each regime (see \cite{CNPT18} for a review).
	
	Physicists have also considered the effect of an interaction between the polymer and the container inside which the poor solvent is kept, see e.g. \cite{MGSY02,RGSY02}. Such an additional interaction with the bottom of the container triggers a surface transition inside the collapsed phase of the IPDSAW, which was recently put on rigorous grounds in~\cite{LP22}. In the present paper, we focus on another interaction of interest, that is an attractive interaction between the polymer and one of the vertical walls of the container. In particular, we display the phase diagram of the model and exhibit another surface transition inside the collapsed phase. 
	
	When only the  solvent-monomers interactions are taken into account, it was shown in \cite{Legrand_2022} that, inside the collapsed phase, a typical configuration of the polymer  is made 
	of a macroscopic volume called {\it bead}, which is unique since only finitely many monomers are to be found outside this bead. 
	For a polymer of length $L\in \N$, this bead, once rescaled horizontally and vertically by $\sqrt{L}$ converges 
	in probability towards a deterministic Wulff shape (see \cite{CNP16}). 
	In \cite{LP22},  the polymer is investigated inside its collapsed phase and some additional interactions are taken into account between the monomers and the bottom of the container.
	In order to keep the model tractable,  a geometric restriction has been imposed on the allowed configurations, namely they are required to describe a unique bead. 
	In the present paper, although we consider additional interactions as well (this time between the monomers and the vertical walls), we managed to get rid of the single bead restriction and display our result in the general framework. 
	From that perspective, the results displayed here are more ambitious. 
	\subsection{The IPDSAW with an attractive wall}
	The configurations of the polymer   are modeled by 
	random walk paths on $\mathbb{Z}^2$ that are \emph{self-avoiding} and take exclusively unitary steps \emph{upwards, downwards and to the right} (see Fig.~\ref{figure : sans bulles}). 
	The fact that the polymer is placed in a repulsive solvent is taken into account  by assuming that the monomers try to exclude the solvent and therefore attract one another. 
	For this reason, any pair of non-consecutive vertices of the walk that are adjacent on the lattice is called \textit{self-touching} and the interactions between monomers are taken into account by assigning an energetic reward $\beta$ to the polymer for each self-touching. In the present paper, we take into account another interaction between the polymer and the medium around it, namely 
	an attraction of the monomer at the \textit{vertical wall} of the container. This interaction is of intensity $\delta$. Note that we consider non-negative interactions, i.e. $(\beta,\delta) \in \mathcal{Q} := [0, +\infty)^2$. 
	It is convenient to represent the configurations of the model as collections of oriented vertical stretches separated by horizontal steps. To be more specific,  for a polymer made of $L\in \N$ monomers, the set of allowed 
	paths is $\Omega_L:=\bigcup_{N=1}^L\mathcal{L}_{N,L}$, where $\mathcal{L}_{N,L}$ 
	consists of all the collections made of  $N$ vertical stretches that have a total length $L-N$, that is
	\begin{equation}\label{defLL}
		\textstyle\mathcal{L}_{N,L}=\Bigl\{\ell:=(\ell_i)_{i=1}^N \in\mathbb{Z}^N:\sum_{n=1}^N|\ell_n|+N=L\Bigr\}.
	\end{equation}
	\begin{figure}[H]
		\centering
		\includegraphics[width=0.5\textwidth]{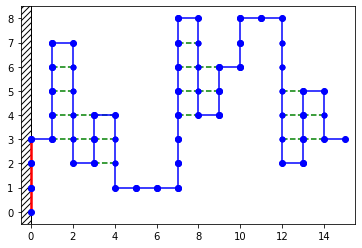}
		\caption{Picture of the trajectory $\ell\in \mathcal{L}_{15,54}$ whose vertical stretches are $(3,4,-5,2,-3,0,0,7,-4,2,2,0,-6,3,-2)$. The wall interaction is highlighted in red, and the self-interaction is represented by a dashed line.}
		\label{figure : sans bulles}
	\end{figure}
	With this representation,  the modulus of a given stretch
	corresponds to the number of monomers constituting this stretch  (and
	the sign gives the direction upwards or downwards). For convenience, we require every configuration to end with a horizontal step, and we note that any two consecutive vertical
	stretches are separated by a horizontal step. The latter explains why $\sum_{n=1}^N |\ell_n|$ must equal $L-N$ in order for $\ell=(\ell_i)_{i=1}^N$ to be associated with a polymer made of $L$ monomers (see Fig.~\ref{figure : sans bulles}). 
	We define the set of all trajectories as $\Omega=\cup_{L\geq 1} \Omega_L$ and for a given trajectory $\ell\in \Omega$, we let $N_\ell$ be
	its horizontal extension (that is also its number of vertical stretches)  and $|\ell |$ be its total length, i.e., $\ell \in \cL_{N_\ell,|\ell|}$.
	The interactions between the polymer and the medium around it are taken into account in a Hamiltonian associated with each path $\ell\in \Omega_L$ and denoted by 
	$\textstyle H_{L,\beta,\delta}(\ell)$. To be more specific, for every configuration $\ell\in \Omega_L$, the attraction between the vertical hard wall and the polymer holds along 
	the first vertical stretch of the configuration as $\delta |\ell_1|$. Moreover, the repulsion between the monomers and the solvent is taken into account by
	rewarding energetically those pairs of consecutive stretches with opposite directions, i.e.,
	\begin{equation}\label{defH}
		\textstyle H_{L,\beta,\delta}(\ell_1,\ldots,\ell_N)=\delta |\ell_1|+\beta\sum_{n=1}^{N-1}(\ell_n\;\tilde{\wedge}\;\ell_{n+1}),
	\end{equation}
	where
	\begin{equation}
		x\, \tilde\wedge\, y=\begin{dcases*}
			|x|\wedge|y| & if $xy<0$,\\
			0 & otherwise.
		\end{dcases*}
	\end{equation}
	With the Hamiltonian  in hand we can define the polymer measure as
	\begin{equation}\label{defpolme}
		P_{L,\beta,\delta}(\ell)=\frac{e^{H_{L,\beta,\delta}(\ell)}}{Z_{L,\beta,\delta}},\quad \ell \in \Omega_L,
	\end{equation}
	where $Z_{L,\beta,\delta}$ is the partition function of the model, i.e.,
	\begin{equation}\label{pff}
		Z_{L,\beta,\delta}=\sum_{N=1}^{L}\sum_{\ell\in\mathcal{L}_{N,L}} \,  e^{H_{L,\beta,\delta}(\ell)}.
	\end{equation}
	\subsection{Reminder on the model without a wall}
	The particular case where the interaction between the monomers and the vertical wall is switched off (corresponding to $\delta=0$)
	has been studied in depth in \cite{CNP16, Legrand_2022, nguyen2013variational}. In this case the existence of the exponential growth rate of the partition function sequence $(Z_{L,\beta,0})_{L\geq 1}$
	is obtained by subadditivity (in $L$) of the logarithm of the former sequence combined with Fekete's lemma. Then, the free energy defined as
	\begin{equation}\label{feww}
		f(\beta,0) = \lim_{L\to \infty} \frac{1}{L}\log Z_{L,\beta,0},
	\end{equation}
	allows us  
	to divide the phase diagram into (i) an 
	{\it extended} phase $[0,\beta_c)=\{\beta\ge0\colon 
	f(\beta,0)>\beta\}$ and (ii) a {\it collapsed} phase $[\beta_c,\infty)=\{\beta\ge0\colon 
	f(\beta,0)=\beta\}$.
	Note that the inequality $f(\beta,0)\geq \beta$ is easily obtained with the following observation. 
	For $L\in \{k^2 \colon k\in \N\}$, we restrict the partition function 
	to a single trajectory $\tilde \ell\in \cL_{\sqrt{L},L}$ defined as  
	\begin{equation}\label{deftwil}
		\tilde \ell_i:=(-1)^{i-1} (\sqrt{L}-1)\quad \text{ for}
		\quad  
		i\in  \{1,\dots,\sqrt{L}\}.
	\end{equation}
	Thus,  the Hamiltonian of $\tilde \ell$ at $\delta=0$ equals $\beta (\sqrt{L}-1)^2=\beta L (1+o_L(1))$ which 
	guarantees us that $f(\beta,0)\geq \beta$.
	\subsection{Outline of the paper}
	In Section \ref{Results}, we state and comment the most important results of the present paper. 
	To begin with, we describe the three different phases ({\it extended}, {\it collapsed} and {\it glued}) into which the phase diagram is divided. 
	Then, we present the surface transition that splits the collapsed phase into three regimes ({\it desorbed-collapsed}, {\it critical} and {\it adsorbed-collapsed}).  We provide the formula of the associated critical curve and we display some sharp asymptotics 
	of the partition function in each regime.
	In Section \ref{prphdiag}, the phase transitions  are proven rigorously, namely the existence of critical curves separating the three aforementioned phases. We take this opportunity to introduce the random walk representation that allows us to provide an alternative expression of the partition function using a random walk of law $\Pbb$ (defined in \eqref{Nantes 3.3}).  In Section \ref{prepsec} we introduce notation and auxiliary mathematical tools that are required to prove our main results. Thus, in Section \ref{auxpartfun}, we settle a class of auxiliary partition functions involving a random walk of law $\Pbb$ constrained to enclose an atypically large area. Some sharp asymptotics of those 
	partition functions are provided in Section \ref{secsharpaux} and proven in Section \ref{proofasympt}. In section \ref{beaddec} we display a method to decompose each polymer trajectory into beads
	that consist of collections of non-zero vertical stretches whose signs alternate. Such decomposition is useful when working inside the collapsed phase because a typical trajectory sampled from the polymer measure turns out to be made of a unique macroscopic bead. Sections \ref{sec:COM} and \ref{sec:analysis_aux} are dedicated to two tilted versions of the law of a random walk under $\Pbb$. One tilting is homogeneous in time whereas the other one is inhomogeneous. Both versions will be applied to study random walk trajectories enclosing an abnormally large area. Some local limit theorems are stated in Section \ref{loclim} concerning both the arithmetic area and the final position of a random walk sampled from the (above mentioned) inhomogeneous tilting of $\Pbb$. Finally, some bounds on the polymer horizontal extension inside the collapsed phase
	are displayed in Section \ref{sec:a-priori-bounds}.   With Section \ref{prodr}, we prove the existence of the surface transition and compute its critical curve.  Finally, with Sections \ref{proofasympt}   and \ref{compid} we prove the  asymptotics of the partition function corresponding to each of the three regimes in the collapsed phase.
	\section{Results}\label{Results}
	We distinguish between two types of results. First, in Section
	\ref{phdiag} below, we describe the phase diagram of the model which is divided
	into three phases:  {\it extended},  {\it collapsed} and {\it glued} (denoted respectively by $\cE$, $\cC$ and $\cG$). A typical trajectory sampled from $P_{L,\beta,\delta}$ has a horizontal extension of order $L$ inside $\cE$, $o(L)$ inside $\cC$, and finally, inside $\cG$, such trajectory takes only finitely many horizontal steps after a very long vertical stretch attached to the wall. The second type of results consists in analyzing more in depth the collapsed phase. The free energy in $\cC$ is equal to $\beta$, which guarantees us that there is no phase transition inside $\cC$. However, depending on the value of $(\beta,\delta)\in \cC$ we will 
	observe that the behavior of the polymer with respect to the attractive wall may change drastically.  This latter phenomenon is associated with a \emph{surface transition} taking place along a critical curve that divides $\cC$ into three regimes: 
	\begin{itemize}
		\item A {\bf desorbed-collapsed} ($\cD\cC$) regime inside which 
		$\delta$ is not large enough  to pin the first vertical stretch of the polymer to the wall. Thus, the first vertical stretch remains of length
		$O(1)$;
		\item An {\bf adsorbed-collapsed} ($\cA\cC$) regime inside which $\delta$ is large enough for the polymer to be pinned at the attractive wall
		along its first vertical stretch, on a length $O(\sqrt{L})$;
		\item A {\bf critical} regime at which the first vertical stretch has length $O(L^{1/4})$.
	\end{itemize}
	\subsection{Phase transition}\label{phdiag}
	Let us denote by $f(\beta,\delta)$ the free energy of the system, that is the exponential growth rate of $(Z_{L,\beta,\delta})_{L\geq 1}$. 
	It turns out that $f(\beta,\delta)$ may actually be expressed as the maximum of $\delta$ and $f(\beta,0)$.  We recall that $f(\beta,0) \geq \beta$ by \eqref{deftwil} and that $\mathcal{Q}:=[0,+\infty)^2$.
	\begin{proposition}\label{existencefe}
		For every $(\beta,\delta) \in \mathcal{Q}$, the following limit exists:
		\begin{equation}\label{deffree}
			\limite{L}{\infty} \frac{1}{L} \log Z_{L,\beta,\delta} = f(\beta,\delta) \in [0,\infty),
		\end{equation}
		and $f(\beta,\delta) = f(\beta,0) \vee \delta$.
		\label{Nantes Proposition 1}
	\end{proposition}
	\par The free energy allows us to distinguish between three phases: {\it collapsed} ($\cC$), {\it extended} ($\cE$) and {\it glued} along the vertical wall ($\cG$). 
	We let $\gb_c$ be the unique positive solution to the equation $\Gamma_\beta=1$ with 
	\beq\label{eq:defGammaBeta}
	\Gamma_\beta=c_\beta\, e^{-\beta}=\frac{e^{-\beta}+e^{-3\beta/2}}{1-e^{-\beta/2}}
	\eeq 
	with $c_\beta$ properly defined in \eqref{Nantes 3.3}. 
	\begin{definition}\label{three phases}
		Rigorously, the three phases are
		\begin{itemize}
			\item $\cE:=\{(\beta,\delta)\in \mathcal{Q} \colon\, \beta< \beta_c {\rm\ and\ }    f(\beta,\delta)>\delta\}$
			\item $\cC:=\{(\beta,\delta)\in  \mathcal{Q} \colon\, \beta\geq \beta_c {\rm\ and\ }  f(\beta,\delta)=\beta\}$
			\item $\cG:=\{(\beta,\delta)\in  \mathcal{Q} \colon\, f(\beta,\delta)=\delta\}$
		\end{itemize}
	\end{definition}
	With Proposition \ref{existencefe}, we may rewrite these three phases as follows:
	\begin{align*}
		\cE&:=\{(\beta,\delta)\in \mathcal{Q} \colon\, \beta< \beta_c,\ \delta\leq f(\beta,0)\},\\
		\cC&:=\{(\beta,\delta)\in \mathcal{Q} \colon\, \beta\geq \beta_c,\ \delta\leq \beta\},\\
		\cG&:=\{(\beta,\delta)\in \mathcal{Q} \colon\, \delta> f(\beta,0)\}.
	\end{align*}
	Fig.\ \ref{fig:ipdsaw} provides a picture of the phase diagram. We observe that the boundaries meet at the tri-critical point $(\beta_c, \beta_c)$.
	\begin{figure}[H]
		\includegraphics[scale=0.5]{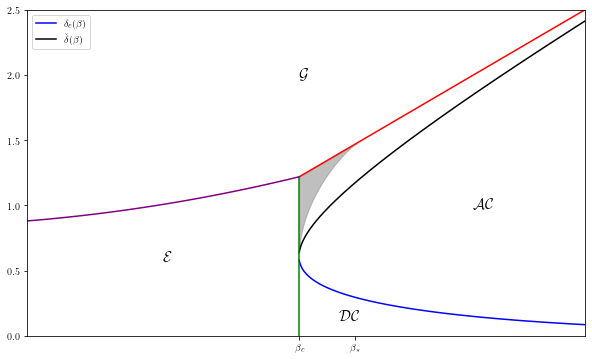}
		\caption{{\small Phase diagram of the polymer pinned at the vertical wall. The three phases $\cE$, $\cC$ and $\cG$ are separated by the purple ($\gd = f(\gb,0)$), red ($\gd = \gb$) and green ($\gb = \gb_c$) curves. The surface transition between the regimes $\cD\cC$ and $\cA\cC$ is indicated by the blue curve ($\delta = \delta_c(\gb)$), for which we have an explicit expression. A change of convexity happens for the Wulff shape at the black curve ($\gd = \check \gd(\gb)$). The bounded grey set $\cC_{\rm bad}$, see~\eqref{eq:defCgood}, is where our method fails and has been computed numerically. The first coordinate of the rightmost point in $\cC_{\rm bad}$, denoted by $\gb_*$, is smaller than $\pi /\sqrt{3}\approx 1.81$ (rigorously) and around 1.47 (numerically).}}
		\label{fig:ipdsaw}
	\end{figure}
	The phase transitions being now identified, the rest of the present section is dedicated to the collapsed phase $\cC$ and in particular to the surface transition 
	that takes place inside $\cC$.
	\subsection{Surface transition}\label{surftransit}
	Figuring out the regime associated with a given coupling parameter  $(\beta,\delta)\in \cC$ requires a detailed analysis of the {\it second-order} term of the exponential growth rate of the partition function sequence $(Z_{L,\beta,\delta})_{L\in \N}$. For this reason,  we set for $L\in \N$, 
	\begin{equation}\label{deftilz}
		\tilde Z_{L,\beta,\delta}=Z_{L,\beta,\delta}\  e^{-\beta L}.
	\end{equation}
	We will prove that the exponential growth rate of $(\tilde Z_{L,\beta,\delta})_{L\in \N}$ is $\sqrt{L}$ with a prefactor 
	$g(\beta,\delta)$ which loses analyticity precisely where the polymer switches from $\cA\cC$ to $\cD\cC$. For $\beta>\beta_c$ and $\delta <\beta$, we denote by
	\begin{equation}\label{second}
		g(\beta,\delta):=\lim_{L\to \infty} \frac{1}{\sqrt{L}} \log \tilde Z_{L,\beta,\delta} 
		=\lim_{L\to \infty} \frac{1}{\sqrt{L}} \Big( \log Z_{L,\beta,\delta} - \beta L \Big),
	\end{equation}
	the so-called {\it surface} free energy (in contrast with the {\it volume} free energy)
	provided that the limit exists. The adsorbed-collapsed regime and the desorbed-collapsed regime may be  rigorously defined as 
	follows:
	\begin{align}
		\mathcal {DC}&:=\{(\beta,\delta) \in \cC \colon \,  g(\beta,\delta)=g(\beta,0)\}, \\
		\nonumber \mathcal {AC}&:=\{(\beta,\delta) \in \cC \colon \,  g(\beta,\delta)>g(\beta,0)\}.
	\end{align}
	Since  for every $\beta> \beta_c$, the function $\delta\mapsto g(\beta,\delta)$ is obviously non-decreasing, 
	we may define the critical curve as 
	\begin{equation}\label{exprcourbecrit}
		\delta_c(\beta):=\inf\{\delta>0\colon\, g(\beta,\delta)>g(\beta,0)\}, \qquad \beta>\beta_c
	\end{equation}
	so that 
	\begin{align}
		\mathcal {DC}&:=\{(\beta,\delta) \in \cC \colon \, \delta\leq \delta_c(\beta) \}, \\
		\nonumber \mathcal {AC}&:=\{(\beta,\delta) \in \cC \colon \, \delta >\delta_c(\beta) \}.
	\end{align}
	In the case that $\gb > \gb_c$ and $\gd=0$, it is known from \cite[Eq. (1.27)]{CNP16} that
	\beq
	g(\gb, 0) < 0.
	\eeq
	We provide a variational formula for $g$ in Theorem~\ref{varfor} below. It turns out that, for technical reasons, this result only holds in a subset of the collapsed phase, which we call $\cC_{\text{good}}$ and whose precise definition in~\eqref{eq:defCgood} below calls for additional notation and lemmas. Let us slightly anticipate by pointing out that the complement $\cC\setminus\cC_{\text{good}}$ is a bounded subset of $\cA\cC$, located far away from the surface transition critical curve. In particular, for $\gb$ large enough, $(\gd,\gb)\in\cC_{\text{good}}$ for every $\gd\in[0,\gb)$.
	Providing an analytic expression of $g$ requires to introduce 
	a handful of auxiliary functions. To that aim, we introduce a probability law $\Pbb$ on $\Z$ (with $\Ebb$ its associated expectation) as 
	\begin{equation}
		\LoiMarchealeatoire{\cdot =k} = \frac{e^{-\frac{\beta}{2}|k|}}{c_\beta}\ , \qquad c_\beta := \somme{k \in \Z}{} e^{-\frac{\beta}{2}|k|} = \frac{1+e^{-\beta/2}}{1-e^{-\beta/2}}.
		\label{Nantes 3.3}
	\end{equation}
	We set $X_0=0$ and we let $X=(X_i)_{i\geq 0}$ be a random walk with i.i.d.\ increments of law $\Pbb$.   
	Throughout the paper we will need to consider  trajectories of $X$ that enclose an abnormally large area. This leads us to apply some tilting 
	procedures to  the  increments of $X$ that will be explained in more details in Section~\ref{sec:COM}. All functions below arise in this context, namely
	\begin{equation}\label{def:cL}
		\lmgf(h) := \log \Ebb [e^{hX_1}], \quad \quad  h\in \big(-\tfrac{\beta}{2}, \tfrac{\beta}{2}\big),
	\end{equation}
	and $\G : (-\beta,\beta) \longrightarrow \R$ as
	\begin{equation}
		\G(h) :=  \int_0^1 \lmgf\Big(h\Big(x-\frac{1}{2}\Big)\Big) \dd x.
		\label{Nantes 4.14}
	\end{equation}
	For every $q\geq 0$  we denote by $\htildeq$  the unique solution in $h \in [0,\beta)$ of 
	\begin{equation} \label{Nantes 3.7}
		\G'(h)=\int_0^1 (x-\tfrac12) \lmgf'\Big(h\Big(x-\tfrac{1}{2}\Big)\Big)\, \dd x=q.
	\end{equation}
	Then, for $\delta \in [0,\beta)$, we define $\cH_\delta:(-\delta,\beta-\delta)\longrightarrow \R$ as
	\begin{equation}
		\cH_\delta(s) :=  \int_0^1 \lmgf\left(s x+\delta-\frac{\beta}{2}\right) \dd x,
		\label{Nantes 44.14}
	\end{equation}
	and we denote by $\sqdelta$ the unique solution in $s\in (-\delta,\beta-\delta)$
	of 
	\begin{equation}
		\cH_\delta'(s) =q. 
		\label{Nantes 3.10}
	\end{equation}
	At this stage, we introduce the function 
	\begin{equation}\label{deftildepsi}
		\psi(q,\delta):=
		\begin{cases}
			-q \tilde h(q)+\cG(\tilde h(q)) &\text{for}\quad 0\leq \delta\leq \delta_0(q),\\
			-q \sqdelta+\cH_\delta(\sqdelta) & \text{for}\quad \delta_0(q)<\delta<  \beta,
		\end{cases}
	\end{equation}
	where 
	\beq\label{eq:delta0q}
	\delta_0(q) := \frac{\beta}{2}- \frac{\tilde h(q)}{2}.
	\eeq
	We observe in particular that $\psi(q,\delta)=\psi(q, 0)$ as long as $\delta\in [0, \delta_0(q)]$. Note that $\psi(q,\delta)$ will be the exponential growth rate of a 
	sequence of auxiliary partition functions introduced in Section \ref{prepsec} (see \eqref{eq:D_Nqdelta}).
	As mentioned above, there is a tiny subset of $\mathcal{C}$, which we denote as $\mathcal{C}_{\text{bad}}$, inside which  the variational characterization of $g$ given in Theorem \ref{varfor} below is not valid. To be more specific,  $\mathcal{C}_{\text{bad}}:=\cC\setminus \cC_{\text{good}}$ with
	\beq
	\label{eq:defCgood}
	\cC_{\rm good} = \{(\beta,\delta) \in \cC \colon \gd \le \bar \gd(\gb) \},
	\eeq
	with
	\beq
	\label{eq:def-bar-gd}
	\bar \gd(\gb) =\beta\wedge\inf\{\gd\in(0,\gb)\colon \cH'_\gd(\gb/2 - \gd -x_\gb) >0\},
	\eeq
	and where $x_\gb$ is the unique solution in $(0,\gb/2)$ of $\lmgf(x)=-\log \Gamma_\gb$. Note in particular that $\bar \gd(\gb) = \gb$ when the set in the r.h.s.\ in \eqref{eq:def-bar-gd} is empty.
	\begin{theorem}\label{varfor}
		For $(\beta,\delta) \in \cC_{\rm{good}}$, the limit in~\eqref{second} exists and equals
		\begin{equation}\label{eq:gbetadelta-varfor}
			g(\beta,\delta)=\max\{ a \log \Gamma_\beta+a\,  \psi(\tfrac{1}{a^2},\delta)\colon\, a>0\}.
		\end{equation}
	\end{theorem}
	\subsection{Critical curve and order of the surface transition}
	With the following theorem, we provide an analytic expression of the critical curve and state that the surface transition inside the collapsed phase is second-order. It turns out that the critical value of $\delta$ corresponds to the value in~\eqref{eq:delta0q} for a suitable choice of $q$:
	\begin{theorem}\label{surfacetransition}
		For $\beta>\beta_c$\ ,
		\begin{equation}\label{closedexprdeltac}
			\delta_c(\beta)=\frac{\beta}{2} - \frac{\tilde{h}(a_\beta^{-2})}{2},
		\end{equation}
		where
		\begin{equation}\label{eq:defqbeta}
			a_\beta:=\argmax\{a \log \Gamma_\beta+a \psi(\tfrac{1}{a^2},0)\colon a>0\}.
		\end{equation}
		The critical curve admits the following explicit expression:
			\beq
			\label{eq:delta_c_explicit_exp}
			\delta_c(\beta) = \log(\cosh(\gb)-\sqrt{\cosh(\gb)^2 - e^{\gb}}),
			\eeq
		and
		\beq
		\label{eq:delta_c_low_temp_asympt}
		\delta_c(\beta) = e^{-\gb}[1+O(e^{-\gb})], \qquad \text{as }\gb\to\infty.
		\eeq
		Moreover, there exists a positive constant $C_\beta$ such that:
		\begin{equation}\label{secorder}
			g(\beta,\delta_c(\beta)+\varepsilon)-g(\beta,\delta_c(\beta))\underset{\varepsilon\to 0}{\sim} C_\beta \, \varepsilon^2.
		\end{equation}
	\end{theorem}
	Note that the existence and uniqueness of $a_\beta$ will be guaranteed by Lemmas \ref{concavity} and \ref{lem:limits}. Moreover, an   explicit expression of $C_\beta$ above can be found in \eqref{Nantes C.22}, and an observation on the large $\gb$-limit of $q_\gb := a_\gb^{-2}$ (interpreted in terms of the model without a wall) is stated in Proposition~\ref{Nantes Lemma 8.3}. The explicit expression in~\eqref{eq:delta_c_explicit_exp} appeared in~\cite[Equation (21)]{RGSY02}.
	\subsection{Sharp asymptotics of the partition function}\label{sharpasympt}
	With Theorems \ref{varfor} and  \ref{surfacetransition} above we analytically characterized the surface transition.  With Theorem \ref{asymptotics} below, we push one step further our analysis of the partition functions by providing {\it sharp} asymptotics. In particular, we answer a group of questions raised in \cite[p.\ 10]{Gut15} 
	and prove that for our model, following the notation therein, $\sigma=1/2$, $\mu= \exp(\gb)$, $\mu_1=\exp(g(\beta,\delta))$ and $g=-1/2$ inside $\cA\cC$ (including the critical regime) or $g=-3/4$ inside $\cD\cC$.
	\begin{theorem}\label{asymptotics}
		For $\beta>\beta_c$ we have in each of the three regimes, as $L\to \infty$:
		\begin{enumerate}
			\item If $\delta<\delta_c(\beta)$ then there exists a positive constant $\Cunder$ such that
			\begin{equation}\label{parfunqsouscrit} 
				Z_{L,\beta,\delta} \simL \frac{\Cunder}{L^{3/4}} \, e^{\beta L+g(\beta, 0)\sqrt{L}}.
			\end{equation}
			\item  If $\delta=\delta_c(\beta)$ then there exists a positive constant $\Ccrit$ such that
			\begin{equation} \label{parfunqcrit}
				Z_{L,\beta,\delta} \simL \frac{\Ccrit}{\sqrt{L}} \, e^{\beta L+g(\beta,\delta) \sqrt{L} }.
			\end{equation}
			\item  If $\delta>\delta_c(\beta)$ and $(\beta,\delta) \in \cC_{\rm{good}}$ then there exists a positive constant $\Cover$ such that
			\begin{equation}\label{parfuneqsupcrit}
				Z_{L,\beta,\delta}\simL\frac{\Cover}{\sqrt{L}} \, e^{\beta L+g(\beta,\delta) \sqrt{L} }.
			\end{equation}
		\end{enumerate}
	\end{theorem}
	Explicit expressions for the constants above can be found in Section~\ref{proofasympt}.%
	\subsection{Uniqueness of the macroscopic bead}
	We close this section with a result concerning  the geometry of the partially-directed self-avoiding walk under the polymer measure.
	To prove this result, we first need to break down every trajectory into a series of {\it beads}. These beads are sub-trajectories consisting of non-zero vertical stretches that alternate in direction. We shall expand on this notion in Section~\ref{beaddec}. In the context of the collapsed phase, physicists have been interested in determining whether a typical trajectory contains a single large bead or multiple smaller ones, and whether the large one touches the vertical wall or not. Thus, for every $\ell \in \Omega_L$ we let $N_\ell$ be its horizontal extension (i.e., $\ell \in L_{N_\ell}, L$) and also $|I_{\text{max}}(\ell)|$ be the length of its largest bead, i.e.,
	\begin{equation}
		|I_{\max}(\ell)| := \max \left\{ \sum_{i=u}^{v} (1 + |\ell_i|) : 1 \leq u \leq v \leq N_\ell,\  \ell_i \ell_{i+1} < 0 \,\  \forall \, u \leq i < v \right\}.
	\end{equation}
	We set $I_0$ the length of the first bead, i.e:
	\begin{equation}
		|I_{0}(\ell)| :=  \sum_{i=1}^{\tau_1} (1 + |\ell_i|),
	\end{equation}
	with $\tau_1$ defined as the end of the first bead, i.e.\ $\tau_1 = \sup \{n \geq 0 : \exists k \in \N,\ \ell_0 = ... = \ell_{k-1} = 0,\ \text{ and }\ \forall k \leq i \leq n-1,\ \ell_i \ell_{i+1} < 0\}$ (with the convention that $\ell_0=0$).
	Our next theorem states that there is a unique macroscopic bead in the collapsed phase (in agreement with previous work of \cite{Legrand_2022}). Moreover, in the adsorbed-collapsed phase, this unique bead is necessarily the first one in the bead decomposition of the trajectory, that is the one pinned at the wall.
	\begin{theorem} \label{The One Bubble}
		For all $\beta > \beta_c$ and $\delta \in [0,\delta_c(\beta)]$,
		\begin{equation}
			\lim_{k \to \infty} \liminf_{L \to \infty} P_{L, \beta,\delta} \left( |I_{\max}(\ell)| \geq L - k \right) = 1.
			\label{Nantes E.45}
		\end{equation}
		For all $\beta > \beta_c$ and $\delta \in [\delta_c(\beta),\beta)$, 
		\begin{equation}
			\lim_{k \to \infty} \liminf_{L \to \infty} P_{L, \beta,\delta} \left( |I_{0}(\ell)| \geq L - k \right) = 1.
			\label{Nantes E.46}
		\end{equation}
	\end{theorem}
	We prove this theorem here, as it directly follows from Theorem~\ref{asymptotics}.
	\begin{proof}[Proof of Theorem \ref{The One Bubble}]
		Let us start with the proof of \eqref{Nantes E.46}. 
		A rough upper bound on the contribution to the partition function of those trajectories whose first bead has length $i$ gives
		\begin{align}\label{sumex}
			\nonumber P_{L, \beta,\delta}(|I_0(\ell)| \leq L-k) 
			&\leq \somme{i=1}{L-k} \frac{Z_{i,\beta,\delta} Z_{L-i,\beta, 0}}{Z_{L,\beta,\delta}} \\
			&\leq  \cst \sum_{i=1}^{L-k} \frac{L^{1/2}}{i^{1/2}(L-i)^{3/4}}\,  e^{\, g(\gb,\gd)\, (\sqrt{i}-\sqrt{L})+g(\gb,0)\, \sqrt{L-i}},
		\end{align}
		where we have used the asymptotics (2) and (3) in  Theorem \ref{asymptotics} to obtain the second inequality.
		After considering separately the case $i\leq L/2$ and $L/2\leq i \leq L-k$, we observe that 
		\begin{equation}\label{boundpolfrac}
			\max_{i\in \{1,\dots,L-k\}} \frac{L^{1/2}}{i^{1/2}\, (L-i)^{3/4}} \leq \max\Big\{\frac{2^{3/4}}{L^{1/4}},\frac{\sqrt{2}}{k^{3/4}}\Big\}
		\end{equation}
		so that for $L$ large enough the l.h.s.\ in \eqref{boundpolfrac} is smaller that $\sqrt{2}/k^{3/4}$.
		We also need to bound from above the exponential terms in the sum in \eqref{sumex}. Using that $g(\gb,\gd)\ge g(\gb,0)$, we obtain: 
		\beq
		g(\gb,\gd)(\sqrt{i}-\sqrt{L})+g(\gb,0)\sqrt{L-i} \leq g(\gb,0)(\sqrt{L-i}+\sqrt{i}-\sqrt{L}).
		\eeq
		We observe that 
		\begin{equation}\label{equalracine}
			\sqrt{L-i}+\sqrt{i}-\sqrt{L}=\sqrt{L} \Big[ \Big(1+2\sqrt{\tfrac{i}{L}(1-\tfrac{i}{L})}\Big)^{1/2}-1\Big]
		\end{equation}
		and that $2\sqrt{\tfrac{i}{L}(1-\tfrac{i}{L})}\in [0,1]$ for every $i\in \{0,\dots,L\}$. At this stage, given that 
		$\sqrt{1+x}-1\geq x/4$ for $x\in [0,1]$ we derive from \eqref{equalracine} that
		$$\sqrt{L-i}+\sqrt{i}-\sqrt{L}\geq \tfrac{1}{2} \,  \sqrt{i(1-\tfrac{i}{L})}.$$
		Since $g(\beta,0)<0$~\cite[Eq. (1.27)]{CNP16}, we can rewrite \eqref{sumex} as
		\begin{align}\label{ismall}
			P_{L, \beta,\delta}(|I_0(\ell)| \leq L-k) 
			&\leq  \cst \frac{1}{k^{3/4}} \sum_{i=1}^{L-k}  
			e^{\,\frac12 g(\gb,0)\, \sqrt{i(1-\tfrac{i}{L})}}.
		\end{align}
		The sum in the r.h.s.\ in \eqref{ismall} may be bounded above by 
		\begin{align}
			\nonumber \sum_{i=1}^{L-k}  
			e^{\,\frac12 g(\gb,0)\, \sqrt{i(1-\tfrac{i}{L})}}&\leq \sum_{i=1}^{L/2}  
			e^{\,\frac{1}{4}  g(\gb,0)\,   \sqrt{i} }+ \sum_{i=L/2}^{L-k}  
			e^{\,\frac{1}{4}  g(\gb,0)\,   \sqrt{L-i} }\\
			&\leq 2\, \sum_{i\ge1}
			e^{\,\frac{1}{4}  g(\gb,0)\,   \sqrt{i} }<\infty.
		\end{align}
		This completes the proof of \eqref{Nantes E.46}.
		\par Since the proof of \eqref{Nantes E.45} is almost identical to the proof of \cite[Theorem 2.2]{Legrand_2022},  we will not reproduce it here with every detail. However, let us stress that Item (1) in Proposition \ref{asymptoticsinter} will play the role of \cite[equation (6.4)]{Legrand_2022} which is the key to obtain the result.
	\end{proof}
	\subsection{On the shape of the bead}
	In this section we discuss an (unexpected) consequence of our analysis concerning the convexity of the globule (i.e.\ the unique macroscopic bead). Let us first recall that in the absence of the pinning interaction ($\gd=0$), it was proven in~\cite{CNP16} that the properly rescaled polymer converges, in the Hausdorff distance, to the following (convex) set
	\beq
	\cS_\beta = \Big\{(x,y)\in\bbR^2\colon x\in[0,a_\gb],\ |y| \le \tfrac12 a_\gb \mathsf{W}_\gb(x/a_\gb)\Big\},
	\eeq
	where $a_\gb$ is as in~\eqref{eq:defqbeta} and the so-called Wulff shape (that is a concave curve)
	\beq
	\mathsf{W}_\gb(t) = \int_0^t \lmgf'\Big((\tfrac12 - x) \tilde h(a_\gb^{-2})\Big)\dd x, \qquad t\in[0,1],
	\eeq
	is intimately linked to the tilt function set forth in~\eqref{Nantes 4.14}. In this paper  we claim without proof that the Wulff shape should remain the same when $\gd \le \gd_c(\gb)$ while, in the case $\gd> \gd_c(\gb)$, it should become
	\beq
	\mathsf{W}_{\gb,\gd}(t) := \int_0^t \lmgf'\Big(s_\gd(\bar a_{\gb,\gd}^{-2})(1-x) + \gd - \gb/2\Big)\dd x, \qquad t\in[0,1],
	\label{eq:Wulff_beta_delta}
	\eeq
	where this time we used the tilt function from~\eqref{Nantes 44.14} with the special value $s = s_\gd(\bar a_{\gb,\gd}^{-2})$, see~\eqref{Nantes 3.10} and~\eqref{defbara} below. This is also natural in view of~\eqref{Nantes 9.5}. We are actually able to prove the following:
	\begin{proposition}
		Let $\deltaconcave(\gb)$ be the unique solution in $(\beta/2,\beta)$ of the equation $\lmgf(\gd - \beta/2) = -\log \Gamma_\beta$. Assume $(\gb,\gd)\in \cC_{\rm good}$. If $\gd_c(\gb)< \gd < \deltaconcave(\gb)$ then $\mathsf{W}_{\gb,\gd}$ is concave (convex globule phase). If $\deltaconcave(\gb)< \gd < \gb$ then $\mathsf{W}_{\gb,\gd}$ is convex (concave globule phase).
		\label{Proposition concave}
	\end{proposition}
	The proof can be found in Section~\ref{proof concave}. Note that $\deltaconcave(\gb) = \gb/2 + x_\gb$, where $x_{\gb}$ has been defined below~\eqref{eq:def-bar-gd}. Anticipating on Remark~\ref{Nantes Remark 4.13}, it turns out that $x_\gb = \tilde h(a_\gb^{-2})/2$ and, by virtue of~\eqref{closedexprdeltac}, we notice that
		\beq
		\deltaconcave(\gb) + \gd_c(\gb) = \gb.
		\eeq
	\begin{figure}[H]
		\centering
		\includegraphics[width=0.8\textwidth]{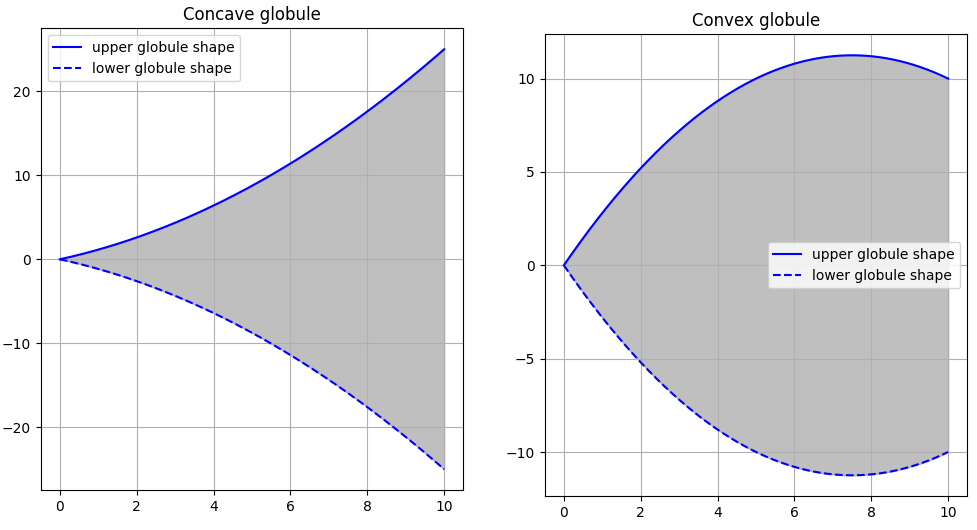}
		\caption{Schematic picture of a concave and a convex globule, on the left ($\gd>\deltaconcave(\gb)$) and right ($\gd < \deltaconcave(\gb)$) respectively.}
		\label{fig:comparison}
	\end{figure}
	\section{Proof of Proposition~\ref{existencefe}: volume free energy}\label{prphdiag}
	This section is devoted to the proof of Proposition~\ref{Nantes Proposition 1}, that is the existence of the (volume) free energy. Indeed, in statistical physics, the loss of analyticity in the free energy function indicates the presence of a phase transition. Extending the definition of the free energy in \eqref{feww} to the case $\delta>0$ is not immediate 
	since the sequence $(Z_{L,\beta,\delta})_{L\geq 1}$ is no longer  trivially sub-additive in $L$.  
	To that aim, we begin by {\it defining} the free energy as 
	\begin{definition}\label{deffe}
		\begin{equation}\label{fer}
			f(\beta,\delta):=\limsup_{L\to \infty} \frac{1}{L} \log Z_{L,\beta,\delta} \in \mathbb{R}.
		\end{equation}
	\end{definition}
	To prove Proposition \ref{Nantes Proposition 1}, we first need some classical results on a random walk representation, first stated in \cite{nguyen2013variational}.
	\subsection{Random walk representation}\label{sec:RWrepr}
	Let $X := (X_i)_{i \geq 0}$ be a random walk on $\Z$ starting at the origin, with i.i.d.\ increments distributed as in \eqref{Nantes 3.3}. We will need to consider the {\it geometric} area enclosed between the random walk trajectory and the horizontal axis
	up to time $N$, as well as its {\it arithmetic} counterpart: 
	\begin{equation}\label{defarea}
		G_N(X):=\sum_{i=0}^N |X_i|\quad \text{and}\quad A_N(X):=\sum_{i=0}^N X_i.
	\end{equation}
	Recall \eqref{defLL}--\eqref{pff}. For $A\subset \Omega_L$ we denote by  $Z_{L,\beta,\delta}(A)$ the partition function restricted to those trajectories $\ell\in A$, i.e., 
	\begin{equation}\label{defZA}
		Z_{L,\beta,\delta}(A)=\sum_{\ell\in A} e^{H_{L,\beta,\delta}(\ell)}.
	\end{equation}
	For $N\in\N$ we define the one-to-one correspondence:
	\begin{align}
		\label{deftnappli}
		\nonumber T_N:\{0\}\times \cL_{N,L}&\mapsto \{(X_i)_{i=0}^{N}\in \{0\}\times \Z^N \colon\,  G_N(X) = L-N\}\\
		(\ell_i)_{i=0}^N & \mapsto ((-1)^{i-1} \ell_i)_{i=0}^N.
	\end{align}
	Then, any subset $A\subset \Omega_L$ is said to be {\it stable by time inversion} if  for every $N\in \N$ we have that 
	$(\ell_i)_{i=1}^N\in  A\cap \cL_{N,L}$ implies $(\ell_{N+1-i})_{i=1}^N \in A\cap \cL_{N,L}$.
	\begin{lemma}\label{rwrep}
		Let $L\in \N$ and $A\subset \Omega_L$ be stable by time inversion. 
		Then,
		\begin{equation}
			Z_{L,\beta,\delta}(A) =e^{\beta L} \, \somme{N=1}{L}\,  \Gamma_\beta^N \, 
			\EsperanceMarcheAleatoire{e^{(\delta - \frac{\beta}{2})|X_N|}1_{\{ X\in T_N(A\cap \cL_{N,L})\} }}.
			\label{Nantesrwrrepaux}
		\end{equation}
	\end{lemma}
	\begin{proof}[Proof of Lemma~\ref{rwrep}] 
		To begin with, we use the stability of $A$ by time inversion to get
		\begin{equation}
			\begin{aligned}
				Z_{L,\beta,\delta}(A) &= \somme{N=1}{L}\,  \sum_{\ell\in A\cap \cL_{N,L}} 
				e^{|\ell_1| \delta}  \, \produit{i=1}{N-1} e^{\frac{ \beta |\ell_i| + \beta|\ell_{i+1}| - \beta |\ell_{i} + \ell_{i+1}|}{2}}\\
				&= \somme{N=1}{L}\,  \sum_{\ell\in A\cap \cL_{N,L}} 
				e^{|\ell_N| \delta}  \, \produit{i=1}{N-1} e^{\frac{ \beta |\ell_{i}| + \beta|\ell_{i+1}| - \beta |\ell_{i} + \ell_{i+1}|}{2}}.
			\end{aligned}
		\end{equation}
		For computational reasons, we add to every $\ell\in \cL_{N,L}$ a zero-length stretch at the beginning  of the configuration, that is, $\ell_0=0$.
		Thus, 
		\begin{equation}
			Z_{L,\beta,\delta}(A) =e^{\beta L} \somme{N=1}{L} \Gamma_\beta^N \sum_{\ell\in A\cap \cL_{N,L}}  e^{|\ell_N| (\delta - \frac{\beta}{2})} \, \produit{i=0}{N-1} \frac{e^{ - \frac{\beta |\ell_{i} + \ell_{i+1}|}{2}}}{c_\beta},
			\label{Nantes 1.9}
		\end{equation}
		where $\Gamma_\beta=c_\beta\, e^{-\beta}$ and where we have used that $\sum_{i=1}^N |\ell_{i}|=L-N$.
		We observe that  the product in the r.h.s. of~\eqref{Nantes 1.9} coincides with the probability that the random walk $X$ defined above 
		follows the trajectory $X_i=(-1)^{i-1} \ell_i$ for $i\in \{0,\ldots, N\}$. Thus, \eqref{Nantes 1.9}
		can be written as 
		\begin{equation}
			Z_{L,\beta,\delta}(A) =e^{\beta L} \somme{N=1}{L} \Gamma_\beta^N \sum_{\ell\in A\cap \cL_{N,L}}  e^{|\ell_N| (\delta - \frac{\beta}{2})} \, 
			\LoiMarchealeatoire{X_i=(-1)^{i-1} \ell_i,\ 0\le i\le N},
			\label{Nantesrwrrep}
		\end{equation}
		Using the one-to-one correspondence $T_N$ defined in \eqref{deftnappli}, we may conclude.
	\end{proof}
	\subsection{Proof}
	In
	the 
	core 
	of 
	the 
	proof 
	of 
	Proposition 
	\ref{Nantes Proposition 1} 
	we 
	will 
	show 
	that
	the $\limsup$ in \eqref{fer} equals the $\liminf$, hence the convergence of the full sequence.
	We first prove Proposition \ref{Nantes Proposition 1} subject to Claim \ref{claim 2} and then prove~Claim \ref{claim 2}.  We recall \eqref{Nantes 3.3},  \eqref{defarea} and 
	for $x \in \R$, we define $x^+ := \max\{0,x\}$. 
	\begin{claim}
		\label{lem:lemma-free-energy}
		For $u \leq \beta/2$ and $\gamma > 0$, there exists $C > 0$ such that, for every $N \geq 1$,
		\begin{equation}
			\EsperanceMarcheAleatoire{e^{u|X_N|} e^{-\gamma G_{N}(X)}} \leq 
			C \EsperanceMarcheAleatoire{e^{(u-\gamma)^+|X_{N-1}|} e^{-\gamma G_{N-1}(X)}}.
			\label{Nantes 5.3}
		\end{equation}
		with $X_0=G_0=0$.
		\label{claim 2}
	\end{claim}
	\begin{proof}[Proof of Proposition \ref{Nantes Proposition 1}] We proceed with lim sup and lim inf successively. By (i) restricting the partition function to the path taking only one vertical stretch of length $L-1$ and (ii) noticing that $Z_{L,\beta,\delta} \geq Z_{L,\beta,0}$ for $L \in \N$ and $\delta \geq 0$, we get that 
		\begin{equation}
			\underset{ L \longrightarrow \infty}{\liminf} \frac{1}{L} \log Z_{L,\beta,\delta} \geq f(\beta,0) \vee \delta.
			\label{Nantes 1.1}
		\end{equation}
		To complete the proof, it remains to consider lim sup instead of lim inf in \eqref{Nantes 1.1}. Therefore, we want to prove that:
		\begin{equation}
			f(\beta,\delta) \leq f(\beta,0) \vee \delta.
			\label{Nantes 1.2}
		\end{equation}
		We first decompose the partition function according to the length of the first vertical stretch (either up or down) and add a reward $\beta$ along the whole second stretch, which gives us the following upper bound :
		\begin{equation}
			Z_{L,\beta,\delta} \leq 2 \somme{k=0}{L-1}e^{k \delta} Z_{L-k-1,\beta,\beta}.
		\end{equation}
		Hence, after taking the logarithm and dividing by the polymer length $L$ we obtain
		\begin{equation}
			\frac{1}{L} \log Z_{L,\beta,\delta} \leq \underset{0 \leq k < L}{\max} \Big\{ \tfrac{k}{L} \delta + (1-\tfrac{k+1}{L})\,  \tfrac{1}{L-k-1}\,  \log(Z_{L-k-1,\beta,\beta}) \Big\} + o_L(1),
		\end{equation}
		from which we deduce, after letting $L\to \infty$ that $f(\beta,\delta) \leq f(\beta,\beta) \vee \delta$. Thus, the proof will be complete once we show that  $f(\beta,0) = f(\beta,\beta)$. To that aim, we start by applying the Cauchy-Hadamard Theorem, which guarantees us that for $(\beta,\delta) \in \mathcal{Q}$, 
		\begin{equation}
			f(\beta,\delta) = \inf \Big\{ \gamma \geq 0 : \somme{L \geq 1}{}Z_{L,\beta,\delta} e^{-\gamma L}  < \infty\Big\}.
			\label{Nantes 1.4}
		\end{equation}
		Using Lemma \ref{rwrep} with $A=\Omega_L$ and the fact that $\{X\in T_N(\cL_{N,L})\}=\{G_N(X)=L-N\}$, we obtain for $\gamma\ge0$,
		\begin{equation}
			\begin{aligned}
				\somme{L \geq 1}{} Z_{L,\beta,\delta}\, e^{-\gamma L} &=
				\somme{N\ge1}{} \somme{L\ge N}{} e^{-(\gamma-\beta) N} \, \Gamma_\beta^N  \, \EsperanceMarcheAleatoire{e^{(\delta - \beta/2)|X_N|} 1_{\{ G_N(X) = L-N \} } e^{-(\gamma-\beta) (L-N)} } \\
				&= \somme{N\ge1}{} \left( \Gamma_\beta e^{-(\gamma-\beta)} \right)^N \EsperanceMarcheAleatoire{e^{(\delta - \beta/2)|X_N|} e^{-(\gamma-\beta) G_{N}(X)}}.    
			\end{aligned}
			\label{Nantes 1.11}
		\end{equation}
		We now pick $\delta = \beta$ and $\gamma > f(\beta,0)$. If we manage to prove that~$ \sum_{L \geq 1}Z_{L,\beta,\beta} e^{-\gamma L} < \infty$ then $f(\gb,\gb)\le f(\gb,0)$ by~\eqref{Nantes 1.4}, which would complete the proof (the reverse inequality clearly holds true). Using \eqref{Nantes 1.11}, it comes that:
		\begin{equation}\label{grandcanga}
			\somme{L \geq 1}{}Z_{L,\beta,\beta} e^{-\gamma L} =  
			\somme{N\ge 1}{} \left( \Gamma_\beta e^{-(\gamma - \beta)} \right)^N \EsperanceMarcheAleatoire{e^{(\beta/2)|X_N|} e^{-(\gamma-\beta) G_{N}(X)}}.    
		\end{equation}
		We recall from~\eqref{deftwil} that $f(\beta,0) \geq \beta$ and therefore $\gamma -\beta>0$. Thus, we can denote by  $k$ the smallest positive integer satisfying $\beta/2-k(\gamma-\beta)\leq 0$. It remains to successively use Claim~\ref{lem:lemma-free-energy} $k$ times to assert that there exists $C>0$, such that for $N\geq k$
		\begin{equation}
			\EsperanceMarcheAleatoire{e^{(\beta/2)|X_N|} e^{-(\gamma-\beta) G_{N}(X)}}
			\leq C\ \EsperanceMarcheAleatoire{ e^{-(\gamma-\beta) G_{N-k}(X)}}.
		\end{equation}
		As a consequence, there exists $C_1>0$ such that \eqref{grandcanga} becomes 
		\begin{equation}\label{upbounvar}
			\begin{aligned}
				\somme{L \geq 1}{}Z_{L,\beta,\beta} e^{-\gamma L} &\leq
				C_1+ C \somme{N\ge k}{} \left( \Gamma_\beta e^{-(\gamma - \beta)} \right)^N \EsperanceMarcheAleatoire{e^{-(\gamma-\beta) G_{N-k}(X)}} \\
				&=
				C_1+ C (\Gamma_\beta e^{\beta - \gamma})^k \somme{N\ge 0}{} \left( \Gamma_\beta e^{-(\gamma - \beta)} \right)^N \EsperanceMarcheAleatoire{  e^{-(\gamma-\beta) G_{N}(X)}}.
			\end{aligned}
		\end{equation}
		At this stage, we recall the following exponential growth rate from~\cite[Lemma 2.1]{CNP16}: 
		\begin{equation}\label{defhbet}
			\mathfrak{h}_\beta(u):=\lim_{N\to \infty} \frac{1}{N}\log \EsperanceMarcheAleatoire{e^{-u G_{N}(X)}}\le 0, \quad u\in [0,\infty).
		\end{equation}
		We now distinguish between two cases. If $\Gamma_\beta \le 1$, then clearly
		$\log \Gamma_\beta-(\gamma-\beta)+\mathfrak{h}_\beta(\gamma-\beta)\le -(\gamma-\beta)$, which is negative, since $\gamma > f(\beta,0)=\beta$. Assume now that $\Gamma_\beta > 1$. Then, it was proven in \cite[Theorem A]{CNP16} that 
		\begin{equation}
			f(\beta,0) = \sup\{u\ge \beta\colon \log \Gamma_\beta-(u-\beta)+\mathfrak{h}_\beta(u-\beta)>0\}>\beta,
		\end{equation}
		and that $f(\beta,0)$ is the only solution in $u$ of $\log \Gamma_\beta-(u-\beta)+\mathfrak{h}_\beta(u-\beta)=0$ (we pay attention to the fact that the excess free energy is $f(\beta,0)-\beta$). Since $\gamma>f(\beta,0)$ we necessarily have that 
		$\log \Gamma_\beta-(\gamma-\beta)+\mathfrak{h}_\beta(\gamma-\beta)<0$
		which implies that  the r.h.s. in~\eqref{upbounvar} is finite. This completes the proof.
	\end{proof}
	\begin{proof}[Proof of Claim \ref{claim 2}] Let $u \in (0,\beta/2]$ and $\gamma > 0$. Since $G_N(X) - G_{N-1}(X) = |X_{N}|$,
		\begin{equation}
			\EsperanceMarcheAleatoire{e^{u|X_N|} e^{-\gamma G_{N}(X)}} \leq 
			\EsperanceMarcheAleatoire{e^{(u - \gamma)^+|X_N|} e^{-\gamma G_{N-1}(X)}}.
		\end{equation}
		If $u-\gamma \le 0$, the claim readily follows. Otherwise, we use the triangular inequality, independence of the increments, and the fact that $u-\gamma< \beta/2$ to obtain as an upper bound:
		\begin{equation}
			\begin{aligned}
				&\EsperanceMarcheAleatoire{e^{(u - \gamma)|X_N - X_{N-1} + X_{N-1}|} e^{-\gamma G_{N-1}(X)}}  \\&\leq  
				\EsperanceMarcheAleatoire{e^{(u-\gamma) |X_1|}}
				\EsperanceMarcheAleatoire{e^{(u-\gamma)|X_{N-1}|} e^{-\gamma G_{N-1}(X)}} \\
				&\leq C \EsperanceMarcheAleatoire{e^{(u-\gamma)|X_{N-1}|} e^{-\gamma G_{N-1}(X)}}.
			\end{aligned}
		\end{equation}
		This completes the proof.
	\end{proof}
	\section{Preparation}\label{prepsec}
	\subsection{Auxiliary partition functions}\label{auxpartfun}
	In what follows, for $q\geq 0$ and $\delta\in [0,\beta)$, we set
	\begin{align}\label{eq:D_Nqdelta}
		D_N(q,\delta):&=\EsperanceMarcheAleatoiresansparenthese \big[e^{(\delta - \frac{\beta}{2}) X_N}\, 1_ {\{\cV_{N,qN^2,+} \} }\big], \qquad N\in\N,
	\end{align}
	where
	\begin{align}\label{defnun}
		\cV_{N,k,+} &= \{A_N = k,\ X_i > 0,\ 0<i\leq N\},
	\end{align} 
	and $A_N$ has been defined in \eqref{defarea}. It turns out that $\psi(q,\delta)$ defined in \eqref{deftildepsi} is the exponential growth rate of the sequence $(D_N(q,\delta))_{N\in \N}$. This result will be established as a byproduct of the proof of Proposition \ref{asympteach3reg}.
	\subsection{Bead decomposition}\label{beaddec}
	The aim of this section is to show how one can decompose a given trajectory in $\Omega_L$ into sub-trajectories 
	that do not interact with each other, referred to as {\it beads}. 
	
	Note that we will extract some estimates from \cite{Legrand_2022} where a similar decomposition has been introduced to study the same model when the wall-polymer interaction is shut-down ($\delta=0$).
	In the latter case, it was proven in \cite[Theorem 2.2]{Legrand_2022} that, inside the 
	collapsed phase,  a typical
	trajectory is made of a unique macroscopic bead from which only finitely many monomers may escape.  
	In the present paper, although the uniqueness of the macroscopic bead still holds true inside the collapsed phase,
	the presence of an attractive wall both changes drastically the asymptotics of the partition function, but triggers also a much richer phenomenology including a surface transition and a radical change of the shape of the macroscopic bead. 
	
	The main difference between the model at $\delta>0$ and the model at $\delta=0$  is 
	that in the former, the very first bead of a trajectory and the following beads  need to be considered separately. 
	The polymer-wall interaction indeed entails that, in large parts of the collapsed phase, the very first bead is the unique macroscopic bead. Therefore, deriving results on the polymer in this phase requires a deep understanding of most features of this first bead.
	\begin{figure}[H]
		\centering
		\includegraphics[width=0.5\textwidth]{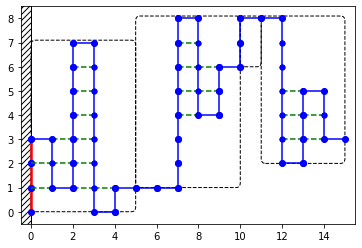}
		\caption{Bead decomposition of a trajectory. The blue lines stand for the polymer configuration, the orange dashed lines stand for the attractive self-interaction, and the red lines stand for the edges pinned at the attractive wall. In this example, we can see four beads each delimited by black dashed lines. The sign of the initial stretch of the third bead is determined by the last stretch of the previous bead. Since the second and fourth beads start with horizontal stretches, the sign of their first non-zero stretch may be either positive or negative.}
		\label{figure : des bulles}
	\end{figure}
	\subsubsection{Decomposition of a trajectory into beads}
	Let us start with a handful of notation. Set $\Omega^{\, \text{o}}:=\emptyset \cup(\cup_{L\geq 1}\, 
	\Omega_L^{\, \text{o}})$ with 
	$\Omega_L^{\, \text{o}}:=\cup_{N=1}^{L/2}\,  \cL_{N,L}^{\, \text{o}}$ where
	\begin{equation}\label{defLnl}
		\cL_{N,L}^{\, \text{o}}:=\Big\{(\ell_i)_{i=1}^N\in \Z^N\colon\, \sum_{i=1}^N|\ell_i|=L-N,\  \ell_i \ell_{i+1}<0,  \ \forall\,  1\leq i<N\Big\},
	\end{equation}
	so that $\Omega_L^{\, \text{o}}$ gathers those trajectories forming a unique bead of length $L$. The associated {\it bead partition function} is defined by
	\begin{equation}\label{partfunbead}
		Z_{L,\beta,\delta}^{\, \text{o}}:=\sum_{\ell\in \Omega_L^{o}} e^{H_{L,\beta,\delta}(\ell)}=\sum_{N=1}^{L/2} Z_{L,\beta,\delta}(\cL_{N,L}^{\, \text{o}}).
	\end{equation}
	Recall the definitions of $\Gamma_\gb$, $D_N(q,\delta)$ and $\cV_{N,k,+}$ in \eqref{eq:defGammaBeta}, \eqref{eq:D_Nqdelta} and \eqref{defnun}. By using Lemma~\ref{rwrep} with $A=\cL_{N,L}^{\, \text{o}}$ and by noticing that $T_N$ (defined in \eqref{deftnappli}) is a 
	one-to-one correspondance between $\cL_{N,L}^{\,\text{o}}$ and $\cV_{N,L-N,+}\cup \cV_{N,L-N,-}$
	we obtain that
	\begin{equation}\label{partfunbeadaux}
		\tilde Z_{L,\beta,\delta}^{\, \text{o}}:=e^{-\beta L} Z_{L,\beta,\delta}^{\, \text{o}}= 2 \sum_{N=1}^{L/2} \Gamma_\beta^N D_{N}\big(\tfrac{L-N}{N^2},\delta\big).
	\end{equation}
	We now allow beads to start with stretches of zero length. To that aim, we define
	$\hat \Omega^{\, \text{o}}:=\cup_{L\geq 1} \, \hat \Omega_L^{\, \text{o}}$ where $ \hat \Omega_L^{\, \text{o}}:=\cup_{k=0}^{L-2}\,   \hat \Omega^{\text{o},k}_L$ with 
	\begin{align}\label{defbeadwithstraightline}
		\hat \Omega_L^{\, \text{o},0}&:=\{\ell \in \Omega_L^{\, \text{o}}\colon\, \ell_1>0\},\\ 
		\nonumber \hat \Omega_{L}^{\, \text{o}, k}&:=\{\ell\in \Omega_L\colon\, N_\ell\geq k,\  \ell_1=\dots=\ell_k=0,\ (\ell_{i+k})_{i=1}^{N_\ell-k}\in \Omega_{L-k}^{\, \text{o}}\}, \ k\in \N.
	\end{align}
	Those trajectories shall be called {\it extended beads}.
	Note that the condition $\ell_1>0$ in the first line of \eqref{defbeadwithstraightline} is imposed by the fact that, when there is no zero-length stretch between two beads, the sign
	of the first vertical stretch of the second bead must correspond to that of the last stretch of the first bead. 
	For the sake of simplicity, we define $\Omega^c_L$ as the subset of $\Omega_L$ that contains trajectories ending with a non-zero stretch, i.e.,
	\begin{equation}
		\Omega^c_L = \left\{ \ell \in \Omega_L : \ell_{N_\ell} \neq 0 \right\}.
	\end{equation}
	With those subsets of trajectories in hand, we may divide a given trajectory as follows: 
	$\tau_0:=0$ and for $j\in  \N$ such that $\tau_{j-1}<N_\ell$,
	\begin{align}
		\tau_j:=\tau_{j-1}+\max\{s>0 \colon\, (\ell_{\tau_{j-1}+i})_{i=1}^{s}\in \hat \Omega^{\, \text{o}}\quad \text{or}\quad (-\ell_{\tau_{j-1}+i})_{i=1}^{s}\in \hat \Omega^{\, \text{o}} \}.
	\end{align}
	Finally, we let $n(\ell)$ be the number of (extended) beads into which a given trajectory $\ell \in \Omega_L$ may be divided. Thus, $\tau_{n(\ell)}=N_\ell$ and 
	$\ell$ may be seen as the concatenation of $n(\ell)$ beads denoted by 
	$\cB_j:=(\ell_{\tau_{j-1}+1},\dots,\ell_{\tau_j})$, $j\in \{1,\dots,n_\ell\}$. The number of monomers in the $j$-th bead is denoted by $|\cB_j|$ 
	and has value $|\cB_j|=\tau_j-\tau_{j-1}+\sum_{i=\tau_{j-1}+1}^{\tau_j} |\ell_i|$. 
	\subsubsection{Bead decomposition of the partition function}
	We recall \eqref{partfunbead} and we 
	note that only the first bead of the trajectory interacts with the vertical wall, provided that the trajectory does not 
	begin with a horizontal stretch. This leads us to define 
	\begin{equation}\label{fbead}
		\bar Z_{L,\beta,\delta}^{\, \text{o}} := Z_{L,\beta,\delta}^{\, \text{o}}+ \sum_{k=1}^L 
		Z_{L-k,\beta,0}^{\, \text{o}} = e^{\beta L} \tilde Z_{L,\beta,\delta}^{\, \text{o}}+  e^{\beta L}\sum_{k=1}^L e^{-\beta k}
		\tilde Z_{L-k,\beta,0}^{\, \text{o}},
	\end{equation}
	where $k$ stands for the number of initial stretches with zero length, as the contribution of the {\it first} (extended) bead to the partition function.
	The contribution of the following beads to the partition function does not involve $\delta$ since they cannot touch the vertical wall, leading us to define
	\begin{equation}\label{folbead}
		\hat Z_{L,\beta}^{\, \text{o}}:= \frac{1}{2} e^{\beta L} \tilde Z_{L,\beta,0}^{\, \text{o}} + e^{\beta L}\sum_{k=1}^L e^{-\beta k}
		\tilde Z_{L-k,\beta,0}^{\, \text{o}}.
	\end{equation}
	Finally, we can decompose the full partition function $Z_{L,\beta,\delta}^{c}$, that is the partition function restrained to $\Omega^c_L$
	according to the number of beads 
	and the length of those beads, namely 
	\begin{align}\label{fullpartfun}
		Z_{L,\beta,\delta}^{c}&=\sum_{k=1}^{L/2}\,  \sumtwo{t_1+\dots+t_k=L}{t_1, \ldots, t_k >1} \,  Z_{L,\beta,\delta}(n_\ell=k, |\cB_1|=t_1,\dots,|\cB_k|=t_k)\\
		\nonumber &=\sum_{k=1}^{L/2}\,  \sumtwo{t_1+\dots+t_k=L}{t_1, \ldots, t_k >1}  \bar Z^{\, \text{o}}_{t_1,\beta,\delta}\,  \prod_{i=2}^k \hat Z^{\, \text{o}}_{t_i,\beta}.
	\end{align}
	Because of \eqref{fullpartfun} above, proving Theorem \ref{asymptotics} requires to derive both the asymptotics of the partition function sequence of the first bead, i.e., $(\bar Z_{L,\beta,\delta}^{\, \text{o}})_{L\in \N}$ 
	and the asymptotics of the partition function 
	of the following beads, namely $(\hat Z_{L,\beta}^{\, \text{o}})_{L\in \N}$. The former is one of the main issue that we tackle in the present paper whereas the latter  has been established in details in \cite{Legrand_2022}
	and we recall it below.
	\begin{proposition}[Corollary 4.2 in \cite{Legrand_2022}]\label{essympwithoutdelta}
		For $\beta>\beta_c$, there exists $K_\beta^{\text{o}}, \hat K_\beta^{\text{o}}>0$ such that
		\begin{align}\label{asymptunibeadnowall}
			Z_{L,\beta,0}^{\, \text{o}}&\simL \frac{K_\beta^{\text{o}}}{L^{3/4}}\, e^{\beta L+g(\beta,0) \sqrt{L}}\\
			\nonumber \hat Z_{L,\beta}^{\, \text{o}}&\simL \frac{\hat K_\beta^{\text{o}}}{L^{3/4}}\, e^{\beta L+g(\beta,0) \sqrt{L}}
		\end{align}		
	\end{proposition}
	\begin{remark}
		Although in \cite[Corollary 4.2]{Legrand_2022} the prefactor in front of the $\sqrt{L}$ term in the exponential 
		is expressed as $\tilde\cG(a_\beta)$ and looks different from $g(\beta,0)$ that is used above, these two quantities are in fact equal
		to $\max\{T_0(a), a>0\}$ (see the definition of $T_\gd$ in Section~\ref{sec:analysis_aux} below).  The expressions of $K_\beta^{\text{o}}$ and $\hat K_\beta^{\text{o}}$ are available in~\cite[Equation (4.36)]{Legrand_2022}.
	\end{remark}
	\begin{remark}
		As a non-trivial by-product of the proof of Theorem~\ref{asymptotics}, we will prove in Lemma~\ref{Nantes Lemma 6.12} that the beads which start with a non-zero vertical stretch bear all the mass coming from the extended beads in the partition function when $\gd \ge \gd_c(\gb)$, that is not only in the adsorbed-collapsed phase \emph{but also at criticality}.
	\end{remark}%
	\subsection{Change of measure}\label{sec:COM}
	In this section we introduce several changes of measure for the position and area of the random walk that will be instrumental in deriving the asymptotics of the partition function.
	\subsubsection{\bf Uniform tilting}
	We remind the reader that $\lmgf$ is the logarithmic moment generating function of  $X_1$ a random variable of law $\Pb_\gb$, defined in~\eqref{def:cL}. That is a smooth, even and strictly convex function on $(-\gb/2, \gb/2)$ with a second derivative bounded from below by a positive constant. Let us first define a tilted transformation of $\LoiMarchealeatoireSansParenthese$. For $|h| < \beta/2$, we let $\ProbaTilteeSansParenthese{h}$ be the probability law defined on $\mathbb{Z}$ by perturbing $\LoiMarchealeatoireSansParenthese$ as follows:
	\begin{equation}
		\frac{\text{d} \ProbaTilteeSansParenthese{h} } {\text{d} \LoiMarchealeatoireSansParenthese}(\,\cdot\,=k) = e^{hk - \lmgf(h)} \quad k \in \Z.
		\label{Nantes 4.3}
	\end{equation}
	In the paper, we will consider the probability that a random walk $X:=(X_i)_{i\in \N}$ starting from $x\in \N$ and with i.i.d.\ increments of law $\ProbaTilteeSansParenthese{h}$
	remains positive (or equivalently, that the random walk starting at the origin remains above level $-x$).
	To that aim, we state Lemma \ref{lemrestpos} below that will be proven in Section \ref{prooflemsrestpos}.
	\begin{lemma}\label{lemrestpos}
		Let $\beta>0$. For every $x\in \bbN$ and $h\in (0,\beta/2)$, we have
		\begin{equation}
			\kappa^x(h) := \ProbaTiltee{h}{X_i > -x, \, \forall i \in \bbN} = 1-e^{-2hx} \frac{1-e^{h-\beta/2}}{1-e^{-h-\beta/2}}, 
			\label{defkappa}
		\end{equation}
		that is continuous in $h$. Moreover, for every $c>0$ and $[h_1, h_2]\subseteq(0,\beta/2)$,
		\begin{equation}
			\suptwo{0\le x\le c \log k}{h\in [h_1,h_2]} \lim_{k\to\infty} |\widetilde{\mathbf{P}}_{h}\left(X_{[1, k]}>-x\right)-\kappa^x(h)|=0.
			\label{eq:kappa2}
		\end{equation}
	\end{lemma}
	\subsubsection{\bf Tilting of the area enclosed by a random walk}\label{sec:tilting}
	We denote by $A_n(X)$ the algebraic area enclosed by $X$ up to time $n$, i.e.:
	\begin{equation}
		A_n(X) = X_1 + ... + X_n
	\end{equation}
	and by $\Lambda_n$ the random vector recording the latter area renormalized by $n$ and the final position $X_n$ of the walk,
	that is,
	\begin{equation}\label{deflam}
		\Lambda_n:=\left(\frac{A_n}{n},X_n\right).
	\end{equation}
	Throughout the paper, we will need to estimate the probability of the event $\frac{1}{n} \Lambda_n=(q,p)\in \mathbb{R}^2$  
	and more importantly to work with the random walk conditioned on such events.
	To that aim, we will use an inhomogeneous exponential perturbation of the law of each increment of $X$, as it was first displayed in~\cite{DH96}. Thus, we define
	\begin{equation}
		\frac{\mathrm{d} \mathbf{P}_{n, \boldsymbol{h}}}{\mathrm{d} \mathbf{P}_{\beta}}(X)=e^{\boldsymbol{h} \cdot \Lambda_{n}-\mathcal{L}_{\Lambda_{n}}(\boldsymbol{h})} \quad \text { with } \quad \mathcal{L}_{\Lambda_{n}}(\boldsymbol{h}):=\log \mathbf{E}_{\beta}\left[e^{\boldsymbol{h} \cdot \Lambda_{n}}\right]
		\label{Nantes 4.16}
	\end{equation}
	where
	\begin{equation}
		\boldsymbol{h} \in \mathcal{D}_{\beta, n}:=\left\{\left(h_{0}, h_{1}\right) \in \mathbb{R}^{2} \colon \Big|\tfrac{h_{0}}{n}+h_{1}\Big|<\beta / 2,\ \left|h_0+h_{1}\right|<\beta / 2\right\}.
	\end{equation}
	Noticing that
	\begin{equation}
		\label{eq:calculLambda_n}
		\boldsymbol{h} \cdot \Lambda_{n} = \sum_{k=1}^n (X_k - X_{k-1}) \Big(h_0[1-\tfrac{k-1}{n}] + h_1\Big),
	\end{equation}
	we see that this change of measure corresponds to an {\it inhomogeneous} tilt on the increments of the random walk.
	We also set a continuous counterpart, namely
	\begin{equation}
		\label{eq:def-Dbeta}
		\mathcal{D}_{\beta}:=\left\{\left(h_{0}, h_{1}\right) \in \mathbb{R}^{2} \colon \left|h_{1}\right|<\beta / 2,\ \left|h_{0}+h_{1}\right|<\beta / 2\right\},
	\end{equation}
	and we observe that, by \eqref{eq:calculLambda_n}, the sequence $(\frac{1}{n} \mathcal{L}_{\Lambda_{n}}(\boldsymbol{h}))_{n\in \N}$ converges for any $\boldsymbol{h}\in \mathcal{D}_{\beta}$ (note that $\mathcal{D}_{\beta}\subseteq\mathcal{D}_{\beta, n}$ for every $n\in\N$) towards: 
	\begin{equation}
		\mathcal{L}_{\Lambda}(\boldsymbol{h}):=\int_{0}^{1} \lmgf\left(h_{0} x+h_{1}\right) \mathrm{d} x .
		\label{Nantes 7.20}
	\end{equation}
	The two items of the following proposition come from~\cite[Lemma 5.4]{CNP16} and~\cite[Lemma 5.3]{CNP16} respectively. We take this occasion to correct a mistake in the original proof of ~\cite[Lemma 5.3]{CNP16}, see Appendix~\ref{sec:corr-proof-diffeo}.
	\begin{proposition}\label{c1diffeo}
		Let $\beta>0$.
		\begin{enumerate}
			\item For every $n\in \N$, the gradient $\nabla\left[\frac{1}{n} \mathcal{L}_{\Lambda_{n}}\right]$ is a $\mathcal{C}^{1}$-diffeomorphism from $\mathcal{D}_{\beta, n}$ to $\mathbb{R}^{2}$. For this reason,  for $(q, p) \in \mathbb{R}^2$ there exists a unique 
			\begin{equation}
				\boldsymbol{h}:=\boldsymbol{h}_{n}(q, p)=(h_{n,0}(q,p),h_{n,1}(q,p))
			\end{equation}
			which solves:
			\begin{equation}
				\mathbf{E}_{n, \boldsymbol{h}}\Big[\frac{1}{n} \Lambda_{n}\Big]=\nabla\Big[\frac{1}{n} \mathcal{L}_{\Lambda_{n}}\Big](\boldsymbol{h})=(q, p).
			\end{equation}
			\item $\nabla \mathcal{L}_{\Lambda}$ is a $\mathcal{C}^{1}$-diffeomorphism from $\mathcal{D}_{\beta}$ to $\mathbb{R}^{2}$. Thus, for $(q, p) \in \mathbb{R}^{2}$ we let $\widetilde{\boldsymbol{h}}(q, p)$ be the unique solution in $\boldsymbol{h} \in \mathcal{D}_{\beta}$ of the equation $\nabla \mathcal{L}_{\Lambda}(\boldsymbol{h})=(q, p)$. 
		\end{enumerate}
	\end{proposition}
	Along the present paper we will need to consider two particular cases of the tilting procedure set up in \eqref{Nantes 4.16}, namely (i) a tilting for which the second coordinate in 
	$\boldsymbol{h}$ is prescribed and (ii) a tilting for which $p=0$.
	%
	%
	%
	%
	%
	%
	\subsubsection{\bf Case (i): The tilt on the final position is prescribed}
	This is the case where the value of $h_1$ is set to be $\delta-\beta/2$ (as suggested by Lemma~\ref{rwrep}). Let $\delta \in (0,\beta)$, $n\in \N$ and set $$\cA_{n,\delta}:=\Big(-\frac{2n}{2n-1}\delta, \frac{2n}{2n-1}(\beta-\delta)\Big),$$ 
	so that for $s\in \cA_{n,\delta}$ and after recalling \eqref{Nantes 4.16} we may consider 
	the perturbed probability measure $ \mathbf{P}_{n, (s, \delta-\frac{\beta}{2}-\frac{s}{2n})}$ . Note that the correction $\frac{s}{2n}$ in the second parameter is introduced as a technical artifact to obtain Proposition  \ref{approxhn} below. To be more specific, we come back to \eqref{eq:calculLambda_n} and write
	\begin{equation}\label{defsupm}
		\frac{\mathrm{d} \mathbf{P}_{n, (s, \delta-\frac{\beta}{2}-\frac{s}{2n})} }{\mathrm{d} \mathbf{P}_{\beta}}(X)=e^{(s, \delta-\frac{\beta}{2}-\frac{s}{2n})\cdot \Lambda_{n}-n\,  \cH_{n,\delta}(s)}
	\end{equation}
	where
	\begin{equation}\label{defHH}
		\cH_{n,\delta}(s)=\frac{1}{n}\,  \lmgf_{\Lambda_n}\Big(s,\delta-\frac{\beta}{2}{- \frac{s}{2n}}\Big), \quad s\in \cA_{n,\delta}.
	\end{equation}
	In what follows, we will need to tune $s\in 	 \cA_{n,\delta}$ in such a way that the expectation of $A_n$ equals $q n^2$ for some $q>0$. This is the object of Lemma \ref{lem:prop-Hgd-Hngd} below, which guarantees the existence and uniqueness of such a parameter $s$. We extend this result to the continuous counterpart of $\cH_{n,\delta}(s)$ that is, in view of~\eqref{eq:calculLambda_n}, 
	\begin{equation}\label{defHHcont}
		\cH_{\delta}(s):=\int_0^{1} \lmgf\Big(sx+\delta-\frac{\beta}{2}\Big)\,  \dd x , \quad s\in \cA_{\delta}:= (-\gd, \gb-\gd),
	\end{equation}
	where we observe that  $\cA_{\delta}\subseteq \cA_{n,\delta}$ for every $n\ge 1$.

	\begin{lemma} 
		\label{lem:prop-Hgd-Hngd}
		Let $\delta\in (0,\beta)$.
		\begin{enumerate}
			\item For every $n \geq 2$, the mapping $\cH_{n,\delta}$ is $\sC^{\infty}$ and strictly convex on $\cA_{n,\delta}$. Moreover, $\cH_{n,\delta}'$ is a $\sC^1$-diffeomorphism from $\cA_{n,\delta}$ to $\R$.
			\item The mapping $\cH_{\delta}$ is $\sC^{\infty}$ and strictly convex on $\cA_{\delta}$. Moreover, the function $\cH_{\delta}'$ is a $\sC^1$-diffeomorphism from $\cA_{\delta}$ to $\R$.
			\item The function $\cH_\gd$ is bounded on $\cA_{\delta}$.
		\end{enumerate}
		\label{Nantes Lemma 9.5}
	\end{lemma}
	\begin{proof}[Proof of Lemma \ref{Nantes Lemma 9.5}] The first item of Lemma \ref{Nantes Lemma 9.5} being a discrete counterpart of the second item, we will only prove the foremost here. Recalling that $\lmgf$ is strictly convex and $\sC^{\infty}$ on $(-\beta/2,\beta/2)$, it comes that $\cH_{N,\delta}$ is also strictly convex and $\sC^{\infty}$ as the finite sum of strictly convex and $\sC^{\infty}$ functions. Moreover, $\cH_{N,\delta}'$ is increasing as the finite sum of increasing functions. Finally, since $\lmgf''$ is bounded from below by a positive constant, we obtain that $\lmgf'(h)$ goes to $\pm\infty$ as $h\to \pm \beta/2$, respectively. Necessarily, $\cH_{N,\delta}'(\cA_{N,\gd}) = \R$, which completes the proof.
		Let us now turn to the third item. It is sufficient to look at the function outside a neighborhood of the origin. A straightforward change of variable yields, provided $s\neq 0$:
		\beq
		\cH_\gd(s) = \int_0^1 \lmgf(sx+\gd-\gb/2)\dd x \le \frac{1}{|s|}\int_{-\gb/2}^{\gb/2} \lmgf(x)\dd x,
		\eeq
		which proves our claim, since the last integral is finite.
	\end{proof}
	As a consequence of Lemma \ref{lem:prop-Hgd-Hngd}, we may define $\sqdeltan$  and  $\sqdelta$ for every $q>0$ as the respective solutions of 
	\begin{equation}
		\label{def sqdelta}
		\cH_{n,\delta}'(s)=q \quad \text{and} \quad  \cH_{\delta}'(s)=q.
	\end{equation}
	For $N\in \N$ and $q>0$ we will need in several instances to consider  $\ProbaSurCritiqueSansParenthese$ the probability law 
	on random walk trajectories displayed in \eqref{defsupm} with $n=N$ and  with parameter 
	$s=\sqdeltaN$,~i.e., 
	\begin{equation}\label{defPNq}
		\ProbaSurCritiqueSansParenthese=\mathbf{P}_{N, \big (\sqdeltaN, \delta-\frac{\beta}{2}-\frac{\sqdeltaN}{2N}\big)}\, .
	\end{equation}
	\begin{remark}\label{Nantes 9.5}
		Under $\ProbaSurCritiqueSansParenthese$ the increments of $X$, namely  $X_k-X_{k-1}$ for $k\in \{1,\dots,N\}$  are independent and follow respectively the tilted law defined in \eqref{Nantes 4.3} with tilt parameter $\delta - \frac{\beta}{2} + \sqdeltaN \frac{2N+1-2k}{2N}$.
	\end{remark} 
	The following proposition quantifies the convergence speed of $\cH_{N,\delta}$ and $\sqdeltaN$ towards $\cH_{\delta}$ and $\sqdelta$, respectively.
	\begin{proposition}\label{approxhn}
		For every $[q_1,q_2] \subseteq (0,\infty)$ and $[s_1,s_2]\subseteq \cA_{\delta}$, there exists $C>0$ and $N_0 \in \N$ such that, for every $N \geq N_0$ and 
		$s\in [s_1,s_2]$:
		\begin{equation}
			|\cH_{N,\delta}(s) - \cH_{\delta}(s)| \leq \frac{C}{N^2}
			\label{Nantes 9.11}
		\end{equation}
		and for every $q \in [q_1,q_2] \cap \frac{\N  }{N^2}$:
		\begin{equation}
			|\sqdeltaN - \sqdelta| \leq \frac{C}{N^2}.
			\label{Nantes 9.12}
		\end{equation}
		\label{Nantes Proposition 9.3}
	\end{proposition}
	The proof of Proposition~\ref{approxhn} can be found in Appendix \ref{Nantes Appendix B.11}. As we will see, the correction in $1/N^2$ is crucial in order to obtain the sharp asymptotics in Theorem \ref{asymptotics}, as well as in Lemma \ref{Nantes Lemma 4.7} below, the proof of which is deferred to Appendix \ref{Proof 4.7}.
	\begin{lemma} \label{Nantes Lemma 4.7}
		For every $B\in \sigma(X_0,\dots,X_N)$, for every $0<q_1<q_2<\infty$ and uniformly in $q \in [q_1,q_2]$, as $N\to \infty$,
		\begin{equation}\label{calculdeproba1}
			\begin{aligned}
				&\mathbf{E}_{\beta}\Big[e^{(\delta-\frac{\beta}{2})X_N}    1_{\{A_N=q N^2,\, X\in B\}}\Big] = \\
				&\tilde c e^{N \big[\cH_{\delta}(\sqdelta)- \sqdelta q\big]}  \left( \ProbaSurCritique{
					A_N=q N^2, X\in B} (1+o(1)) + O(e^{-\log(N)^2}) \right),
			\end{aligned}
		\end{equation} 
		with $\tilde c=\exp\Big ( \frac{1}{2} \Big[ \cL(\delta - \beta/2 +  \sqdelta) - \cL(\delta - \beta/2) \Big] \Big) $.
	\end{lemma}
	\subsubsection{\bf Case (ii): The final position is set to zero.}
	
	This is the case where the value of the final position $p$ in Proposition \ref{c1diffeo} (1) equals $0$. For every $q>0$ and $n\in \N$ there exists a unique solution of $\nabla\left[\frac{1}{n} \mathcal{L}_{\Lambda_{n}}\right](\boldsymbol{h})=(q,0)$ denoted by $\boldsymbol{h}_{n}(q, 0)=(h_{n,0}(q,0),h_{n,1}(q,0))$. However the fact that $\mathcal{L}$ is even entails that 
	\begin{equation}\label{partcasepnul}
		h_{n,1}(q,0))=-\frac{h_{n,0}(q,0)}{2}\, \Big(1+\frac{1}{n}\Big).
	\end{equation}
	Equality \eqref{partcasepnul}, combined with \eqref{Nantes 4.16} brings us to introduce a new probability law $\mathbf{P}_{n, h}$   on the random walk $X$ obtained 
	as 
	\begin{equation}\label{defpsym}
		\mathbf{P}_{n, h}:=\mathbf{P}_{n, \big(h,-\frac{1}{2} h (1+\frac{1}{n})\big)} \quad \text{for}\quad   h\in \bigg(-\frac{n \beta}{n-1}, \frac{n \beta}{n-1}\bigg).
	\end{equation}
	The  normalisation constant of $\mathbf{P}_{n, h}$  may be expressed  as $e^{-n\, \cG_{n}(h)}$ with
	\begin{equation}\label{eq:Gnh}
		\begin{aligned}
			\cG_{n}(h) & :=\frac{1}{n} \mathcal{L}_{\Lambda_{n}}\left[h,-\frac{h}{2}\left(1+\frac{1}{n}\right)\right] \quad \text { for } \quad h \in\left(-\frac{n \beta}{n-1}, \frac{n \beta}{n-1}\right).
		\end{aligned}
	\end{equation}
	In view of~\eqref{eq:calculLambda_n}, its continuous counterpart comes as
	\begin{equation}
		\cG(h) := \int_0^1 \lmgf\left(h\left(\frac{1}{2} - x\right)\right) \dd x, \quad \text{ for } h \in (-\beta,\beta).
	\end{equation}
	The following lemma can be seen as the particular case of Items (1) and (2) in Proposition~\ref{c1diffeo} when $p=0$. 
	\begin{lemma}[Lemma 5.3 in~\cite{Legrand_2022}]
		\label{Nantes Lemma 4.5}
		Let $\beta>0$. %
		\begin{enumerate}
			\item  For $n \geq 2$,  the mapping $\cG_{n}$ is $\sC^{2}$ and strictly convex on $(-\frac{n \beta}{n-1}, \frac{n \beta}{n-1})$. Moreover, $\cG_{n}'$ is a $\sC^1$-diffeomorphism from $(-\frac{n \beta}{n-1}, \frac{n \beta}{n-1})$ to $\R$.
			\item  The mapping $\cG$ is $\sC^{2}$ and strictly convex on $(-\beta,\beta)$. Moreover, $\cG'$ is a $\sC^1$-diffeomorphism from $(-\beta,\beta)$ to $\R$.
		\end{enumerate}
	\end{lemma}
	\begin{remark}\label{htildeqincreasing}
		For the sake of conciseness, for every $q\geq 0$ we set
		\beq
		h_n^q:=h_{n,0}(q,0)
		\eeq
		and we let $\htildeq$ be the first coordinate of 
		$\widetilde{\boldsymbol{h}}(q, 0)$ that we introduced in Proposition \ref{c1diffeo} (2). Once again, because $\mathcal{L}$ is even we observe that the continuous counterpart of \eqref{partcasepnul} holds true, i.e.,
		\begin{equation} \label{Nantes 7.22}
			\boldsymbol{\Tilde{h}}(q,0) = \Big( \htildeq, -\frac{\htildeq}{2}\Big).
		\end{equation}
		The functions $q\mapsto h_{n}^q$ and $q\mapsto \htildeq$ are consequently
		the restrictions to $(0,\infty)$ of the inverse functions of $\cG_{n}'$ and $\cG'$.	
		As a consequence of Lemma~\ref{Nantes Lemma 4.5} the function  $q \longrightarrow \htildeq$ is increasing. 
	\end{remark}
	We finally observe that the exponential tilt of $\mathbf{P}_{n, h}$ may be expressed as
	\begin{equation}
		\frac{\mathrm{d} \mathbf{P}_{n, h}}{\mathrm{~d} \mathbf{P}_{\beta}}(X)=e^{-\psi_{n, h}\left(A_{n}, X_{n}\right)}, \text { with } \psi_{n, h}(a, x):=-\frac{h}{n}a+\frac{h}{2}\left(1+\frac{1}{n}\right) x+n \cG_{n}(h), \quad x, a \in \mathbb{Z}.
		\label{Nantes 4.24}
	\end{equation}
	As in the previous case,  Proposition~\ref{propapproxG} below provides  the convergence  speed of the discrete quantities 
	$\cG_{N}(h)$ and $h_N^q$  towards $\cG(h)$ and $ \htildeq$ respectively. 
	\begin{proposition}[Propositions 5.1 and 5.4 in~\cite{Legrand_2022}]\label{propapproxG}
		For every $[q_1,q_2] \subseteq (0,\infty)$ and $[h_1,h_2]\subseteq (-\gb,\gb)$, there exists $C>0$ and $N_0 \in \N$ such that, for every $N \geq N_0$ and 
		$h\in [h_1,h_2]$: 
		\begin{equation}
			|\cG_{N}(h) - \cG(h)| \leq \frac{C}{N^2}\ ,
			\label{Nantes 4.20}
		\end{equation}
		and for every $q \in [q_1,q_2] \cap \frac{\N  }{N^2}$:
		\begin{equation}
			| h_N^q - \htildeq| \leq \frac{C}{N^2}.
			\label{Nantes approxaux}
		\end{equation}
	\end{proposition}
	\begin{remark}[Time-reversal property]
		\label{rmk:time-rev-prop}
		If $|h|<\beta/2$ and $Z$ is distributed as $\ProbaTilteeSansParenthese{h}$, as in~\eqref{Nantes 4.3}, then one can check that $-Z$ is distributed as $\widetilde{\mathbf{P}}_{-h}$. Recalling~\eqref{eq:calculLambda_n} we note that under $\mathbf{P}_{n, h}$, the increments of $X$, namely  $X_k-X_{k-1}$ for $k\in \{1,\dots,N\}$  are independent and follow respectively the tilted law
		$\ProbaTilteeSansParenthese{\frac{h}{2}\left(1-\frac{2 k-1}{n}\right)}$.
		Therefore, $X$ is time-reversible, i.e.,
		\begin{equation}
			\left(X_{k}\right)_{k=0}^{n} \underset{\text {(law)}}{=}\left(X_{n-k}-X_{n}\right)_{k=0}^{n}
			\label{Nantes 7.32}
		\end{equation}
		We deduce therefrom that the random walk $X$ distributed as $\mathbf{P}_{n, h}$ is an inhomogeneous Markov chain that satisfies for all $j \in\{1, \ldots, n-1\}$ and $y \in \mathbb{Z}$,
		\begin{equation}
			\mathbf{P}_{n, h}\left(\left(X_{j+k}\right)_{k=1}^{n-j-1} \in \cdot\ ,\ X_{n}=0 \mid X_{j}=y\right)=\mathbf{P}_{n, h}\left(\left(X_{n-j-k}\right)_{k=1}^{n-j-1} \in \cdot\ ,\ X_{n-j}=y\right).
		\end{equation}
		Finally, note that the case $h=0$ corresponds to the random walk $X$ with i.i.d. increments of law $\LoiMarchealeatoireSansParenthese$.
	\end{remark}
	\subsection{Analysis of auxiliary functions}
	\label{sec:analysis_aux}
	In this section we analyse the function $\psi$ displayed in~\eqref{deftildepsi} and the function
	\beq
	\label{eq:def_penalite}
	\penalite: a\in (0,\infty)\mapsto a\log \Gamma_\beta+a \psi (\tfrac{1}{a^2},\delta),
	\eeq
	which play a key role in deriving the asymptotic behaviour of the auxiliary partition functions in~\eqref{eq:D_Nqdelta} and ultimately expressing the surface free energy as a variational formula, see~\eqref{eq:gbetadelta-varfor}.
	\par Let us start with the regularity properties of $\psi$. Recalling the two cases in~\eqref{deftildepsi}, we first define, for every $0\le \gd < \gb$:
	\beq
	\label{eq:explicit-qdelta0}
	q_\delta := \inf\Big\{q>0 \colon \gd > \gd_0(q) =  \tfrac{\gb}{2} - \tfrac{\htildeq}{2}\Big\}.
	\eeq
	Two cases arise:
	\begin{enumerate} 
		\item If $0\le \gd<\gb/2$ then, by Remark~\ref{htildeqincreasing}, $q_\delta$ is actually the unique solution in $q>0$ of 
		${\htildeq}/{2}= {\beta}/{2}-\delta$ and, by definition of $\htildeq$, we have the following explicit expression:
		\begin{equation}\label{eq:explicit-qdelta}
			q_\delta := \int_0^1 (x-\tfrac12) \lmgf'\Big((\tfrac{\beta}{2}-\delta)(2x-1)\Big)\dd x.
		\end{equation}
		Note that the factor $x-1/2$ in~\eqref{eq:explicit-qdelta} may be replaced by $x$, since $\lmgf'$ is odd. Moreover, the cases $\gd \le \gd_0(q)$ and $\gd > \gd_0(q)$ in the expression of $\psi$, see~\eqref{deftildepsi}, correspond to $q\le q_\gd$ and $q> q_\gd$, respectively.
		\item If $\gb/2 \le \gd < \gb$ then the inequality in~\eqref{eq:explicit-qdelta0} is always true, hence $q_\gd =0$ and, from~\eqref{deftildepsi}, $\psi(q,\delta) = \cH_\delta(\sqdelta)-q \sqdelta$ for every $q>0$. 
	\end{enumerate}
	\begin{lemma}\label{regulpsitilde}
		Let $0\le \gd < \gb$. The mapping $q\mapsto \psi(q,\delta) $ is $\sC^1$ on $(0,\infty)$. Moreover, it is $\sC^2$ on $(0,q_\delta)\cup (q_\delta,\infty)$ if $0<\gd< \beta/2$ (in which case $q_\gd>0)$ and it is $\sC^2$ on $(0,\infty)$ if $\beta/2 \le \gd < \beta$ (in which case $q_\gd=0)$.
	\end{lemma}
	\begin{proof}[Proof of Lemma~\ref{regulpsitilde}] 
		By~\eqref{eq:delta0q} and~\eqref{deftildepsi},
		\begin{equation}\label{exprepsitilde}
			\psi(q,\delta)=\ind_{\{\frac{\htildeq}{2}\leq \frac{\beta}{2}-\delta\}}\ \big(\cG(\htildeq)-q \htildeq\big)+
			\ind_{\{\frac{\htildeq}{2}> \frac{\beta}{2}-\delta\}}\ \big(\cH_\delta(\sqdelta)-q \sqdelta \big).
		\end{equation}
		Since, by Lemmas~\ref{Nantes Lemma 4.5} and~\ref{Nantes Lemma 9.5}, both functions $q\mapsto \cG(\htildeq)-q \htildeq$ and $q\mapsto \cH_\delta(\sqdelta)-q \sqdelta$ are $\sC^2$ on $(0,\infty)$, it is sufficient to 
		show that the first derivatives coincide at $q_\delta$. Indeed, we compute, when $q\in(0, q_\gd)$,
		\beq\label{derivpar}
		(\partial_q\psi)(q,\delta)=\tilde h'(q) \cG'(\htildeq)-q \tilde h'(q)-\htildeq= -\htildeq,
		\eeq
		and when $q\in(q_{\gd},\infty)$,
		\beq
		(\partial_q\psi)(q,\delta)= \partial_q (\sqdelta) (\cH_\delta)'(\sqdelta) -q \partial_q (\sqdelta)- \sqdelta
		=-\sqdelta.
		\eeq
		We may now check that the first derivatives coincide at $q_\delta$. Indeed, one can verify that $\cH'_\gd(\tilde h(q_\gd)) = \cH'_\gd(\gb-2\gd) = q_\gd$, by~\eqref{eq:explicit-qdelta}. As for the second derivatives, we obtain
		\beq
		\label{derivparordre2}
		(\partial^2_q\psi)(q,\delta)=
		\begin{cases}
			\tilde h'(q) & q \in(0,q_\gd),\\
			s_\gd'(q) & q\in(q_\gd,\infty).
		\end{cases}
		\eeq
	\end{proof}
	Let us now turn to the properties of the function $\penalite$ defined in~\eqref{eq:def_penalite}. In the rest of the paper we often jump from one dummy variable $a\in(0,\infty)$ to another dummy variable $q\in(0,\infty)$ via the relation $q=1/a^2$. First, we focus on the concavity of the function. To this end, we define:
	\beq
	\label{eq:def_qstar}
	\qstar := \inf\{q>0\colon \gd - \gb/2 + s_\gd(q)\ge 0\}
	\eeq
	and notice that
	\beq
	\cH'_\gd(\gb/2 - \gd) = \int_0^1 x \lmgf'((\gd - \gb/2)(1-x))\dd x,
	\eeq
	which implies, since $\lmgf$ is odd, that
	\beq
	\sign(\cH'_\gd(\gb/2 - \gd)) = \sign(\gd - \gb/2).
	\label{Nantes 4.54}
	\eeq
	We may now distinguish between two cases:
	\begin{enumerate}
		\item If $0<\gd \le \gb/2$ then $\cH_\gd'(\gb/2- \gd)\le 0$, hence $s_\gd(0)\ge \gb/2 -\gd$, by Lemma~\ref{Nantes Lemma 9.5}. Since $s_\gd$ is increasing, $\gd - \gb/2 + s_\gd(q)\ge 0$ for every $q>0$, and $\qstar = 0$.
		\item If $\gb/2 < \gd< \gb$ then $\cH_\gd'(\gb/2- \gd)>0$, hence $s_\gd(0)< \gb/2 -\gd$. Therefore, $q_\gd^*$ is the only solution in $q>0$ of the equation $\gd-\gb/2+s_\gd(q)=0$.
	\end{enumerate}
	\begin{lemma}[Concavity/Convexity]
		\label{concavity}
		For every $\beta>\beta_c$ and $\delta<\beta$, the function 
		$\penalite$ is $\sC^1$ on $(0,\infty)$. It is $\sC^2$ on $(0,1/\sqrt{q_\delta})\cup (1/\sqrt{q_\delta},\infty)$ if $0<\gd< \beta/2$ and it is $\sC^2$ on $(0,\infty)$ if $\beta/2 \le \gd < \beta$. If $\gd \le \gb/2$ (in which case $\qstar = 0$) it is
		strictly concave on $(0,\infty)$. If $\gb/2<\gd <\gb$ (in which case $\qstar>0$) then it is strictly concave on $(0, 1/\sqrt{\qstar})$ and strictly convex on $(1/\sqrt{\qstar}, \infty)$. 
	\end{lemma}
	The proof can be found in Appendix \ref{Nantes Appendix fct 1}
	The next step is to determine the limits of $\penalite(a)$ when $a\to 0$ and $a\to +\infty$. Recall that the parameter $a$ stands for the prescribed horizontal extension ($L= qN^2$ hence $N=a\sqrt{L}$ from \eqref{eq:D_Nqdelta} and~\eqref{defnun}).  This step is important to restrict the set of possible horizontal extensions to a {\it compact} set, see Section~\ref{sec:a-priori-bounds}. To this end, we first notice that 
	for every $|x|<\gb/2$, the map $\gd \mapsto \cH'_\gd(\gb/2 - \gd -x)$ is increasing. Indeed, its derivative writes:
	\beq
	\int_0^1 t(1-t) \lmgf''((\gb/2-\gd-x)t+\gd - \gb/2) \dd t >0,
	\eeq
	which is positive, by strict convexity of $\lmgf$.
	Recall the definitions of $\cC_{\rm{good}}$ and $\bar \gd(\gb)$ in~\eqref{eq:defCgood} and~\eqref{eq:def-bar-gd}.%
	\begin{lemma}[Limits]\label{lem:limits} For every $0< \gd < \gb$,
		$\penalite(a)$ converges to $-\infty$ as $a\to 0$. For every $0< \gd <\bar \gd(\gb)$,
		$\penalite(a)$ converges to $-\infty$ as $a\to +\infty$ and for every $\bar \gd(\gb)< \gd < \gb$ (provided this case is not empty), $\penalite(a)$ converges to $+\infty$ as $a\to +\infty$.
	\end{lemma}
	The proof of this lemma can be found in Appendix~\ref{Nantes Appendix fct 2}.
	As an immediate corollary of Lemma~\ref{concavity} and Lemma~\ref{lem:limits}, we obtain:
	\begin{corollary}
		\label{cor:unique_max_in_Cgood}
		If $(\gb,\gd)\in \cC_{\rm{good}}$ then $\penalite$ admits a unique maximizer on $(0,\infty)$.
	\end{corollary}
	In view of Lemma~\ref{lem:limits}, a natural question is to determine for which values of $\beta$ we have $\bar \gd(\gb)<\gb$, respectively $\bar \gd(\gb)=\gb$. This is the content of the following lemma.
	\begin{lemma}
		\label{lem:bar-gd}
		If $\gb>\gb_c$ is close enough to $\gb_c$ then $\gb/2 \le \bar \gd(\gb) < \gb$. 
		However, there exists $\gb_* > \gb_c$ such that $\bar \gd(\gb)=\gb$ for every $\gb\ge\gb_*$.
	\end{lemma}
	The proof of Lemma \ref{lem:bar-gd} can be found in Appendix \ref{Nantes Appendix fct 3}. Numerically, we have $\gb_c \approx 1,219$ and $\gb_* \le \pi/\sqrt{3} \approx 1,814$ (see the proof of Lemma~\ref{Nantes Lemma 9.13BIS}). As a straightforward consequence of Lemma~\ref{lem:bar-gd}, the set $\mathcal{C}_{\text{bad}}:=\cC\setminus \cC_{\text{good}}$ defined in~\eqref{eq:defCgood} is bounded.

	\par Let us now make a few remarks on the maximizer of $q \mapsto \penalite(1/\sqrt{q})$ when $(\gb, \gd)\in \cC_{\rm {good}}$. First, we observe that:
	\begin{lemma}
		\label{lem:technic_IPP}
		For every $q \geq 0$:
		\begin{equation}\label{Nantes 4.29}
			\cH_\delta(\sqdelta) = \lmgf(\sqdelta + \delta - \beta/2) - q \sqdelta ;
		\end{equation}
		\begin{equation} \label{Nantes 4.30}
			\cG(\htildeq) = \lmgf({\htildeq}/{2}) - q\htildeq.
		\end{equation}
	\end{lemma}
	\begin{proof}[Proof of Lemma~\ref{lem:technic_IPP}]
		An integration by part gives:
		\begin{equation}
			\begin{aligned}
				\cH_\delta(s) = \int_0^1 \lmgf(s t+ \delta - \beta/2) \dd t
				&=[t \lmgf(s t+ \delta - \beta/2)]_0^1 - s \int_0^1 t \lmgf'(s t+ \delta - \beta/2) \dd t \\
				&= \lmgf(s + \delta - \beta/2) - s \cH_\delta'(s),
			\end{aligned}
		\end{equation}
		and \eqref{Nantes 4.29} readily follows from \eqref{Nantes 3.10}.
		A similar computation combined with~\eqref{Nantes 3.7} gives \eqref{Nantes 4.30}.
	\end{proof}
	Combining the derivative of $\penalite$ 
	in \eqref{Nantes 4.26} and \eqref{eq:calcul_derivees_T} with Lemma~\ref{lem:technic_IPP}, we get that
	\beq
	\penalite'(1/\sqrt{q}) =
	\begin{cases}
		\log \Gamma_\beta + \lmgf(\sqdelta + \delta - \beta/2) & (q > q_\delta)\\
		\log \Gamma_\beta + \lmgf(\htildeq/2) & (q < q_\delta).
	\end{cases}
	\eeq
	These observations lead to the following:
	\begin{remark}[On the maximizer of $\penalite$]\label{Nantes Remark 4.13}
		If $(\gb,\gd)\in \cC_{\rm {good}}$ then the unique maximizer of $q \mapsto \penalite(1/\sqrt{q})$, that we denote by $\bar q_{\gb,\gd}$, satisfies $\bar q_{\gb,\gd}= \bar q_{\gb,0}$ and $-\log \Gamma_\beta = \lmgf(\tilde h(\bar q_{\gb,0})/2)$ if $(\gb,\gd)\in \cD\cC$ and $-\log \Gamma_\beta = \lmgf(s_\gd(\bar q_{\gb,\gd}) + \delta - \beta/2)$ if $(\gb,\gd)\in \cA\cC$.
	\end{remark}
	\par To close this section, we shortly come back to the function $\psi$ and state a lemma that will be essential for computing the order of the surface transition. Recall that $\gd_0(q) = \frac{\beta}{2}-\frac{\tilde h(q)}{2}$.
	\begin{lemma}\label{inegcentral}
		For $(\gb, \gd)\in \Cc$ and $q>0$ such that $\gd > \gd_0(q)$ it holds that
		$\psi(q,\delta)>\psi(q,0)$. Moreover, as $\gep\to0^+$,
		\begin{equation} \label{Nantes 4.27}
			\psi(q,\gd_0(q) + \varepsilon) - \psi(q,0) \sim C
			\varepsilon^2,
		\end{equation}
		where $C = \frac{\lmgf'(\htildeq/2)( \lmgf'(\htildeq/2) - 4q) }{2\htildeq( \lmgf'(\htildeq/2) - 2q)} > 0$.
	\end{lemma}
	The detailed proof is deferred to Appendix \ref{Nantes Appendix C.1}.
	\subsection{Sharp asymptotics of auxiliary partition functions}\label{secsharpaux}
	In Proposition~\ref{asympteach3reg} below, we provide sharp asymptotics for the auxiliary partition function introduced in Section~\ref{eq:D_Nqdelta}, in each of the three (desorbed, critical and adsorbed) regimes lying in the collapsed phase. Its proof is postponed to Section~\ref{compid}. Beforehand, we define $\vartheta : (-\beta/2,\beta/2) \rightarrow \R$ as:
	\begin{equation}
		\vartheta(h)= \int_0^1 x^2 \lmgf''\Big(h\Big(x-\frac{1}{2}\Big)\Big)\dd x 
		\int_0^1  \lmgf''\Big(h\Big(x-\frac{1}{2}\Big)\Big)\dd x -
		\left[ 
		\int_0^1 x \lmgf''\Big(h\Big(x-\frac{1}{2}\Big)\Big)\dd x
		\right]^2\, .
		\label{Nantes 7.6}
	\end{equation}
	We also recall the definitions of $\kappa^x(h)$ in \eqref{defkappa}, $\psi$ in \eqref{deftildepsi}, $\gd_0(q)$ in~\eqref{eq:delta0q}, and $q^*_\gd$ in \eqref{eq:def_qstar}.
	\begin{proposition}
		\label{asympteach3reg}
		Let $\beta > \beta_c$ and $0<q_1<q_2<\infty$.%
		\begin{enumerate}
			\item For $\delta<\min\{ \gd_0(q)\colon\;  q \in [q_1,q_2]\}$,
			\begin{equation}
				D_{N}(q,\delta) = \frac{C_{\beta,q,\delta}}{N^2}e^{N {\psi}(q,0)}(1+o(1)),
				\label{Nantes 7.1}
			\end{equation}
			where $o(1)$ is uniform in $q\in (q_1,q_2)$, and
			\begin{equation}
				C_{\beta,q,\delta} =  \frac{\kappa^0\big(  \tfrac{\tilde{h}(q)}{2} \big)}{2 \pi \vartheta (\tilde{h}(q))^{\frac{1}{2}}} \left( \frac{1}{1-e^{\tilde u}} - \frac{1-\kappa^0\big(  \frac{\tilde{h}(q)}{2} \big)}{1-e^{\tilde{u}-\tilde{h}(q)}} \right),
				\label{Nantes 7.24}
			\end{equation}
			with $\tilde u:=\delta-\gd_0(q)$.
			\label{Nantes Proposition 7.1}
			\item
			For $q \in (q_1,q_2)$ and
			$\delta = \gd_0(q)$, for all $R \geq 0$ and uniformly over $c\in [-R,R]$
			\begin{equation}
				D_{N}(q+\tfrac{c}{\sqrt{N}},\delta)= 
				\frac{C_{\beta,q,c}^{\rm{crit}}}{N^{3/2}} e^{N \psi(q,0)-\tilde h(q) c\sqrt{N}}(1+o(1)), 
				\label{Nantes 7.5}
			\end{equation}
			where $o(1)$ is uniform in $q\in (q_1,q_2)$, and
			\begin{equation}
				C_{\beta,q,c}^{\rm{crit}}=\kappa^0(\htildeq/2) \int_0^\infty f_{\htildeq}(c,z) \dd z,
			\end{equation}
			with $f_{\htildeq}$ defined in \eqref{def_f_h}.
			\label{Proposition Partition function critical curve}
			\item 
			For $\gd>0$ and $[q_1,q_2]\subset(q^*_\gd, +\infty)$ such that $\delta > \delta_0(q_1)$.
			We have the following estimate:
			\begin{equation}
				D_{N}(q,\delta)= 
				\kappa^0\big(\htildeq \big) \frac{\xi(q,\delta)}{N^{3/2}} e^{N \psi(q,\delta)}(1+o(1)),
				\label{Nantes 7.4}
			\end{equation}
			where the $o(1)$ is uniform over $q \in (q_1,q_2)$, and 
			\begin{equation}
				\xi(q,\delta) := \frac{e^{\lmgf( \delta-\frac{\beta}{2} + \sqdelta) - \lmgf( \delta-\frac{\beta}{2})} }{\sqrt{2\pi \int_0^1 x^2 \lmgf''\Big( \delta-\frac{\beta}{2} + \sqdelta x \Big) \dd x }}.
			\end{equation}
			\label{Proposition Partition function above critical curve}
		\end{enumerate}
	\end{proposition}
	In our way of proving Theorem~\ref{varfor}, we shall need to check the assumption in Item (3) of Proposition~\ref{asympteach3reg}. To this end, we can rely on the following lemma:
	\begin{lemma} 
		\label{Nantes Lemme 5.1} 
		If $(\beta,\delta) \in \cC_{\rm{good}}$ then the maximizer of $q \mapsto \penalite(1/\sqrt{q})$ is larger than $q^*_\gd$.
	\end{lemma}
	\begin{proof}[Proof of Lemma~\ref{Nantes Lemme 5.1}]
		If $\gd \le \gb/2$ there is nothing to prove since then $q^*_\gd=0$ by Item (1) below \eqref{Nantes 4.54}. Now assume that $\gb/2 < \gd < \gb$.
		By the definition of $q_\gd^*$ in~\eqref{eq:def_qstar} and Remark~\ref{Nantes Remark 4.13}, 
		when $\delta \geq \beta/2$, we have 
		$s_\gd(q_\gd^*) + \delta - \beta/2 = 0$, and the maximizer $\bar q_{\gb,\gd}$ 
		satisfies $\lmgf(s_\gd(\bar q_{\gb,\gd}) + \delta - \beta/2) = -
		\log \Gamma_\beta > 0$. Because both $q \in \R_+ 
		\mapsto s(q)$ and $x \in [0,\beta/2] \mapsto 
		\lmgf(x) $ are increasing functions, and since $\lmgf(0)=0$, this 
		proves that $\bar q_{\gb,\gd} > q_\gd^*$.
	\end{proof}
	Finally, Lemma \ref{rough-bounds} gives a uniform control on the sequence $(D_N(q,\delta))_{N\geq 1}$:
	\begin{lemma}\label{rough-bounds}
		Let $q_2>0$. There exists $c>0$ such that for every $N\in \N$ and $q\in (0,q_2]$, 
		\begin{equation}
			D_{N}(q,\delta)\leq c\,  e^{N \psi(q,\delta)}.
		\end{equation}
	\end{lemma}
	\begin{proof}[Proof of Lemma~\ref{rough-bounds}]
		We distinguish between two cases.\\
		(i) If $\gd_0(q)< \delta \le \beta$, we obtain with the help of Lemma~\ref{Nantes Lemma 4.7}:
		\begin{equation}
			\begin{aligned}\label{calculdeproba2}
				D_{N}(q,\delta)=  \mathbf{E}_{\beta}\Big[e^{(\delta-\frac{\beta}{2})X_N}  1_{\{\cV_{N,qN^2,+} \}}  \Big]&\leq 
				\cst \ProbaSurCritique{\cV_{N,qN^2,+}}\  e^{N \big[\cH_{N,\delta}(s_{q}^\delta)-s_{q}^\delta q\big]}\\
				&\leq  \cst \, e^{N \big[\cH_{N,\delta}(s_{q}^\delta)-s_{q}^\delta q\big]}.
			\end{aligned} 
		\end{equation}
		It remains to apply \eqref{Nantes 9.12} in Proposition \ref{approxhn}  to conclude that for $N$ large enough, and for $q\in [q_1,q_2]$,
		\begin{align}\label{calculdeprobafin}
			e^{N \big[\cH_{N,\delta}(s_{q}^\delta)-s_{q}^\delta q\big]}
			&\leq  2 e^{N \big[\cH_{\delta}(s_{q}^\delta)-s_{q}^\delta q\big]}=2 e^{\psi(q,\delta)N}.
		\end{align} 
		(ii) If $0\le \delta\le \gd_0(q)$, we apply the tilting in \eqref{Nantes 4.16} with 
		$\boldsymbol{h}=(\htildeq,-\frac{\htildeq}{2})$ to get
		\begin{equation}
			\begin{aligned}\label{calculdeproba3}
				D_{N}(q,\delta)&= \mathbf{E}_{\beta}\Big[e^{(\delta-\frac{\beta}{2})X_N}  1_{\{\cV_{N,qN^2,+} \}}  \Big]\\
				&= 
				\mathbf{E}_{N,\htildeq,-\frac12 \htildeq} \Big[e^{(\delta-\gd_0(q)) X_N}
				1_{\{\cV_{N,q N^2,+}\}}  \Big]\  e^{N \big[\cG_{N}(\htildeq)-\tilde h_{q} q\big]}\\
				&\leq  e^{N \big[\cG_{N}(\htildeq)-\tilde h_{q} q\big]},
			\end{aligned}
		\end{equation}
		where we have used that  $X\in\cV_{N,q,+}$ necessarily implies $X_N\geq 0$. Using~\eqref{Nantes 4.20} in Proposition~\ref{propapproxG}, together with the continuity of
		$q\mapsto \htildeq$, there exists an $N_0\in \N$ such that uniformly in $q\in [0,q_2]$, for $N\geq N_0$, 
		\begin{align}\label{calculdeprobater}
			e^{N \big[\cG_{N}(\htildeq)-\tilde h_{q} q\big]}&\leq 2 
			e^{N \big[\cG(\htildeq)-\htildeq q\big]}=2e^{\psi(q,\delta)N}.
		\end{align} 
		This completes the proof.
	\end{proof}

	\subsection{Local limits}\label{loclim}
		The last main tool that we will use throughout the paper are Gnedenko-type local limit theorems. In this section, we present  three theorems of that type involving $A_N$ and $X_N$, and introduce a change of measure used in the $\mathcal{AC}$ phase.
	\subsubsection{Local limit inside  \texorpdfstring{$\mathcal{DC}$}{DC} phase and at  the critical curve}
	We
	recall the definitions of $\lmgf_\Lambda$ in~\eqref{Nantes 7.20} and of $\cD_\gb$ in~\eqref{eq:def-Dbeta}. For every $\bh \in \cD_\gb$, we define the matrix
	\begin{equation}
		\bB(\bh) := \text{Hess } \lmgf_\Lambda(\bh)
		\label{Nantes 8.52}
	\end{equation}
	and the following Gaussian probability density:
	\begin{equation}f_{\bh} : z \in \R^2 \mapsto \frac{1}{2\pi \sqrt{\det \bB(\bh)}} \exp \left(
		-\frac{1}{2} \langle \bB(\bh)^{-1}z,z \rangle
		\right).
		\label{def_f_h}
	\end{equation}
	Recall the definition of $\boldsymbol{\Tilde{h}}(q,0)$ in~\eqref{Nantes 7.22}. The following proposition is a slight quantitative upgrade of \cite[Proposition 6.1]{CNP16}, in the sense that we provide a rate of convergence to zero.
	\begin{proposition}\label{loccontrlimtheoarepos}
		Let $[q_1,q_2] \subset \R$. As $N\to\infty$,
		\begin{equation}
			\underset{q \in [q_1,q_2]}{\sup} \underset{x,y \in \Z}{\sup} 
			\left| 
			N^2 \mathbf{P}_{N,h_N^q}(A_N = qN^2 + x, X_N = y ) - f_{\boldsymbol{\Tilde{h}}(q,0)}\left(\frac{x}{N^{3/2}}, \frac{y}{N^{1/2}} \right)
			\right| 
			= O \left( \frac{(\log N )^4}{\sqrt{N}} \right).
		\end{equation}
		\label{Nantes Proposition 8.10}
	\end{proposition}
	\begin{proof}[Proof of Proposition~\ref{Nantes Proposition 8.10}] Change the constant $A$ by $\log N$ in the proof of~\cite[Proposition 6.1]{CNP16}, and everything follows. 
	\end{proof}
	We also need a local limit theorem that applies exclusively to the area enclosed by the walk. Let us first denote by $l_{\sigma^2}$ the density of  $\cN(0,\sigma^2)$, i.e., 
	\begin{equation}
		l_{\sigma^2}(x) = \frac{1}{\sqrt{2\pi \sigma^2}}\exp\Big(-\frac{x^2}{2\sigma^2}\Big), \qquad x\in \bbR.
	\end{equation}
	We also set
	\beq
	b(h) := \int_0^1 (x-1/2)^2 \lmgf''(h(x-1/2)) \dd x, \qquad |h|<\beta/2.
	\eeq
	%
	%
	\begin{lemma}
		Let $[q_1,q_2] \subset \R$. As $N \longrightarrow \infty$,
		\begin{equation}
			\underset{q \in [q_1,q_2]}{\sup}
			\underset{x \in \bZ}{\sup}
			\left| N^{3/2} \Proba{N,h_N^q}{A_N=qN^2+x} - l_{b\big(\htildeq\big)} \left( \frac{x}{N^{3/2}} \right) \right| = O \left( \frac{(\log N)^4}{\sqrt{N}} \right),
			\label{Nantes 8.42}
		\end{equation}
		\label{Nantes Lemma 8.15}
	\end{lemma}
The proof of this lemma is left to the reader, as it follows very closely that of Carmona, Nguyen and Pétrélis~\cite[Section 6.1]{CNP16}.
	The purpose of the next lemma is to show that the endpoint $X_N$  has variations of size $\sqrt{N}$ around $0$ under $\ProbaSansParenthese{N,h_N^q}$. Its proof is postponed to Appendix~\ref{Nantes Proof of Lemma 9.8}.
	\begin{lemma} Let $0 < q_1 < q_2 < \infty$. There exists $C,c > 0$ such that, for all $q \in [q_1,q_2]$ and $b>0$,
		\begin{equation}
			\limsup_{N\to \infty}\ \Proba{N,h_N^q}{|X_N| \geq b \sqrt{N}} \leq C e^{-c b^2}.
		\end{equation}
		\label{Nantes Lemma 4.24}
	\end{lemma}
	\subsubsection{Local limit for the \texorpdfstring{$\mathcal{AC}$}{AC} phase}
	First, we recall \eqref{defPNq}, that is the relevant change of measure in the $\mathcal{AC}$ phase. We then define $c(s) := \int_0^1 x^2 \lmgf''(\delta - \beta/2 + sx) \dd x$ for $s \in (-\delta, \beta - \delta)$.
	\begin{lemma}
		Let $\gd>0$ and $[q_1,q_2]\subset(q^*_\gd, +\infty)$. As $N \longrightarrow \infty$,
		\begin{equation}
			\underset{q \in [q_1,q_2]}{\sup}
			\underset{x \in \bbZ}{\sup}
			\left| N^{3/2} \ProbaSurCritique{A_N=qN^2+x} -l_{c\big(\sqdelta\big)} \left( \frac{x}{N^{3/2}} \right) \right| = O \left( \frac{(\log N)^4}{\sqrt{N}} \right).
			\label{Nantes 9.23}
		\end{equation}
		\label{Nantes Lemma 8.29}
	\end{lemma}
	The proof of this lemma can be found in Appendix~\ref{Nantes Appendix A.3}. 
	We are now left with stating the counterpart of Lemma \ref{Nantes Lemma 4.24} inside the $\mathcal{AC}$ phase, which ensures us that the endpoint $X_N$ has variations of size $\sqrt{N}$ around its mean under $\ProbaSurCritiqueSansParenthese$. 
	\begin{lemma} Let $\gd>0$ and $[q_1,q_2]\subset(q^*_\gd, +\infty)$. There exists $C,c > 0$ such that, for all $q \in [q_1,q_2]$ and $b>0$,
		\begin{equation}
			\limsup_{N\to \infty}\ \ProbaSurCritique{\Big|X_N - \EsperanceSurCritique{X_N}\Big| \geq b \sqrt{N}} \leq C e^{-c b^2}.
		\end{equation}
		\label{Nantes Lemma 9.8}
	\end{lemma}
	The proof of Lemma~\ref{Nantes Lemma 9.8} is postponed to Appendix \ref{Nantes Proof of Lemma 9.8}. Note that the following equation, derived in \eqref{Nantes AA.20}, will be useful in the sequel:
	\begin{equation}
		\begin{aligned}
			\EsperanceSurCritique{X_N}
			&= O(1) + N\int_0^1 \lmgf' \left( 
			\delta - \frac{\beta}{2} + \sqdeltaN t
			\right) \dd t, \qquad \text{as } N \to \infty.
		\end{aligned}
		\label{Nantes A.20}
	\end{equation}
	\begin{remark} We can actually deduce from \eqref{Nantes A.20} and Lemma \ref{Nantes Lemma 9.8} the limiting value for the size of the (renormalized) last stretch when $q>0$ is fixed. For $\delta > \delta_0(q)$ and $X_N$ sampled from $\ProbaSurCritiqueSansParenthese$, the following convergence in probability holds:
		\begin{equation}
			\frac{X_N}{N} \rightarrow \int_0^1 \lmgf' \left( 
			\delta - \frac{\beta}{2} + \sqdeltaN t
			\right) \dd t \quad \text{ as $N \rightarrow \infty$.}
		\end{equation}		
	\end{remark}
	\subsection{A-priori bounds on the horizontal extension}
	\label{sec:a-priori-bounds}
	Let us recall the notation used in \eqref{partfunbead}.
	\begin{lemma}\label{compact}
		For  $(\beta,\delta)\in \cC_{\rm{good}}$, there exists $0< a_1 < a_2 < \infty$ such that 
		\beq\label{eq:compact}
		Z_{L,\beta,\delta}^{\text{o}}=[1+o(1)] Z_{L,\beta,\delta}^{\text{o}}(N_{\ell}\in  [a_1,a_2]\sqrt{L}).
		\eeq
	\end{lemma}
	We will see in the proof that the lower bound on $N_\ell$ does not require that $(\beta,\delta)\in \cC_{\text{good}}$.
	\begin{proof}
		We split the proof into two parts, and prove that there 
		there exists a function $g : \R_+ \to \R$ such that $\lim g(a_1) = +\infty$ as $a_1\to 0$, a function $\tilde g: \R_+ \to \R$ such that $\lim \tilde g(a_2) = +\infty$ as $a_2 \to \infty$, and $L_0 \geq 0$ such that, for $L \geq L_0$,
		\begin{equation} \label{Nantes Compact Localisation0}
			Z_{L,\beta,\delta}^{\text{o}}(N_\ell \leq a_1 \sqrt{L}) e^{-\beta L} \leq e^{-g(a_1)\sqrt{L}},
		\end{equation}
		\begin{equation} \label{Nantes Compact Localisation}
			Z_{L,\beta,\delta}^{\text{o}}(N_\ell \geq a_2 \sqrt{L}) e^{-\beta L} \leq e^{-\tilde g(a_2)\sqrt{L}}.
		\end{equation}
		This is enough to conclude the proof: if we consider the trajectory $B = ((-1)^i N)_{i < N}$ with $N = \lfloor \sqrt{L} \rfloor$, which we complete with an additional vertical stretch (not longer than $N$) if some monomers remain, we obtain
		\beq
		\label{eq:easyLBZ}
		Z_{L,\beta,\delta}^{\text{o}} e^{-\beta L} \ge Z_{L,\beta,\delta}^{\text{o}}(B) e^{-\beta L} \geq e^{-\beta \sqrt{L}}.
		\eeq
		Combining \eqref{eq:easyLBZ} with \eqref{Nantes Compact Localisation0} and \eqref{Nantes Compact Localisation} gives the desired result.\\
		\par Let us now prove \eqref{Nantes Compact Localisation0}.
		Let $a_1>0$ (to be specified later) and $\varepsilon \leq a_1$. By Lemma~\ref{rwrep}, and since $\Gamma_\gb \le 1$, we have
		\beq
		Z_{L,\beta,\delta}^{\text{o}}(N_\ell = \varepsilon \sqrt{L}) e^{-\gb L} \leq \EsperanceMarcheAleatoire{e^{(\delta - \beta/2)X_{\varepsilon \sqrt{L}}} 1\{A_{\varepsilon \sqrt{L}} = L - \varepsilon \sqrt{L}\}}.
		\eeq
		Using the tilting defined in~\eqref{Nantes 4.3}, one has:
		\begin{equation}
			\label{eq:aprioribounds}
			Z_{L,\beta,\delta}^{\text{o}}(N_\ell = \varepsilon \sqrt{L}) e^{-\gb L} = \ProbaTiltee{\delta - \beta/2}{A_{\varepsilon \sqrt{L}} = L - \varepsilon \sqrt{L}} \left(\EsperanceMarcheAleatoire{e^{(\delta - \beta/2)U_1}}\right)^{\varepsilon \sqrt{L}}.
		\end{equation}
		Denoting $(U_i)_{i\ge 1}$ the increments of the random walk $X$, one has that
		\beq
		\{A_{\varepsilon \sqrt{L}} = L - \varepsilon \sqrt{L}\} \subset \Big\{|U_1| + ... + |U_{\varepsilon \sqrt{L}}| \geq \frac{\sqrt{L}}{2\varepsilon}\Big\},
		\eeq
		for $L$ large enough. Hence, by Chernov's bound, for every $\gl>0$:
		\beq
		\ProbaTiltee{\delta - \beta/2}{A_{\varepsilon \sqrt{L}} = L - \varepsilon \sqrt{L}} \leq \exp \left( -\lambda \frac{\sqrt{L}}{2 \varepsilon} + \varepsilon \sqrt{L} \log\left(\EsperanceTiltee{\delta - \beta/2}{e^{\lambda |U_1|}}  \right)\right).
		\eeq
		Since $\EsperanceTiltee{\delta - \beta/2}{e^{\lambda |U_1|}} = 1 + \lambda \EsperanceTiltee{\delta - \beta/2}{|U_1|} + O(\lambda^2)$ as $\gl \to 0$ and $\log(1+x) \leq 2x$ for $x$ small enough, we obtain for $\gl$ small enough:
		\begin{equation}
			\label{eq:aprioribounds2}
			\ProbaTiltee{\delta - \beta/2}{A_{\varepsilon \sqrt{L}} = L - \varepsilon \sqrt{L}} 
			\leq \exp \left( -\lambda \frac{\sqrt{L}}{2 \varepsilon} + 2\lambda \varepsilon \sqrt{L}  \EsperanceTiltee{\delta - \beta/2}{|U_1|}  \right).
		\end{equation}
		From \eqref{eq:aprioribounds} and \eqref{eq:aprioribounds2} there indeed exists $g : \R_+ \rightarrow \R$ such that 
		\begin{equation}
			Z_{L,\beta,\delta}^{\text{o}}(N_\ell = \varepsilon \sqrt{L}) e^{-\gb L} \leq e^{-g(\varepsilon)\sqrt{L}}
			\quad \text{and} \quad \limite{\varepsilon}{0}g(\varepsilon) = +\infty.
		\end{equation}
		Choosing $a_1$ small enough completes this part of the proof.\\
		\par Let us now move on to the proof of \eqref{Nantes Compact Localisation}, starting again from the formula in Lemma~\ref{rwrep}. If $\delta \leq \beta/2$, we simply bound the exponential therein by one and get
		\begin{equation}
			\EsperanceMarcheAleatoire{e^{(\delta - \beta/2)X_N} 1\{A_N = qN^2, X_{[1,N]}>0\}} \Gamma_\beta^N \leq e^{-cN}, \qquad \text{with } c = -\log \Gamma_\beta >0,
		\end{equation}
		which is enough to conclude. Assume now that $\delta > \beta/2$ and $(\gb,\gd)\in \cC_{\rm good}$ (defined in \eqref{eq:defCgood}). Using Lemma \ref{rough-bounds}, for every $q>0$,
		\begin{equation}
			\EsperanceMarcheAleatoire{e^{(\delta - \beta/2)X_N} 1\{A_N = qN^2, X_{[1,N]}>0\}} \leq \cst e^{N \psi(q,\delta)}.
		\end{equation}
		Using the function $\penalite$ defined in \eqref{eq:def_penalite}, we therefore have for every $a_2>0$,
		\begin{equation} \label{Above var tabula}
			\somme{N=a_2 \sqrt{L}}{L} \Gamma_\beta^N \EsperanceMarcheAleatoire{e^{(\delta - \beta/2)X_N} 1\{A_N = L-N, X_{[1,N]}>0\}} \leq \cst \somme{a \in [a_2,\infty) \cap (\bbN/\sqrt{L})}{} e^{\sqrt{L} T_\delta(a)}.
		\end{equation}
		Using~\eqref{Poitiers 11} (proof of Lemma~\ref{lem:limits}), one can see that $\limite{a}{\infty} \penalite'(a)<0$. Therefore, there exists $a_3>0$ such that $\penalite'(a)< \cst <0 $ for all $a \geq a_3$, hence $\penalite(a) \leq \penalite(a_3) + \cst (a-a_3)$, which settles~\eqref{Nantes Compact Localisation}.
	\end{proof}
	\begin{remark}
		Combining the second equations in \eqref{eq:compact} with \eqref{partfunbeadaux} we obtain also 
		\begin{equation}\label{uniqbead}
			\tilde Z^{\, \text{o}}_{L,\beta,\delta}=(1+o(1))  \sum_{a\in [a_1,a_2]\cap (\N/\sqrt{L})} 
			\Gamma_\beta^{a\sqrt{L}} D_{a\sqrt{L}}( q(a,L) , \delta),
		\end{equation}
		with
		\beq
		\label{eq:qaL}
		q(a,L):=\frac{1}{a^2}-\frac{1}{a\sqrt{L}}.
		\eeq
	\end{remark}
	\section{Proof of Theorems~\ref{varfor} and~\ref{surfacetransition}}\label{prodr}
	\subsection{Proof of Theorem~\ref{varfor}}
	We will actually restrict the partition function to beads during the proof and show in this section that the variational formula written in \eqref{eq:gbetadelta-varfor} is the limit of $(1/L) \log \tilde Z^{\,o}_{L,\beta,\delta}$ as $L\to \infty$ instead of~\eqref{second}. This change is actually harmless, since both the restricted and unrestricted partition functions have the same surface free energy, as we shall establish in Theorem~\ref{asymptotics}. Pick $(\beta,\delta) \in \cC_{\rm{good}}$ and recall the definition of $\penalite$ in \eqref{eq:def_penalite}.
	By Corollary~\ref{cor:unique_max_in_Cgood}, the maximum of $T_\delta$ is unique so that we may set 
	\begin{equation}
		\label{defbara}
		\bar a_{\beta,\delta}:=\argmax\{T_\delta(a), \  a\in (0,\infty)\},
	\end{equation}
	and we write $\bar a$ instead of $\bar a_{\beta,\delta}$ when there is no risk of confusion. Let us now recall Lemma~\ref{compact} and point out that we may always enlarge the width of the interval $[a_1,a_2]$ given therein, if needed, so that \eqref{eq:compact} holds with $\bar a\in (a_1,a_2)$.
	At this stage, we let  $\bar a_L$ be the closest point of  $\bar a$ in $\frac{\N}{\sqrt{L}}$. Therefore $| \bar a_L-\bar a|\leq \frac{1}{\sqrt{L}}$ and there exists $L_0\in \N$
	such that $\bar a_L\in (a_1,a_2)$ for every $L\geq L_0$. We proceed in two steps.\\
	\par (I) Let us start with the upper bound. We use \eqref{uniqbead} to state that
	\begin{equation}
		\limsup_{L\to \infty}  \frac{1}{\sqrt{L}} \log \tilde Z^{\,o}_{L,\beta,\delta}
		=\limsup_{L\to \infty}  \frac{1}{\sqrt{L}} \log \sum_{a\in [a_1,a_2]\cap (\N/\sqrt{L})} 
		\Gamma_\beta^{a\sqrt{L}} D_{a\sqrt{L}}( q(a,L) , \delta),
	\end{equation}
	where, for every $a\in [a_1,a_2]\cap (\N/\sqrt{L})$ we have that $q(a,L)\in [\tfrac{1}{a_1^2},\tfrac{1}{a_2^2}]$
	and $|q(a,L)-\frac{1}{a^2}|\leq 1/(a_2\sqrt{L})$.
	Then, by Lemma~\ref{rough-bounds},
	\begin{equation}\label{restr}
		\limsup_{L\to \infty}  \frac{1}{\sqrt{L}} \log \tilde Z^{\,o}_{L,\beta,\delta} \leq \limsup_{L\to \infty}  \frac{1}{\sqrt{L}} \log \sum_{a\in [a_1,a_2]\cap (\N/\sqrt{L})} 
		\Gamma_\beta^{a\sqrt{L}} e^{a\sqrt{L} \, \psi( q(a,L) , \delta)}.
	\end{equation}
	We recall from Lemma~\ref{regulpsitilde} that $q \mapsto \psi(q,\delta)$ is $\sC^1$ and therefore Lipshitz on $[q_1,q_2] := [1/a_1^2, 1/a_2^2]$. Thus, there exists $c>0$ such that for every $L\in \N$
	and $a\in [a_1,a_2]\cap (\N/\sqrt{L})$  we have 
	\begin{equation}\label{psiqmia}
		|\psi( q(a,L), \delta)-\psi( \tfrac{1}{a^2}, \delta) |\leq \frac{c}{\sqrt{L}}.
	\end{equation}
	Thus, \eqref{restr} becomes
	\begin{equation}\label{restreinte}
		\limsup_{L\to \infty}  \frac{1}{\sqrt{L}} \log \tilde Z^{\,o}_{L,\beta,\delta}
		\leq \limsup_{L\to \infty}  \frac{1}{\sqrt{L}} \log \sum_{a\in [a_1,a_2]\cap (\N/\sqrt{L})} 
		e^{\sqrt{L} T_\delta(a)}.
	\end{equation} 
	Recalling~\eqref{defbara}, we obtain for $L$ large enough,
	\begin{align}\label{supoverrest}
		\sum_{a\in [a_1,a_2]\cap (\N/\sqrt{L})} 
		e^{\sqrt{L} T_\delta(a)}\leq (a_2-a_1)\sqrt{L} e^{\sqrt{L} \big[ \bar a \log \Gamma_\beta+ \bar a  \psi(\tfrac{1}{\bar a^2}, \delta)\big]}.
	\end{align}
	It remains to combine \eqref{restreinte} with \eqref{supoverrest}  to  assert that
	\begin{align}\label{limsupfin}
		\limsup_{L\to \infty}  \frac{1}{\sqrt{L}} \log \tilde Z^{\,o}_{L,\beta,\delta} \leq  \bar a  \log \Gamma_\beta+ \bar a  \psi(\tfrac{1}{\bar a^2}, \delta),
	\end{align}
	which completes the proof of the upper bound.\\
	\par (II) It remains to prove the lower bound. We first consider the case $\delta\neq  \delta_0(1/\bar a^2)={\beta}/{2}-{\tilde h(1/\bar a^2)}/{2}$.
	More precisely we will focus on the case $\delta>\delta_0(1/\bar a^2)$ since the case $\delta<\delta_0(1/\bar a^2)$ is dealt with in a similar manner.
	We recall \eqref{uniqbead}
	and we restrict the sum to $a=\bar a_L $ 
	such that 
	\begin{align}\label{lowb}
		\liminf_{L\to \infty}  \frac{1}{\sqrt{L}} \log  \widetilde Z^{\, \text{o}}_{L,\beta,\delta}\geq \liminf_{L\to \infty}  \frac{1}{\sqrt{L}} \log  \Gamma_\beta^{\, \bar a_L \sqrt{L}}  \ D_{ \bar a_L  \sqrt{L}} \big(q(\bar a_L,L), \delta\big).
	\end{align}
	Since $q\mapsto \htildeq$ is continuous, we can assert that there exists $\gep>0$ such that $\delta> \delta_0(q) :=\frac{\beta}{2}-\frac{\tilde h(q)}{2}$ for every $q\in (\frac{1}{\bar a^2}-\gep,\frac{1}{\bar a^2}+\gep)$.
	For $L$ large enough, it comes straightforwardly that $|q(\bar a_L,L)-1/\bar a^2|<\gep$, and therefore, thanks to Lemma \ref{Nantes Lemme 5.1}, we may apply Proposition \ref{asympteach3reg}, Case (3) to assert that there exists $c>0$ such that 
	for $L$ large enough
	\begin{align}\label{lowba}
		\widetilde Z^{\,\text{o}}_{L,\beta,\delta}\geq \frac{c}{L^{3/4}}\ \Gamma_\beta^{\, \bar a_L \sqrt{L}}  \ e^{\bar a_L  \sqrt{L}\  \psi\big( q(\bar a_L,L), \delta\big)}.
	\end{align}
	We take the logarithm on both sides in \eqref{lowba}, divide by $\sqrt{L}$ and use the continuity of $q\mapsto \psi(q,\delta)$ together with the fact that $\lim_{L\to \infty} q(\bar a_L,L)=1/\bar a^2$  to deduce that
	\begin{align}
		\liminf_{L\to \infty} \frac{1}{\sqrt{L}} \log \tilde Z^{\,o}_{L,\beta,\delta} &\geq   \bar a \log \Gamma_\beta+ \bar a  \psi(\tfrac{1}{\bar a^2}, \delta).
	\end{align}
	This completes the proof of the lower bound in the case $\delta\neq   \delta_0(1/\bar a^2)$.\\ 
	\par It remains to obtain the lower bound in the case   $\delta= \delta_0(1/\bar a^2)$. By monotonicity in $\delta$ and using the above case, we can write for every $n\in \N$ that
	\begin{align}\label{liminfprepa}
		\liminf_{L\to \infty} \frac{1}{\sqrt{L}} \log \tilde Z^{\,o}_{L,\beta,\delta} &\geq  \liminf_{L\to \infty} \frac{1}{\sqrt{L}} \log \tilde Z^{\,o}_{L,\beta,\delta-\frac{1}{n}}\\
		\nonumber &=\max\big\{T_{\delta-\frac{1}{n}}(a),a>0\big\}\geq  T_{\delta-\frac{1}{n}}(\bar a_{\beta,\delta}).
	\end{align}
	It remains to prove that, at $x>0$ fixed, $\delta \in (0,\beta) \rightarrow T_\delta(x) $ is continuous to assert that $\lim_{n\to \infty} T_{\delta-\frac{1}{n}}(\bar a)=T_\delta(\bar a)$. Indeed,	by \eqref{eq:def_penalite}, one can see that $\delta \in (0,\beta) \rightarrow T_\delta(x)$ is continuous if and only if $\delta \in (0,\beta ) \rightarrow \psi(q,\delta)$ is continuous. Recall the definition of $\psi$ in \eqref{deftildepsi}. One can see that $\psi$ is continuous when $\delta \leq \delta_0(q)$, as it is equal to $\psi(q,0)$. When $\delta> \delta_0(q)$, $\delta \rightarrow \cH_\delta(x)$ is continuous, see \eqref{Nantes 44.14}, and $\delta \rightarrow s_\delta(q)$ is continuous by Lemma \ref{lem:prop-Hgd-Hngd} (recall that $s_\delta = (\cH_\delta')^{-1}$). Finally, $\delta \in (0,\beta) \rightarrow \psi(q,\delta)$ is also continuous at $\delta = \delta_0(q)$ thanks to Lemma~\ref{inegcentral}. This completes the proof of Theorem~\ref{varfor}.
	\subsection{Proof of Theorem~\ref{surfacetransition}}
	(i) We start by proving~\eqref{closedexprdeltac} via upper and lower bounds. Recall the definitions of $\gd_c(\beta)$ and $T_\delta$ in~\eqref{exprcourbecrit} and~\eqref{eq:def_penalite}, and assume that $\delta>\delta_0(1/a_{\beta}^2)= (\beta - \tilde{h}(1/a_\beta^{2}))/{2}$. Then,
	Lemma \ref{inegcentral} guarantees that $\psi(1/a_\beta^2,\delta)>\psi(1/a_\beta^2,0)$.
	Therefore, by Theorem~\ref{varfor}, $g(\beta,\delta)\geq T_\delta(a_\beta)>T_0(a_\beta)=g(\beta,0)$, implying that
	\begin{equation}\label{uppbounddel}
		\delta_c(\beta)\leq \frac{\beta}{2} - \frac{\tilde{h}(1/a_\beta^{2})}{2}.
	\end{equation}
	Let us now assume by contradiction that this inequality is strict, i.e.\ there exists $0<\delta_0<\delta_0(1/a_{\beta}^2)$ such that $g(\beta,\delta_0)>g(\beta,0)$. By~\eqref{deftildepsi} we may claim that $T_{\delta_0}(a_\beta)=T_0(a_\beta)$. Moreover, since $g(\beta,\delta_0)>g(\beta,0)$, Theorem~\ref{varfor} yields that
	there exists $a_0>0$ such that $a_0\neq a_\beta$ and 
	$T_{\delta_0}(a_0)>T_0(a_\beta)=T_{\delta_0}(a_\beta)$. Assume that $a_0<a_\beta$ (the proof is similar otherwise). Since $\gd_0 \le \gb/2$, Lemma~\ref{concavity} yields that $T_{\delta_0}$ is strictly concave and we obtain, for every $\gep>0$
	\begin{align}\label{usestct}
		0&>\frac{T_{\delta_0}(a_\beta)-T_{\delta_0}(a_0)}{a_\beta-a_0}>\frac{T_{\delta_0}(a_\beta+\gep)-T_{\delta_0}(a_\beta)}{\gep}\geq \frac{T_{0}(a_\beta+\gep)-T_{0}(a_\beta)}{\gep},
	\end{align}
	where, for the last inequality, we have used that 
	$\delta\mapsto T_\delta(a)$ is non-decreasing for every $a>0$. 
	It remains to let $\gep\to 0$ in the r.h.s. of \eqref{usestct} to obtain, on the one hand,
	\begin{equation}\label{inegdert}
		0>\lim_{\gep\to 0^+}  \frac{T_{0}(a_\beta+\gep)-T_{0}(a_\beta)}{\gep}=(T_0)^{'}(a_\beta).
	\end{equation}
	On the other hand, $(T_0)^{'}(a_\beta)=0$ since,  by Lemmas \ref{concavity} and \ref{lem:limits}, $T_0$ is $\sC^1$ and reaches its maximum on $(0,\infty)$ at $a_\beta$. We finally get the contradiction, which proves that the inequality in~\eqref{uppbounddel} is actually an equality.\\
	%
	(ii) Let us now prove~\eqref{eq:delta_c_explicit_exp} and \eqref{eq:delta_c_low_temp_asympt}. Using Remark \ref{Nantes Remark 4.13} and \eqref{Nantes A.1}, one can see that:
		\begin{equation}
			-\log \Gamma_\beta = \lmgf(\mathsf{h}/2) = \log \left(\frac{1}{c_\beta} \left( \frac{1}{1-e^{(\mathsf{h}-\beta)/2}}  +  \frac{1}{1-e^{(-\mathsf{h}-\beta)/2}} - 1\right) \right).
			\label{Poitiers 18}
		\end{equation}
		By \eqref{eq:defGammaBeta} and \eqref{closedexprdeltac}, we obtain:
		\beq
		e^{\gb} + 1 = \frac{1}{1-e^{-\gd_c(\gb)}} + \frac{1}{1- e^{\gd_c(\gb)}e^{-\gb}}.
		\eeq
		Letting $u = \exp(\gb/2)>1$ and $\mathsf{X} = \exp(\gd_c(\gb))$, we are left to solve
		\beq
		u^2 + 1 = \frac{\mathsf{X}}{\mathsf{X}-1} + \frac{u^2}{u^2 - \mathsf{X}},
		\eeq
		that is
		\beq
		\mathsf{X}^2 - \Big(u^2 + \frac{1}{u^2}\Big) \mathsf{X} + u^2= 0,
		\qquad \Big(\text{or} \quad
		\mathsf{X}^2 - 2\cosh(\gb) \mathsf{X} + e^{\gb}= 0\Big),
		\eeq
		for which we compute
		\beq
		\gD(u) := \Big(u^2 + \frac{1}{u^2}-2u\Big)\Big(u^2 + \frac{1}{u^2}+2u\Big).
		\eeq
		It turns out that
		\beq
		u^4 \gD(u) = (u-1)(u^3 - u^2 - u -1)(u^4+2u^3 +1) >0
		\eeq
		as soon as $\gb > \gb_c$, see~\cite[p.19]{legra2021}. Therefore,
		\beq
		\mathsf{X} = \cosh(\gb) \pm \sqrt{\cosh(\gb)^2-e^{\gb}}.
		\eeq
		Since $\gd_c(\gb) \le \gb/2$, we readily obtain \eqref{eq:delta_c_explicit_exp} and \eqref{eq:delta_c_low_temp_asympt}.\\
	(iii) The proof of~\eqref{secorder} (second-order transition) is quite computational, hence its proof is postponed to Appendix \ref{Nantes Appendix C.2}.\\
	
	\par Let us end this section with a remark. Some of the observations made during the proof of Item (ii) in Theorem~\ref{surfacetransition} lead to the following:
		\begin{proposition} When $\beta$ goes to infinity, $q_\beta=  1 + O(\beta^{-2})$.
			\label{Nantes Lemma 8.3}
		\end{proposition}
		This implies that the horizontal extension of the polymer, after renormalization by $\sqrt{L}$, converges to one for the model without the attractive wall, in the large $\gb$-limit. In other words, the associated Wulff shape looks more and more like a square.
		\begin{proof}[Proof of Proposition~\ref{Nantes Lemma 8.3}]
			Let us denote $\mathsf{h} := \tilde h(q_\beta)$ in this proof. By Remark \ref{htildeqincreasing} and the lines below, $\mathsf{h}$ is defined by
			\begin{equation}
				q_\beta = \int_0^1 x \lmgf' \Big( \mathsf{h} \Big( x - \frac{1}{2} \Big) \Big) \dd x.
				\label{Nantes 8.45}
			\end{equation}
			An integration by part gives:
			\begin{equation}
				q_\beta = \frac{1}{\mathsf{h}} \lmgf \Big( \frac{\mathsf{h}}{2} \Big) - \frac{1}{\mathsf{h}} \int_0^1  \lmgf \Big( \mathsf{h} \Big( x - \frac{1}{2} \Big) \Big) \dd x.
			\end{equation}
			We consider the first and second terms separately. By \eqref{closedexprdeltac} and \eqref{eq:delta_c_explicit_exp}, we first obtain
			\begin{equation}\label{pr:asympt-h-c}
				\mathsf{h} = \beta - 2e^{-\beta}[1+O(e^{-\beta})],
				\qquad \text{as } \gb\to\infty.
			\end{equation}
			Using \eqref{Poitiers 18}, \eqref{pr:asympt-h-c} and the fact that $-\log \Gamma_\beta  = \beta + O(e^{-\beta/2})$, we have, for the first term,
			\begin{equation}
				\frac{1}{\mathsf{h}} \lmgf \Big( \frac{\mathsf{h}}{2} \Big) = 1 + O(e^{-\beta/2}).
				\label{Poitiers 19}
			\end{equation}
			Letting $a := [1-e^{(\mathsf{h}-\beta)/2}]^{-1}>0$ and $b := [1-e^{(-\mathsf{h}-\beta)/2}]^{-1}-1>0$ in~\eqref{Poitiers 18}, and using that $\log(a) \leq \log(a+b) \leq \log(a) + b/a$, one has:
			\begin{equation}
				\log \left( \frac{1}{1-e^{(\mathsf{h}-\beta)/2}}  +  \frac{1}{1-e^{(-\mathsf{h}-\beta)/2}} - 1\right) = - \log(1-e^{(\mathsf{h}-\beta)/2}) + O(e^{-\beta/2}).
				\label{Poitiers 14}
			\end{equation}
			Using \eqref{Poitiers 14}, the parity of $\lmgf$, and \eqref{pr:asympt-h-c}, we obtain, for the second term,
			\begin{equation}
				\begin{aligned}
					\frac{1}{\mathsf{h}} \int_0^1  \lmgf \Big( \mathsf{h} \Big( x - \frac{1}{2} \Big) \Big) \dd x &= \frac{2}{\mathsf{h}} \int_{1/2}^1  \lmgf \Big( \mathsf{h} \Big( x - \frac{1}{2} \Big) \Big) \dd x \\
					&= \frac{2 + O(e^{-\beta/2})}{\beta} \int_{1/2}^1 - \log(1-e^{\mathsf{h} (x-1/2) -\beta/2}) \dd x.
					\label{Poitiers 20}
				\end{aligned}
			\end{equation}
			Letting $u := e^{\mathsf{h}(x-1/2)  - \beta/2}$, we now write that
			\begin{equation}
				\begin{aligned}
					-\int_{1/2}^1 \log \Big(
					{1-e^{\mathsf{h}(x-1/2)  - \beta/2}} \Big) \dd x &= 
					-\frac{1}{\mathsf{h}}\int_{e^{-\beta/2}}^{e^{(\mathsf{h}-\beta)/2}} \frac{\log(1-u)}{u} \dd u \\&\leq -\frac{1}{\mathsf{h}}\int_{0}^{1} \frac{\log(1-u)}{u} \dd u 
					= \frac{\pi^2}{6\mathsf{h}}.
				\end{aligned}
				\label{Poitiers 21}
			\end{equation}
			Combining \eqref{Nantes 8.45}, \eqref{Poitiers 19}, \eqref{Poitiers 20} and \eqref{Poitiers 21}, we finally obtain:
			\begin{equation}
				q_\beta = 1 + O(\beta^{-2}).
			\end{equation}
		\end{proof}
	\section{Proof of Theorem~\ref{asymptotics}} \label{proofasympt}
	In order to obtain the asymptotics of the sequence of partition functions $(Z_{L,\beta,\delta})_{L\in \N}$, we will use three mains tools:
	\begin{itemize}
		\item the bead decomposition of the partition function introduced in Section \ref{beaddec};
		\item Proposition~\ref{asymptoticsinter} stated below and proved in Section~\ref{asymptoneead} that provides us with the 
		asymptotics of the partition function
		associated with the very first bead;
		\item Proposition~\ref{essympwithoutdelta} that provides the asymptotics of the partition function associated with the beads that cannot touch the wall.
	\end{itemize}
	\begin{proposition}\label{asymptoticsinter}
		For $\beta>\beta_c$, we have in each of the three regimes, as $L\to \infty$:
		\begin{enumerate}
			\item If $\delta<\delta_c(\beta)$ then there exists a positive constant $C^-_{\beta,\delta}$ such that
			\begin{equation}\label{parfunqsouscritinter} 
				\bar Z^{\, \text{o}}_{L,\beta,\delta} \simL \frac{C^-_{\beta,\delta}}{L^{3/4}} \, e^{\beta L+g(\beta, 0)\sqrt{L}}.
			\end{equation}
			\item  If $\delta=\delta_c(\beta)$ then there exists a positive constant $C^{\rm{crit}}_{\beta,\delta}$ such that
			\begin{equation} \label{parfunqcritinter}
				\bar Z^{\, \text{o}}_{L,\beta,\delta} \simL \frac{C^{\rm{crit}}_{\beta,\delta}}{\sqrt{L}} \, e^{\beta L+g(\beta,\delta) \sqrt{L} }.
			\end{equation}
			\item  If $\delta>\delta_c(\beta)$ and $(\beta,\delta) \in \cC_{\rm good}$ then there exists a positive constant $C^+_{\beta,\delta}$ such that
			\begin{equation}\label{parfuneqsupcritinter}
				\bar Z^{\, \text{o}}_{L,\beta,\delta} \simL\frac{C^+_{\beta,\delta}}{\sqrt{L}} \, e^{\beta L+g(\beta,\delta) \sqrt{L} }.
			\end{equation}
		\end{enumerate}
	\end{proposition}
	Let us now prove Theorem \ref{asymptotics} subject to Proposition \ref{asymptoticsinter}.
	\subsection{Proof of \texorpdfstring{\eqref{parfuneqsupcrit}}{}: Supercritical case}
	Let	$(\beta,\delta)\in \cC_{\rm good}$ and  $\delta>\delta_c(\beta)$. We define 
	\beq
	\zeta_\beta :=  \rm argcosh (e^{-\beta/2} \cosh(\beta))
	\eeq
	and introduce a probability measure on $\N$:
	\begin{equation}
		\mu_2(n) := C_0^{-1} \hat Z_{n,\beta}^\circ e^{-\beta n},\qquad \text{ with } C_0 = \Big( 1+\frac{2e^{-\beta}}{1-e^{-\beta}}\Big) \big(e^\beta - 1 - e^{\zeta_\beta + \beta/2} \big).
		\label{Nantes 6.4}
	\end{equation}
	It is indeed a probability thanks to~\cite[(4.8)]{Legrand_2022} and $C_0<1$ for every $\gb>\gb_c$~\cite[Corollary 3.3]{Legrand_2022}. We also state a lemma that will be proven at the end of this section:
	\begin{lemma} \label{Nantes Lemma 6.2} For $(\beta,\delta) \in \cC_{\rm{good}}$
		\begin{equation}
			R(\beta,\delta) := \somme{L \geq 2}{} \bar Z^\circ_{L,\beta,\delta} e^{-\beta L} = \bar K_{\beta,\delta} + \frac{e^{-\beta}}{1-e^{-\beta}}\bar K_{\beta,0}, 
			\label{St Flour E.1}
		\end{equation}
		with
		\begin{equation}
			\bar K_{\beta,\delta} = \begin{cases}
				\frac{2(e^{\gd - \beta/2 - \zeta_\beta}-e^{\gd-\gb})}{1-e^{\delta - \beta/2- \zeta_\beta}}
				& \text{ when } \delta <  \zeta_\beta + \beta/2 \\
				+\infty &\text{otherwise}. 
			\end{cases}
		\end{equation}
	\end{lemma}
	Hence, we introduce another probability measure, when $\delta < \zeta_\beta + \beta/2$:
	\begin{equation}
		\mu_3(n) :=R(\beta,\delta)^{-1} \bar Z_{n,\beta,\gd}^\circ e^{-\beta n}.
	\end{equation}
	We now prove that:
	\begin{equation}
		Z_{L,\beta,\delta}
		\leq
		e^{\beta L + g(\beta,\delta) \sqrt{L}} \frac{C_{\beta,\delta}^{+}}{ \sqrt{L}(1-C_0)\big(1-e^{-\beta}\big)}(1+o(1))
		\label{Nantes E.36}
	\end{equation}
	and that 
	\begin{equation}
		Z_{L,\beta,\delta}
		\geq
		e^{\beta L + g(\beta,\delta) \sqrt{L}} \frac{C_{\beta,\delta}^{+}}{ \sqrt{L}(1-C_0)\big(1-e^{-\beta}\big)}(1+o(1)),
		\label{Nantes E.39}
	\end{equation}
	the combination of which settles~\eqref{parfuneqsupcrit}.
	Beforehand,	we state a useful inequality: using the sub-exponential asymptotics of $\bar Z^\circ_{L,\beta,\delta}e^{-\gb L}$ in \eqref{parfuneqsupcritinter}, there exists $\varepsilon_1 : \N \rightarrow \R_+$, $\varepsilon_2 : \N \rightarrow \R_+$ and $(M_L)_{L \in \N}$ a sequence of integer such that $M_L \rightarrow \infty$ with $M_L = o(L)$, $\varepsilon_1(L), \varepsilon_2(L) \rightarrow 0$, and 
	\begin{equation} \label{Nantes 6.10}
		\forall \, k \in [0,M_L],\quad (1-\varepsilon_1(L)) \bar Z^\circ_{L,\beta,\delta}e^{-\gb L} \leq  \bar Z^\circ_{L-k,\beta,\delta}e^{-\gb(L-k)} \leq (1-\varepsilon_2(L)) \bar Z^\circ_{L,\beta,\delta}e^{-\gb L}.
	\end{equation}
	\begin{proof}[Proof of \eqref{Nantes E.36}]
		Since $C_0<1$, the series $\somme{r \geq 0}{}C_0^r \mu_2^{r*}[1,\infty]$ converges. Therefore, for all $\varepsilon > 0$ there exists $K>0$ such that $\somme{r \geq 0}{}C_0^r \mu_2^{r*}[K,\infty] < \varepsilon$.
		We start from \eqref{fullpartfun}. The proof depends on the sign of $g(\beta,\delta)$:\\
		(i)
		If $g(\beta,\delta)>0$ then, by \eqref{parfuneqsupcritinter},  $e^{-\beta t_1}\bar{Z}^\circ_{t_1, \beta,\delta} \sim_{t_1} \frac{C_{\beta,\delta}^+ e^{g(\beta,\delta) \sqrt{t_1}} }{\sqrt{t_1}} $, which is a nondecreasing sequence diverging to $+\infty$. Hence, there exists $\varepsilon : \N \mapsto \R_+$ such that $\varepsilon(t_1) \to 0$ as $t_1 \to + \infty$ and, for all $t_1 \in \mathbb{N}$, $e^{-\beta t_1}\bar{Z}^\circ_{t_1, \beta,\delta} \leq (1+\varepsilon(t_1)) \frac{C_{\beta,\delta}^+ e^{g(\beta,\delta) \sqrt{t_1}} }{\sqrt{t_1}} $ and $\Big\{ (1+\varepsilon(t_1)) \frac{C_{\beta,\delta}^+ e^{g(\beta,\delta) \sqrt{t_1}} }{\sqrt{t_1}}  \Big\}_{t_1 \in \N}$ is a nondecreasing sequence. Therefore,
		\begin{equation}
			\begin{aligned}
				Z_{L,\beta,\delta}^c &= e^{\beta L}  \sum_{r=1}^{L/2} \sum_{t_1 + \cdots + t_r = L} e^{-\beta t_1}\bar{Z}^\circ_{t_1, \beta,\delta} \prod_{j=2}^r \hat{Z}^\circ_{t_j, \beta} e^{-\beta t_j} \\
				&\le  (1+\varepsilon(L))  e^{\beta L + g(\beta,\delta) \sqrt{L}} \frac{C_{\beta,\delta}^{+}}{ \sqrt{L}}  \sum_{r=1}^{L/2} \sum_{t_1 + \cdots + t_r = L} \prod_{j=2}^r \hat{Z}^\circ_{t_j, \beta} e^{-\beta t_j} \\
				&\stackrel{\text{by~\eqref{Nantes 6.4}}}{\leq} 
				(1+\varepsilon(L))  e^{\beta L + g(\beta,\delta) \sqrt{L}} \frac{C_{\beta,\delta}^{+}}{ \sqrt{L}}
				\Big(1+	\somme{r \geq 1}{} C_0^r \mu_2^{r*}([1,L])\Big). 
			\end{aligned}
		\end{equation}
		To conclude, one can observe that
		\begin{equation}
			Z_{L,\beta,\delta}e^{-\gb L} = \somme{k=0}{L}e^{-\beta k}  (e^{-\beta(L-k)}Z_{L-k,\beta,\delta}^c),
			\label{Nantes 6.9}
		\end{equation}
		where $k$ is the number of zero-length stretches at the end of the polymer,
		and use dominated convergence.\\
		(ii) If $g(\beta,\delta)<0$, using Lemma \ref{Nantes Lemma 6.2}, one has :
		\begin{equation}
			\begin{aligned}
				Z_{L,\beta,\delta}^c &= e^{\beta L} R(\beta,\delta) \somme{r \geq 0}{} C_0^r (\mu_3 * \mu_2^{r*})(L).
			\end{aligned}
		\end{equation}
		To compute this sum, we use \cite[Corollary 4.13 and Theorem 4.14]{FKZ11} that stands the two following claims:
		\begin{claim} For $\beta>0$, $r \geq 0$ and $\delta \in [\delta_c(\beta),\zeta_\beta + \delta/2)$, it holds that $\mu_3 * \mu_2^{r*}(n) \sim_n \frac{e^{g(\beta,\delta) \sqrt{n}}}{\sqrt{n}}$.
		\end{claim}
		\begin{claim}  For $\beta>0$, $\varepsilon > 0$ and $\delta \in [\delta_c(\beta),\zeta_\beta + \delta/2)$, there exists $n_0(\varepsilon) \in \N$ and $C(\varepsilon) > 0$ such that
			\begin{equation}
				\mu_3*\mu_2^{r*}(n) \leq C(\varepsilon) (1+\varepsilon)^r \frac{e^{g(\beta,\delta) \sqrt{n}}}{\sqrt{n}}, \qquad n \geq n_0(\varepsilon),\ r \in \N \cup \{0\}.
			\end{equation}	
		\end{claim}
		\noindent Dominated convergence gives
		\begin{equation}
			Z_{L,\beta,\delta}^c = e^{\beta L + g(\beta,\delta)\sqrt{L}} \frac{C_{\beta,\delta}^+}{\sqrt{L}(1-C_0)}.
		\end{equation}
		Equation \eqref{Nantes 6.9} concludes.\\
		(iii) When $g(\beta,\delta) = 0$, we  decompose the partition function according to the extended beads and split the sum according to whether the volume of the first bead is smaller or greater than $L-\sqrt{L}$:
		\begin{equation}
			\begin{aligned}
				Z_{L,\beta,\delta}^c &= e^{\beta L}  \sum_{r=1}^{L/2} \sum_{t_1 + \cdots + t_r = L} e^{-\beta t_1}\bar{Z}^\circ_{t_1, \beta,\delta} \prod_{j=2}^r \hat{Z}^\circ_{t_j, \beta} e^{-\beta t_j} \\
				&\leq 
				e^{\beta L} \frac{C_{\beta,\delta}^{+}}{ \sqrt{L}}(1+o(1))
				\somme{r \geq 0}{} C_0^r \mu_2^{r*}([1,\sqrt{L}]) + e^{\beta L} 
				\sum_{r=1}^{L/2} \somme{t_1=1}{L-\sqrt{L}} \sum_{t_2 + \cdots + t_r = L- t_1} \prod_{j=2}^r \hat{Z}^\circ_{t_j, \beta} e^{-\beta t_j} \\
				&\leq e^{\beta L} \frac{C_{\beta,\delta}^{+}}{ \sqrt{L}(1-C_0)}(1+o(1)) + C L e^{\beta L + g(\beta,0) L^{1/4}} \\
				&\leq e^{\beta L} \frac{C_{\beta,\delta}^{+}}{ \sqrt{L}(1-C_0)}(1+o(1)),
			\end{aligned}
		\end{equation}
		having used that $g(\beta,0) < 0$.\\
		\par This completes the proof of \eqref{Nantes E.36}.
	\end{proof}
	\begin{proof}[Proof of \eqref{Nantes E.39}]
		Recalling~\eqref{Nantes 6.10} and restricting the sum in \eqref{fullpartfun} to $t_1\ge L-M_L$, we obtain :
		\begin{equation}
			Z_{L,\beta,\delta}^c \geq e^{\beta L} Z^\circ_{L,\beta,\delta} (1-\varepsilon_1(L)) \sum_{r=0}^{\infty} C_0^r \mu_2^{r*}([1,M_L]) \geq \frac{ e^{\beta L} Z^\circ_{L,\beta,\delta} (1-\varepsilon_1(L))}{1-C_0}(1+o(1)).
		\end{equation}
		To conclude, one can use \eqref{Nantes 6.9}.
	\end{proof}
	Using \eqref{Nantes E.36} and \eqref{Nantes E.39}, we therefore have, with $C_0$ defined in \eqref{Nantes 6.4} and $C_{\beta,\delta}^{+}$ in \eqref{parfuneqsupcritinter}:
	\begin{equation}
		\Cover = \frac{C_{\beta,\delta}^{+}}{(1-C_0)\big(1-e^{-\beta}\big)}.
	\end{equation}
	\begin{proof}[Proof of Lemma \ref{Nantes Lemma 6.2}] We take large inspiration from the proof of \cite[Lemma 3.2]{Legrand_2022}. Recalling that 
		\begin{equation}
			Z^\circ_{L,\beta,\delta} e^{-\beta L}=  \somme{N=1}{L/2}\Gamma_\beta^N \EsperanceMarcheAleatoire{e^{(\delta - \beta/2 )X_N} 1\{ X_{[1,N]}>0, A_N = L-N \} },
		\end{equation}
		a computation gives:    
		\begin{equation}
			\begin{aligned}
				\somme{L=2}{\infty} Z_{L,\beta,\delta}^\circ e^{-\beta L} &= 2\somme{N \geq 1}{} \Gamma_\beta^N \somme{ L \geq 2N-1}{} \somme{k \in \N}{} e^{(\delta -\beta/2) k} \Proba{\beta}{X_{[1,N-1]}>0, X_N = k, A_N = L-N} \\
				&= 2\somme{ k\ge 1}{} e^{(\delta -  \beta/2)k} \somme{N \geq 1}{} \Gamma_\beta^N \Proba{\beta}{X_N = k, X_{[1,N-1]} > 0} \\
				&= 2\somme{ k\ge 1}{} e^{(\delta -  \beta/2)k} \somme{N \geq 1}{} \Gamma_\beta^N \Proba{\beta}{X_N = -k, X_{[1,N-1]} > -k},
			\end{aligned}
		\end{equation}
		having used the time-reversal property for the last equality. Defining $\rho_k = \inf \{ i \geq 1: X_i \leq -k \}$, it comes:
		\begin{equation}
			\somme{L=2}{\infty} Z_{L,\beta,\delta}^\circ e^{-\beta L} = 2\somme{ k\ge 1}{} e^{(\delta - \beta/2)k} \EsperanceMarcheAleatoire{\Gamma_\beta^{\rho_k} 1\{ X_{\rho_k = -k} \}}.
		\end{equation}
		We denote $r_{\beta,k} := \EsperanceMarcheAleatoire{\Gamma_\beta^{\rho_k} 1\{ X_{\rho_k = -k} \}}$. We now compute $\EsperanceMarcheAleatoire{\Gamma_\beta^{\rho_k}}$, which will lead us to have an exact expression of $r_{\beta,k}$. First, we remark that, because the increments of $X$ follow discrete Laplace law, $(\rho_k,X_1,...,X_{\rho_{k}-1})$ and $X_{\rho_k}$ are independent, and $-X_{\rho_k} = k + \cG(1-e^{-\beta/2})$, with $\cG(.)$ a geometric law over $\N \cup \{0\}$. Reminding that $\Gamma_\beta = c_\beta/e^\beta$:
		\begin{equation}
			\begin{aligned}
				\EsperanceMarcheAleatoire{\Gamma_\beta^{\rho_k}} &=
				\frac{r_{\beta,k}}{\Proba{\beta}{X_{\rho_k} = -k}} = \frac{r_{\beta,k}}{1-e^{-\beta/2}}.
			\end{aligned}
		\end{equation}
		Thanks to \cite[(3.23)]{Legrand_2022}, $(e^{-\zeta_ \beta X_n + \log(\Gamma_\beta)n})_{n \in \N}$ is a martingale. A stopping-time argument therefore gives:
		\begin{equation}
			\EsperanceMarcheAleatoire{\Gamma_\beta^{\rho_k}} = \EsperanceMarcheAleatoire{e^{- \zeta_\beta X_{\rho_k}}}^{-1} = e^{-k \zeta_\beta} \frac{1-e^{\zeta_\beta - \beta/2}}{1-e^{-\beta/2}}.
		\end{equation}
		Hence, $r_{\beta,k} = e^{-k \zeta_\beta}(1-e^{\zeta_\beta - \beta/2})$. Note that  the polymer does not interact with the wall if the first stretch is zero. Hence,  using~\eqref{fbead} at the first line and a change of variable at the second line:

		\begin{equation}
			\begin{aligned}
				\somme{L \geq 2}{} \bar Z^\circ_{L,\beta,\delta} e^{-\beta L} &= \somme{L\ge2}{} e^{-\beta L} Z_{L,\beta,\delta}^\circ + \somme{L\ge2}{} \somme{k=1}{L-2} e^{-\beta L}  Z_{L-k,\beta,0}^\circ  \\
				&=  \somme{L\ge2}{} e^{-\beta L} Z_{L,\beta,\delta}^\circ  
				+\frac{ e^{-\beta}}{1-e^{-\beta}}  \somme{L\ge2}{} e^{-\beta L} Z_{L,\beta,0}^\circ  .
			\end{aligned}
		\end{equation} 
		A geometric sum gives \eqref{St Flour E.1}. 
	\end{proof}
	\subsection{Proof of \texorpdfstring{\eqref{parfunqcrit}}{}: Critical case}
	To prove \eqref{parfunqcrit}, one can use the exact same ideas and nothing changes much. Henceforth,
	\begin{equation}
		\Ccrit = \frac{C_{\beta,\delta}^{\text{crit}}}{(1-C_0)\big(1-e^{-\beta}\big)},
	\end{equation}
	with $C_0$ defined in \eqref{Nantes 6.4} and $C_{\beta,\delta}^{\text{crit}}$ in \eqref{Nantes 6.38}.
	\subsection{Proof of \texorpdfstring{\eqref{parfunqsouscrit}}{}: Subcritical case}
	We have to change our strategy to prove \eqref{parfunqsouscrit}. Indeed, in this case, the contribution from the first bead does not dominate the total partition function.
	We start with computation of $\bar Z^\circ_{L,\beta,\delta}$:
	\begin{lemma} With $C_{\beta,\delta}^-$ defined in \eqref{Nantes 6.70},
		\begin{equation}
			\bar Z^\circ_{L,\beta,\delta} \underset{L \longrightarrow \infty}{\sim} \frac{\bar K_{\beta,\delta}}{L^{3/4}} e^{\beta L + g(0,\beta) \sqrt{L}},
		\end{equation}
		with $\bar K_{\beta,\delta}=C_{\beta,\delta}^- + e^{\beta}(1-e^{-\beta})^{-1}C_{\beta,0} $.
		\label{lem:sharp-asympt-bead-minus}
	\end{lemma}
	\begin{proof}[Proof of Lemma~\ref{lem:sharp-asympt-bead-minus}]
		The proof is identical to the one from \cite[Corollary 4.2]{Legrand_2022}. To remind it briefly, we denote $h(L) := L^{-3/4} e^{\sqrt{L} g(0,\beta)}$. Noticing that $h(L)\sim h(L-k)$ as $L\to\infty$, dominated convergence implies that 
		\begin{equation}
			\frac{e^{-\beta L}}{h(L)} \bar Z^\circ_{L,\beta,\delta} =  \frac{e^{-\beta L}}{h(L)}Z^\circ_{L,\beta,\delta} + \somme{k=1}{L-2} e^{-\beta k} \frac{e^{-\beta(L-k)}}{h(L)} Z^\circ_{L-k,\beta,0} \stackrel{L\to\infty}{\longrightarrow} C_{\beta,\delta}^- + \frac{e^{\beta}}{1-e^{-\beta}}C_{\beta,0}.
			\label{Nantes E.24}
		\end{equation}
	\end{proof}
	We now move on to the computation of $Z_{L,\beta,\delta}$ . 
	By doing the same as (4.5) to (4.15) in \cite{Legrand_2022}, one can have
	\begin{equation}
		Z_{L,\beta,\delta} = \frac{\Cunder}{L^{3/4}}e^{\beta L + \sqrt{L}g(\beta,\delta)},
	\end{equation}
	with $C_0$ defined in \eqref{Nantes 6.4} and
	\begin{equation}
		\Cunder = \frac{1}{1-e^{-\beta}} \left( \frac{\bar K_{\beta,\delta}}{1-K_2} + \frac{K_1 R(\beta,\delta)}{(1-C_0)^2} \right).
	\end{equation}
	$C_0$ defined in \eqref{Nantes 6.4} and $K_1 := \frac{1+e^{-\beta}}{2(1-e^{-\beta})} C_{\beta,0}^-$.
	\subsection{Proof of Proposition~\ref{asymptoticsinter}.}\label{asymptoneead}
	To prove this proposition, we first work on $\widetilde Z_{L,\beta,\delta}^{ \text{o}}$, that is the partition function of a (simple) bead. It will be then necessary to consider the zero horizontal segments at the beginning of the polymer, which will be addressed in Lemma \ref{Nantes Lemma 6.12}, see Section~\ref{sec:first_stretch} below.
	\par Thanks to Lemma~\ref{compact}, it suffices to consider the partition function restricted to those trajectories with a 
	horizontal extension in $[a_1, a_2]\sqrt L$. The unique maximizer of  $T_\delta$ (see Corollary~\ref{cor:unique_max_in_Cgood} and~\eqref{defbara}) is denoted by $\bar a$ instead of 
	$\bar a_{\beta,\delta}$ in the present proof, for ease of notation. 
	Provided we enlarge a little bit the interval $[a_1,a_2]$ above, we may always assume that $\bar a\in (a_1,a_2)$.
	Let us pick $b,\eta>0$ and set, for $L\in \N$,
	\begin{align}\label{setdenum}
		\nonumber \cT_{L}&:=[a_1,a_2]\cap \frac{\N}{\sqrt{L}},\\
		\nonumber \cS_{\eta,L}&:=\big[\bar a-\eta, \bar a+\eta\big] \cap \frac{\N}{\sqrt{L}},\\
		\cR_{b,L}&:=\big[\bar a-\frac{b}{ L^{1/4}}, \bar a+\frac{b}{ L^{1/4}}\big] \cap \frac{\N}{\sqrt{L}}.
	\end{align}
	The structure of the proof for the supercritical case (Section~\ref{sec:prop6.1-supercrit}), the critical case (Section~\ref{sec:prop6.1-crit}) and the subcritical case (Section~\ref{sec:prop6.1-subcrit}) are the same: we first show in Claim~\ref{claimrest} that the partition function can be restricted to $\cS_{\eta,L}$ for any $\eta > 0$. The proof of this part is common to all three cases. Then, we prove that the partition function can be restricted to $\cR_{b,L}$, which finally enables us to provide the desired sharp asymptotics. Those two parts require specific ideas, which are displayed in the following sections. Throughout the rest of the section we shall use the notation $q_{a,L}$ defined in~\eqref{eq:qaL}. 
	\begin{claim}\label{claimrest}
		For every $\eta>0$, there exists  $\gamma>0$ such that for $L\in \N$,
		\begin{equation}\label{interbound}
			\widetilde Z^{\text{o}}_{L,\beta,\delta}(N_\ell / \sqrt{L} \in \cT_{L}\setminus \cS_{\eta,L})\leq  e^{\sqrt{L} T_{\delta}(\bar a)} e^{-\gamma \sqrt{L}}.
		\end{equation}
	\end{claim}
	\begin{proof}[Proof of Claim \ref{claimrest}]
		We combine Lemma \ref{rough-bounds} with \eqref{psiqmia}  to obtain that there exists $c>0$ such that for every $L\in \N$ and $a\in [a_1,a_2]\cap (\N/\sqrt{L})$ we have  
		\begin{align}\label{roughbound}
			\nonumber D_{a\sqrt{L}}(q_{a,L},\delta)&\leq c e^{a\sqrt{L} \psi(q_{a,L},\delta)}\\
			&\leq  c e^{a\sqrt{L} \psi(\frac{1}{a^2},\delta)}.
		\end{align}
		From~\eqref{partfunbeadaux} and~\eqref{roughbound} we deduce that 
		\begin{equation}
			\widetilde Z_{L,\beta,\delta}^{\, \circ}(N_\ell / \sqrt{L} \in \cT_{L}\setminus \cS_{\eta,L})\leq 2c \sum_{a\in\cT_{L}\setminus \cS_{\eta,L}} e^{\sqrt{L} T_{\delta}(a)}.
		\end{equation}
		By uniqueness of the maximizer of $T_\delta$ on $(0,\infty)$, see~\eqref{defbara}, we have
		$\sup\{T_\delta(a), a\notin [\bar a-\eta,\bar a+\eta]\}<T_\delta(\bar a)$, which completes the proof.
	\end{proof}
	\subsection{Proof of~\texorpdfstring{\eqref{parfuneqsupcritinter}}{}: Supercritical case}
	\label{sec:prop6.1-supercrit}
	Let $(\beta, \delta) \in \cC_{\rm{good}}$. The proof of~\eqref{parfuneqsupcritinter} is a straightforward consequence of Lemma \ref{Nantes Lemma 6.12}, Claim~\ref{claimrest}, and the two following claims.
	\begin{claim}\label{claimrestbig}
		For every $\gep>0$ and $\eta>0$, there exists  $b>0$ such that for $L\in \N$,
		\begin{equation}\label{interbound2}
			\widetilde Z^{\text{o}}_{L,\beta,\delta}(N_\ell / \sqrt{L} \in \cS_{\eta,L}\setminus \cR_{b,L})\leq \frac{\gep}{\sqrt{L}} e^{\sqrt{L} T_{\delta}(\bar a)}.
		\end{equation}
	\end{claim}
	\begin{claim}\label{claimmain}
		There exists $m:(0,\infty)\mapsto (0,1)$ such that $\lim_{b\to \infty} m(b)=1$ and such that  for every $b>0$, 
		\begin{equation}\label{interbound3}
			\widetilde Z^{\text{o}}_{L,\beta,\delta}\big(N_\ell / 
			\sqrt{L} \in \cR_{b,L} \big)= (1+o_{\color{violet}L}(1))\,  m(b) \frac{c_{3,\beta,\delta}}{\sqrt{L}}  e^{\sqrt{L} T_{\delta}(\bar a)}
		\end{equation}
		where $o_{\color{violet} L }(1)$ depends on $b$ and $c_{3,\beta,\delta}:=\kappa^0\big(\tilde h(\bar a^{-2})\big) \xi\big(\bar a^{-2}, \delta \big) e^{\partial_1 \psi(\bar a^{-2},\delta)} \sqrt{ \frac{2 \pi}{-\penalite''(\bar a)  }}$.
	\end{claim}
	\begin{proof}[Proof of Claim \ref{claimrestbig}]
		Since $\delta>\delta_c(\beta)$ we may use Theorem \ref{varfor} and \eqref{exprcourbecrit} to get that 
		\begin{equation}
			T_\delta(\bar a)=\max\{T_\delta(a),\, a>0\}>\max\{T_0(a),\, a>0\}\geq T_0(\bar a).
		\end{equation}
		As a consequence, $\psi(1/\bar a^{2},\delta)>\psi(1/\bar a^{2},0)$ which, with the help of \eqref{deftildepsi}, guarantees that 
		\begin{equation}\label{goodint}
			\delta>\delta_0(1/\bar a^{2}) := \frac{\beta}{2}-\frac{\tilde h(1/\bar a^{2})}{2}.
		\end{equation}
		For a given $\eta>0$,
		$a\in \cS_{\eta,L}$ implies that
		\begin{equation}\label{eq:q_a_l}
			q_{a,L}=\tfrac{1}{a^2}-\tfrac{1}{a\sqrt{L}} \in [(\bar a+\eta)^{-2}-\eta, (\bar a-\eta)^{-2}+\eta],
		\end{equation}
		provided $L$ is chosen large enough.  By continuity of $q\mapsto \tilde h(q)$ and~\eqref{goodint}, we obtain that $\delta> \delta_0(q_{a,L})$ for every $a\in\cS_{\eta,L}$, provided $L$ is large and  $\eta>0$ is small enough. We can therefore use Item (3) in Proposition \ref{asympteach3reg} for every $a\in \cS_{\eta,L}$ to get that 
		there exists $c_1>0$ such that for $L\in \N$,
		\begin{align}\label{decouppar}
			\nonumber \widetilde Z_{L,\beta,\delta}^{\circ}(N_\ell / \sqrt{L} \in  \cS_{\eta,L}\setminus \cR_{b,L})&=\sum_{a\in  \cS_{\eta,L}\setminus \cR_{b,L}}  \, (\Gamma_\beta)^{a\sqrt{L}} D_{a\sqrt{L}}(q_{a,L}, \delta)\\
			&\leq \frac{c_1}{L^{3/4}} \sum_{a\in  \cS_{\eta,L}\setminus \cR_{b,L}}   \,  e^{\sqrt{L} \big[ a \log \Gamma_\beta+  a  \psi(q_{a,L}, \delta)\big]}.
		\end{align}
		At this stage we split the sum in the r.h.s. in \eqref{decouppar} into $A_L+B_L$ where 
		\begin{align}\label{splitab}
			A_L&:= \sum_{a\in [\bar a-\eta,\bar a-\frac{b}{L^{1/4}}]\cap \frac{\N}{\sqrt{L}}}   \,  e^{\sqrt{L} \big[ a \log \Gamma_\beta+  a  \psi(q_{a,L}, \delta)\big]}\\
			\nonumber B_L&:= \sum_{a\in [\bar a+\frac{b}{L^{1/4}}, \bar a +\eta]\cap \frac{\N}{\sqrt{L}}}   \,  e^{\sqrt{L} \big[ a \log \Gamma_\beta+  a  \psi(q_{a,L}, \delta)\big]}.
		\end{align}
		We only consider $A_L$ in the rest for the proof since $B_L$ is dealt with in a completely similar manner.
		With \eqref{psiqmia} we assert that there exists $c_3>0$ such that $|\psi(q_{a,L},\delta)-\psi(\frac{1}{a^2},\delta)|\leq \frac{c_3}{\sqrt{L}}$
		Consequently, there exists $c_4>0$ such that 
		\begin{align}\label{sumpart}
			A_L&\leq  c_4 \sum_{a\in [\bar a -\eta,\bar a-\frac{b}{L^{1/4}}]\cap \frac{\N}{\sqrt{L}}}   \,  e^{\sqrt{L} T_{\delta}(a)}=
			c_4  e^{\sqrt{L} T_{\delta}(\bar a)} \sum_{a\in [\bar a -\eta,\bar a-\frac{b}{L^{1/4}}]\cap \frac{\N}{\sqrt{L}}}   \,  e^{\sqrt{L}[T_{\delta}(a)-T_{\delta}(\bar a)]}.
		\end{align}
		By Lemma \ref{concavity}, there exists $c>0$ such that $T_\delta''(a)\leq -c,$ for all $a$ in any compact subset of $(0,1/\sqrt{q_\delta^*}) =: (0,a^*)$. Moreover, by Lemma~\ref{Nantes Lemme 5.1}, we have $\bar a < a^*$. Therefore,
		\beq
		\label{sumpartBIS}
		A_L \le c_4  e^{\sqrt{L} T_{\delta}(\bar a)} \sum_{a\in [\bar a -\eta,\bar a-\frac{b}{L^{1/4}}]\cap \frac{\N}{\sqrt{L}}}   \,  e^{-\frac{c}{2} (a-\bar a)^2 \sqrt{L}}.
		\eeq
		The sum in the r.h.s.\ of \eqref{sumpartBIS} may then be bounded from above by
		\begin{equation}\label{equ}
			\sum_{n\ge b L^{1/4}}  e^{-\frac{c}{2} \, (\frac{n}{L^{1/4}})^2}=L^{1/4} \int_b^\infty e^{-\frac{c}{2} x^2} \dd x\, (1+o(1)).
		\end{equation}
		The claim follows by combining~\eqref{decouppar}--\eqref{equ}. 
	\end{proof}
	\begin{proof}[Proof of Claim \ref{claimmain}]
		We start by defining (omitting some parameters for conciseness):
		\begin{align}\label{decoupmain}
			\widetilde Z^+_{L,b}:=\widetilde Z_{L,\beta,\delta}^{\circ}\Big(N_\ell / \sqrt{L} \in \cR_{b,L}\Big)&=\sum_{a\in  \cR_{b,L}}  \, (\Gamma_\beta)^{a\sqrt{L}} D_{a\sqrt{L}}(q_{a,L},\delta).
		\end{align}
		We observe that there exists $c_6>0$ such that $|q_{a,L}-\frac{1}{\bar a^2}|\leq c_6/L^{1/4}$ for $a\in \cR_{b,L}$. Thus, since $h\mapsto \kappa^0(h)$ and $q\mapsto \htildeq$ are continuous (see Lemma~\ref{lemrestpos}, Lemma~\ref{Nantes Lemma 4.5} and Remark~\ref{htildeqincreasing}) we deduce from Item (3) in Proposition \ref{asympteach3reg}
		that 
		\begin{equation}\label{asd}
			D_{a\sqrt{L}}(q_{a,L},\delta)=\kappa^0(\tilde h(\tfrac{1}{\bar a^{2}})) \, \frac{\xi(\frac{1}{\bar a^{2}},\delta)}{ (a \sqrt{L})^{3/2}}\,  e^{a \psi(q_{a,L},\delta)\sqrt{L}} \ (1+o(1))
		\end{equation}
		with $o(1)$ uniform in $a\in  \cR_{b,L}$. Lemma~\ref{regulpsitilde} and a Taylor expansion gives
		\begin{equation}\label{expantay}
			\Big|\psi\big(q_{a,L},\delta\big)-\psi\Big(\tfrac{1}{a^2},\delta\Big)+\partial_1\psi\Big(\tfrac{1}{a^2},\delta\Big)\, \tfrac{1}{a\sqrt{L}}\Big|
			\le \tfrac{1}{a\sqrt L} \sup_{t\in [0,1]} 
			\Big|\partial_1\psi\Big(\tfrac{1}{a^2}-\tfrac{t}{a\sqrt L},\delta\Big)-\partial_1\psi\Big(\tfrac{1}{a^2},\delta\Big) \Big|
		\end{equation}
		that is $o(\frac{1}{\sqrt{L}})$ uniformly in $a\in [\bar a-\tfrac{b}{ L^{1/4}}, \bar a+\tfrac{b}{ L^{1/4}}]$. This allows us to rewrite \eqref{decoupmain} as
		\begin{align}\label{decoupmainfol}
			\widetilde Z^+_{L,b}&=(1+o(1))\, \frac{\kappa^0(\tilde h(\tfrac{1}{\bar a^{2}}))\  \xi(\frac{1}{\bar a^{2}},\delta)}{\bar a^{3/2} L^{3/4}}\, e^{\partial_1\psi\big(\frac{1}{\bar a^2},\delta\big)}
			\sum_{a\in  \cR_{b,L}} e^{\sqrt{L}\, T_{\delta}(a)}.
		\end{align}
		At this stage, we recall that $\bar a$ is the maximizer of $T_\delta$ on $(0,\infty)$. Thus, $(T_\delta)^{'}(\bar a)=0$
		and 
		\begin{equation}\label{expantaysec}
			T_\delta(a)=T_\delta(\bar a)+\tfrac{1}{2} (T_\delta)^{''}(\bar a) (a-\bar a)^2+o((a-\bar a)^2).
		\end{equation}
		As a consequence, we can rewrite~\eqref{decoupmainfol}
		as 
		\begin{align}\label{decoupmainfolsec}
			\widetilde Z^+_{L,b}
			&=(1+o(1))\, \frac{c_{3,\beta,\delta}}{\bar a^{3/2}  L^{3/4}}  e^{\sqrt{L} T_{\delta}(\bar a)}
			\sum_{a\in\cR_{b,L}} e^{\frac{1}{2} (T_\delta)^{''}(\bar a) (a-\bar a)^2 \sqrt{L} }.
		\end{align}
		We finally set $v=-(T_\delta)^{''}(\bar a)>0$ and compute by a Riemann sum approximation: 
		\begin{align}\label{calcint}
			\lim_{L\to \infty} \frac{1}{L^{1/4}} \sum_{a\in\cR_{b,L}} e^{\frac{1}{2} (T_\delta)^{''}(\bar a) (a-\bar a)^2 \sqrt{L} }&= \int_{-b}^b e^{-\frac{1}{2} v x^2}\, \dd x =\sqrt{\frac{2\pi}{v}}  \int_{-\sqrt{v}b}^{\sqrt{v} b} e^{-\frac{1}{2} u^2}\, \frac{\dd u}{\sqrt{2\pi}}.
		\end{align}
		
	\end{proof}
	\subsection{Proof of~\texorpdfstring{\eqref{parfunqcritinter}}{}: Critical case}
	\label{sec:prop6.1-crit}
	Analogously to Section~\ref{sec:prop6.1-supercrit}, the proof of~\eqref{parfunqcritinter} is a consequence of Lemma~\ref{Nantes Lemma 6.12}, Claim~\ref{claimrest}, and the two following claims:
	%
	%
	\begin{claim}\label{claimrestbig2}
		For every $\gep>0$ and every $\eta>0$, there exists  $b>0$ such that for $L\in \N$,
		\begin{equation}\label{interbound5}
			\widetilde Z_{L,\beta,\delta}^{\circ}(N_\ell / \sqrt{L} \in \cS_{\eta,L}\setminus \cR_{b,L})\leq \frac{\gep}{\sqrt{L}} e^{\sqrt{L} T_{\delta}(\bar a)}.
		\end{equation}
	\end{claim}
	\begin{claim}\label{claimmain2}
		There exists $m:(0,\infty)\mapsto (0,1)$ such that $\lim_{b\to \infty} m(b)=1$ and such that  for every $b>0$,
		\begin{equation}\label{interbound6}
			\widetilde Z_{L,\beta,\delta}^{\circ}(N_\ell / \sqrt{L} \in \cR_{b,L})=(1+o_L(1)) m(b)\frac{C^{\rm{crit}}_{\beta,\delta,b}}{\sqrt{L}} \, e^{g(\beta,\delta) \sqrt{L}} ,
		\end{equation}
		where $o_L(1)$ depends on $b$, and
		\begin{equation} \label{Nantes 6.38}
			C^{\rm{crit}}_{\beta,\delta,b} =    \frac{ (1+o(1)) \kappa^0(\htildeqbar/2) e^{\htildeqbar} } {\bar a^{3/2} 2 \pi \sqrt{\det} \sqrt{L}} \int_{-\infty}^{\infty} e^{-\gamma t^2} \int_0^\infty \exp \left(  
			-\frac{\alpha_0}{2\det} \left( 
			z + \frac{\alpha_1 2t}{\alpha_2 \bar a^{5/2}}
			\right)^2
			\right) \dd z \dd t,
		\end{equation}
		with $\bar q = 1/{\bar a^2}$, $\det := \det[\bB(\boldsymbol{\Tilde{h}}(\bar q,0))]$ (see \eqref{Nantes 8.52}) and $\gamma = \bar q^{3/2} \htildeqbar \Big(1 + \frac{\bar q}{2\lmgf'(\htildeqbar/2) } \Big)$.%
	\end{claim}
	\begin{proof}[Proof of Claim \ref{claimrestbig2}]
		Let us first remind from Remark \ref{htildeqincreasing} that $q\mapsto \htildeq$ is increasing. The proof of Claim \ref{claimrestbig2} requires more attention, as we have to treat separately the cases $a<\bar a$ and $a > \bar a$. We therefore set
		\begin{equation} \label{Nantes 7.50}
			A_L:= \sum_{a\in [\bar a-\eta,\bar a-\frac{b}{L^{1/4}}]\cap \frac{\N}{\sqrt{L}}}   (\Gamma_\beta)^{a\sqrt{L}} D_{a\sqrt{L}}(q_{a,L},\delta),
		\end{equation}
		and 
		\begin{equation} \label{Nantes 7.50bis}
			B_L:= \sum_{a\in [\bar a+\frac{b}{L^{1/4}}, \bar a +\eta]\cap \frac{\N}{\sqrt{L}}}  (\Gamma_\beta)^{a\sqrt{L}} D_{a\sqrt{L}}(q_{a,L},\delta).
		\end{equation}
		(i) Let us start with \eqref{Nantes 7.50bis}.  
		Using Lemma \ref{Nantes Lemma 4.7} and removing the condition $X_{[1,a \sqrt{L}]}>0$, one can see that 
		\begin{equation}
			D_{a\sqrt{L}}(q_{a,L},\delta) 
			\leq
			c e^{a \psi(q_{a,L},0) \sqrt{L} }		
			\bP_{a \sqrt{L}, \tilde h(q_{a,L})} (A_{a\sqrt{L}} = q_{a,L} a^2 L).
		\end{equation}
	We now use Lemma \ref{Nantes Lemma 8.15} to bound from above this probability. Note that $\tilde h(q_{a,L}) = h^{q'}_{a \sqrt{L}}$ for a certain $q'$ verifying $|q_{a,L} -  q'| \leq {\cst}/{(a \sqrt{L})}$, by Proposition \ref{propapproxG}. Hence, using Lemma \ref{Nantes Lemma 8.15} with this $q'$, we get that there exists $c>0$ such that, uniformly in $a \in [a_1,a_2]$ such that $a \leq \bar a$,
	\begin{equation}
		D_{a\sqrt{L}}(q_{a,L},\delta) \leq \frac{c}{L^{3/4}}e^{a \psi(q_{a,L},0) \sqrt{L}}.
	\end{equation}
	The same ideas as displayed between \eqref{decouppar} and \eqref{equ} end the proof.\\
	(ii) Let us now move on to \eqref{Nantes 7.50}. Using \eqref{calculdeproba2} with $q=q_{a,L}$ and $N = a \sqrt{L}$, and deleting the condition $X_{[1,N]}>0$, we may write
	\begin{equation}
		\begin{aligned}
			D_{a\sqrt{L}}(q_{a,L},\delta) 
		&\leq \cst \, \ProbaSurCritique{A_N = q_{a,L} a^2 L} e^{a \psi(q_{a,L},\delta) \sqrt{L}}.
	\end{aligned}
\end{equation}
Using Lemma \ref{Nantes Lemma 8.29}, one has that there exists $c>0$ such that, uniformly in $a \in [a_1,a_2]$ such that $a \geq \bar a$,
\begin{equation}
	D_{a\sqrt{L}}(q_{a,L},\delta) \leq \frac{c}{L^{3/4}}e^{a \psi(q_{a,L},\delta) \sqrt{L}}.
\end{equation}
The same ideas as displayed between \eqref{decouppar} and \eqref{equ} end the proof. 
\end{proof}
\begin{proof}[Proof of Claim~\ref{claimmain2}]
We compute:
\begin{equation}
	\widetilde Z_{L,\beta,\delta}^{\circ}(N_\ell / \sqrt{L} \in \cR_{b,L})
	=\somme{a \in \cR_{b,L}}{} \left(\Gamma_\beta \right)^{a \sqrt{L}} D_{a \sqrt{L}}(q_{a,L},\delta).
\end{equation}
Let $\bar q = 1/{\bar a^2}$. 
{
	We start by computing: 
	\begin{equation}\label{Nantes 7.98}
		\sumtwo{|k|\le b L^{1/4}\colon}{N(k) = \bar a \sqrt{L}+k\in \bbN} \left( \Gamma_\beta\right)^{N(k)} \EsperanceMarcheAleatoire{e^{(\delta - \beta/2)X_{N(k)}} 1_{\{ A_{N(k)} = L - N(k) \}} }.
	\end{equation}
	Expanding the following expression as $L\to\infty$, we note that
	\beq
	L - N(k) = \bar q N(k)^2 + \mathsf{c}\,N(k)^{3/2},
	\eeq
}
with
\beq
\begin{aligned}
	\label{eq:expand-mathsf-c}
	&\mathsf{c} = -\Big(\frac{1}{\bar a^{1/2}} + \frac{2k}{\bar a^{5/2}} \Big) \frac{1}{L^{1/4}} + \Big({ \frac{k}{2\bar a^{3/2}}} + \frac{2k^2}{\bar a^{7/2}} \Big)\frac{1}{L^{3/4}}
	+ O\Big(\frac{1}{L^{ 5/4}}\Big).\\
	&{\mathsf{c}\, N(k)^{1/2} = -1 -\frac{2k}{\bar a^2}+ \frac{k^2}{\bar a^3 L^{1/2}}+ O\Big(\frac{1}{L}\Big).}
\end{aligned}
\eeq
Recall Proposition \ref{loccontrlimtheoarepos} and all definitions therein. We set and compute:
\begin{equation} \label{Nantes E.1}
	\begin{cases}
		&\alpha_0 = \int_0^1 \lmgf''(\tilde h(\bar q)(x-1/2)) \dd x = 2\lmgf'(\tilde h(\bar q)/2) / \tilde h(\bar q), \\
		&\alpha_1 = \int_0^1 x\lmgf''(\tilde h(\bar q)(x-1/2)) \dd x = \lmgf'(\tilde h(\bar q)/2) / \tilde h(\bar q), \\
		&\alpha_2 = \int_0^1 x^2\lmgf''(\tilde h(\bar q)(x-1/2)) \dd x = (\lmgf'(\tilde h(\bar q)/2)-2\bar q) / \tilde h(\bar q),
	\end{cases}
\end{equation}
and we set
\beq
\det := \det[\bB(\boldsymbol{\Tilde{h}}(q,0))], \qquad \textrm{where} \quad \bB(\boldsymbol{\Tilde{h}}(q,0)) = \begin{pmatrix}
	\alpha_2 & \alpha_1 \\
	\alpha_1 & \alpha_0
\end{pmatrix}.
\eeq 
Using~\eqref{Nantes 7.5} and expanding the scalar product in~\eqref{def_f_h}, the sum in~\eqref{Nantes 7.98} is shown to be asymptotically equivalent to
\begin{equation} \label{Nantes 7.100}
	\begin{aligned}
		\frac{\kappa^0(\tilde h (\bar q)/2)}{\bar a^{3/2} L^{3/4}} \sumtwo{|k|\le b L^{1/4}\colon}{N(k) = \bar a \sqrt{L}+k\in \bbN} &\exp \Big([\log \Gamma_\beta + \psi(\bar q,0)]N(k) - \mathsf{c} \tilde h(\bar q)  N(k)^{1/2} - \frac{\mathsf{c}^2}{2\alpha_0} \Big) \\
		&\qquad \times \int_{0}^{\infty} \exp \Big( -\frac{\alpha_0}{2\det} \Big( z - \mathsf{c}\frac{\alpha_1}{\alpha_0}\Big)^2 \Big) \frac{\dd z}{{2 \pi \sqrt{\det}}}.
	\end{aligned}
\end{equation}
By~\eqref{eq:expand-mathsf-c}, one has:
\begin{equation}
	- \mathsf{c} \tilde h(\bar q)  N(k)^{1/2} - \frac{\mathsf{c}^2}{2\alpha_0} =  \tilde h(\bar q) + 
	\frac{2k \tilde h(\bar q)}{\bar a^2}- \frac{k^2}{\sqrt{L}}\left( \frac{\tilde h(\bar q)}{\bar a^3} + \frac{2}{\bar a^{5} \alpha_0} \right) + o(1),
\end{equation}
where we neglected the terms which vanish as $L\to \infty$, uniformly in $|k|\le bL^{1/4}$.
Therefore, the sum in \eqref{Nantes 7.100} is equal to:
\begin{equation}
	\begin{aligned}
		e^{\tilde h(\bar q)} \sumtwo{|k|\le b L^{1/4}\colon}{N(k) = \bar a \sqrt{L}+k\in \bbN}
		&
		\exp \left(  k\Big(\log \Gamma_\beta + \psi(\bar q,0) + 2 \frac{\tilde h(\bar q)}{\bar a^2}  \Big) - \frac{k^2}{\sqrt{L}}\Big( \frac{\tilde h(\bar q)}{\bar a^3} + \frac{1}{\bar a^{5} \alpha_0} \Big) \right) 
		\\
		&\qquad \times \int_{0}^{\infty} \exp \left( -\frac{\alpha_0}{2\det} \left( z - \mathsf{c}\frac{\alpha_1}{\alpha_0}\right)^2 \right) \frac{\dd z}{{2 \pi \sqrt{\det}}}.
	\end{aligned}
\end{equation}
The terms in front of $k$ in the exponential turn out to cancel out. Indeed,
\begin{equation}
	\begin{aligned}
		\log \Gamma_\beta + \psi(\bar q,0) + 2 \frac{\tilde h(\bar q)}{\bar a^2} &=  \log \Gamma_\beta + \cG(\tilde h(\bar q)) + \bar q \tilde h(\bar q) 
		\quad &\textrm{by~\eqref{deftildepsi},}\\
		&= \log \Gamma_\beta + \cL(\tilde h(\bar q)/2) \quad &\textrm{by~\eqref{Nantes 4.30},}\\
		&= 0 \quad &\textrm{by Remark~\ref{Nantes Remark 4.13}.}
	\end{aligned}
\end{equation}
As for the coefficient in front of $-k^2/\sqrt{L}$, we obtain
\begin{equation}
	\frac{ \htildeqbar}{\bar a^3} + \frac{1}{\bar a^{5} \alpha_0}= 
	\bar q^{3/2} \htildeqbar \Big(1 + \frac{\bar q}{2\lmgf'(\htildeqbar/2) } \Big) =: \gamma >0.
\end{equation}
We therefore get that \eqref{Nantes 7.100} is asymptotically equivalent to:
\begin{equation}
	\frac{\kappa^0(\tilde h(\bar q)/2) e^{\tilde h(\bar q)} }{2 \pi \bar a^{3/2}  \sqrt{\det\times L} } e^{\sqrt{L}\penalite(\bar a)}\int_{-b}^{b} e^{-\gamma t^2} \int_{0}^{\infty} \exp \Big(  
	-\frac{\alpha_0}{2\det} \Big( 
	z + \frac{2}{\bar a^{5/2}}{\frac{\ga_1}{\ga_0}}t
	\Big)^2
	\Big) \dd z \dd t.
\end{equation}
\end{proof}
\subsection{Proof of~\texorpdfstring{\eqref{parfunqsouscritinter}}{}: Subcritical case}
\label{sec:prop6.1-subcrit}
Let $\delta < \delta_c(\beta)$. 
Analogously to the two previous sections, the proof of~\eqref{parfunqsouscritinter} is a consequence of Lemma~\ref{Nantes Lemma 6.12}, Claim~\ref{claimrest},
and the two following claims:
%
%
\begin{claim}\label{claimrestbig3}
For every $\gep>0$ and every $\eta>0$, there exists  $b>0$ such that for $L\in \N$,
\begin{equation}\label{interbound8}
	\widetilde Z^{\, \text{o}}_{L,\beta,\delta}(N_\ell / \sqrt{L} \in \cS_{\eta,L}\setminus \cR_{b,L})\leq \frac{\gep}{L^{3/4}} e^{\sqrt{L} T_{\delta}(\bar a)}.
\end{equation}
\end{claim}
\begin{claim}\label{claimmain3}
There exists $m:(0,\infty)\mapsto (0,1)$ such that $\lim_{b\to \infty} m(b)=1$ and such that for every $b>0$, 
\begin{equation} \label{Nantes 6.70}
	\widetilde Z^{\, \text{o}}_{L,\beta,\delta}{ (N_\ell / \sqrt{L} \in \cR_{b,L})}=(1+o_L(1)) m(b)\frac{C^-_{\beta,\delta}}{L^{3/4}} \, e^{g(\beta,{ 0}) \sqrt{L} } ,
\end{equation}
with $C^-_{\beta,\delta} = C_{\beta,1/\bar a^2, \delta}\sqrt{\frac{2\pi}{|\penalite''(\bar a)|}}$, $C_{\beta,1/\bar a^2, \delta}$ defined in \eqref{Nantes 7.24} and the $o_{L}(1)$ depends on $b$.
\end{claim}
The proofs of Claim
\ref{claimmain3} follow that of 
Claim \ref{claimmain}. 
The prefactor $1/L^{3/4}$ in~\eqref{Nantes 6.70} instead of $1/L^{1/2}$ in the critical and supercritical regimes comes from the application of Item (1) in Proposition~\ref{asympteach3reg}, which carries a prefactor $1/N^2$ instead of the $1/N^{3/2}$ present in Items (2) and (3).
We thus focus on:
\begin{proof}[Proof of Claim~\ref{claimrestbig3}]
The line of proof follows that of Claim~\ref{claimrestbig}.
By the definition of $\gd_0(q)$ in~\eqref{eq:delta0q} and Theorem \ref{surfacetransition}, $\delta_0(1/\bar a^2) = \delta_c(\beta) > \delta$. Since $a \rightarrow \delta_0(1/a^2) $ is continuous in $a$, there exists a constant $\nu_0>0$ such that, for every $\nu < \nu_0$ and $a \in [\bar a -\nu,\bar a + \nu]$, $\delta_0(1/a^2) < \delta_c(\beta)$. Hence, picking $\nu$ smaller than $\nu_0$ and using Item (1) instead of Item (3) in Proposition~\ref{asympteach3reg}, we get the claim, following the same ideas as in \eqref{decouppar}--\eqref{equ}.
\end{proof}
\subsection{Conclusion : from beads to extended beads}
\label{sec:first_stretch}
We may finally conclude the proof of Proposition \ref{asymptoticsinter} by proving the following:
\begin{lemma}
\label{Nantes Lemma 6.12} 
We have
\begin{equation}
	\bar Z^{\, \text{o}}_{L,\beta,\delta} = e^{\beta L}\left( \tilde Z^{\, \text{o}}_{L,\beta,\delta} + [1+o(1)]\frac{e^{-\beta}}{1-e^{-\beta}}  \tilde Z^{\, \text{o}}_{L,\beta,0} \right), \qquad L\to\infty.
	\label{Nantes 6.51}
\end{equation}
Moreover, when $\delta \geq \delta_c(\beta)$,
\begin{equation}
	\bar Z^{\, \text{o}}_{L,\beta,\delta} = e^{\beta L} \tilde Z^{\, \text{o}}_{L,\beta,\delta} [1+o(1)], \qquad L\to\infty.
	\label{Nantes 6.50}
\end{equation}
\end{lemma}
\begin{proof}[Proof of Lemma~\ref{Nantes Lemma 6.12}]
The proof follows from (i) the decomposition of $\bar Z^{\, \text{o}}_{L,\beta,\delta}$ in \eqref{fbead}, (ii) the asymptotics in Claim~\ref{claimmain3} and (iii) and application of the dominated convergence theorem, in the same fashion as in the proof of Lemma~\ref{lem:sharp-asympt-bead-minus}.
\end{proof}
\section{Proof of Proposition~\ref{asympteach3reg}: Sharp asymptotics of the auxiliary partition functions}
\label{compid}
{ In this section we prove Proposition~\ref{asympteach3reg} in several steps. We recall that the aim is to provide sharp asymptotics for the auxiliary partition functions introduced in~\eqref{eq:D_Nqdelta} in terms of the function $\psi$ defined in~\eqref{deftildepsi}. 
The proof is close in spirit to \cite[Section 5]{Legrand_2022}. 
In complement to the event $\cV_{N,k,+}$ defined in~\eqref{defnun}, we define, for $N,k,x \in \N^3$,
\begin{equation}
	\cV_{N,k}^x = \{X_N = x,\ A_N= k,\ X_i > 0,\ 0<i<N\},
	\label{Nantes 8.2}
\end{equation}
so that 
\beq
\cV_{N,k,+} = \bigcup_{x\in \N} \cV_{N,k}^x.
\eeq
This section is divided into subsections corresponding to the different items in Proposition~\ref{asympteach3reg}.
}
\subsection{Proof of Item (1): the subcritical regime}
\label{Nantes Section 6}
In this regime, only small changes are required to make the proof in~\cite[Section 5]{Legrand_2022} work. For the purpose of the proof we set:
\begin{equation}\label{eq:def-tilde-u}
	\tilde{u} := \delta- \delta_0(q) = \delta - \frac{\beta}{2} + \frac{\htildeq}{2}<0.
	\end{equation}
	We divide the proof into four steps. In the first step, we present a decomposition of the partition function that is suitable for computations. In Step 2, we compute the main term. In Step 3, we prove Lemma \ref{lemrestpos}, that we use to compute the main term. In Step 4, we handle the error term.
	\subsubsection{Step 1 : Decomposition of the auxiliary partition function and main ideas \label{Nantes Section 6.3}}
	Using~\eqref{eq:def-tilde-u} and the fact that $A_N = qN^2$ on the event under consideration, we get 
	\begin{equation}
\begin{aligned}
	D_N(q,\delta)&=\EsperanceMarcheAleatoire{e^{(\delta - \beta/2)|X_N|}\ind_{\{\cV_{N,qN^2,+}  \}}}  \\
	&= \EsperanceMarcheAleatoire{e^{\left( \tilde{u}+\frac{\htildeq}{2N} \right)X_N} e^{ \frac{\htildeq}{N}A_N - \frac{\htildeq}{2}\left(  1{+}\frac{1}{N} \right) X_N }\ind_{\{\cV_{N,qN^2,+}  \}}} e^{-q{\htildeq}{N}}.
	\label{6.22A}
\end{aligned}
\end{equation}
Recall the definition of $\bE_{n,h}$ in \eqref{Nantes 4.24}. Since $\tilde{u}<0$, there exists $N_0 \in \N$ and $c>0$ that depends on $\delta$ only such that, for all $N \geq N_0$,
\begin{equation}
\begin{aligned}
	&\EsperanceMarcheAleatoire{e^{\left( \tilde{u}+\frac{\htildeq}{2N} \right)X_N} e^{ \frac{\htildeq}{N}A_N - \frac{\htildeq}{2}\left( 1+\frac{1}{N} \right) X_N }\ind_{\{\cV_{N,qN^2,+}  \}}} \\
	& = \bE_{N,\htildeq}{e^{\left( \tilde{u}+\frac{\htildeq}{2N} \right)X_N}\ind_{\{\cV_{N,qN^2,+} \}} } e^{N \cG_N(\htildeq)}\\
	&=   \Big[ \bE_{N,\htildeq}{e^{\left( \tilde{u}+\frac{\htildeq}{2N} \right)X_N}\ind_{\{\cV_{N,qN^2,+},\ X_N \leq c\log N \}} } + O\left(N^{-3}\right) \Big] e^{N \cG_N(\htildeq)}.
\end{aligned}
\label{6.22B}
\end{equation}
Therefore, it remains to prove that
\begin{equation}
\label{6.22C}
\bE_{N,\htildeq}{e^{\left( \tilde{u}+\frac{\htildeq}{2N} \right)X_N}
	\ind_{\{\cV_{N,qN^2,+},\ X_N \leq c\log N \}} } = \frac{C_{\beta,q,\delta}}{N^2}(1 + o(1)).
	\end{equation}
{Indeed, combining~\eqref{6.22A}, \eqref{6.22B} and recalling from~\eqref{deftildepsi} that $\psi(q,\gd) = \psi(q,0) = -q\htildeq + \cG(\htildeq)$ when $\gd \le \gd_0(q)$ leads to the desired result.}
	For the rest of the proof we focus on obtaining~\eqref{6.22C}.
	Using the change of measure used above in the opposite direction, we retrieve:
	\begin{equation}
\begin{aligned}
	&\bE_{N,\htildeq}{e^{\left(  \tilde{u}+ \frac{\htildeq}{2N}\right) X_N}
		1_{\{\cV_{N,qN^2,+},\ X_N \leq c\log N \}} } e^{\psi(\gd, 0)N}\\
	&\qquad= \EsperanceMarcheAleatoire{e^{(\delta - \beta/2)X_N} \ind_{\{\cV_{N,qN^2,+},\ X_N \leq c\log N \}}}.
\end{aligned}
\end{equation}
{ Recall the definition of $h_N^q$ below Lemma~\ref{Nantes Lemma 4.5}.} We now set $\aN := (\log N)^2$ and define two boxes:
\begin{equation}
\begin{aligned}
	\cC_{N} & :=\left[\mathbf{E}_{N, h_{N}^{q}}\left(X_{\aN}\right)- \left(\aN\right)^{3 / 4}, \mathbf{E}_{N, h_{N}^{q}}\left(X_{\aN}\right)+ \left(\aN\right)^{3 / 4}\right] \\
	\mathcal{D}_{N} & :=\left[\mathbf{E}_{N, h_{N}^{q}}\left(A_{\aN}\right)- \left(\aN\right)^{7 / 4}, \mathbf{E}_{N, h_{N}^{q}}\left(A_{\aN}\right)+ \left(\aN\right)^{7 / 4}\right] 
	\label{Nantes 8.8}
\end{aligned}
\end{equation}
and rewrite
\begin{equation}
\EsperanceMarcheAleatoire{e^{(\delta - \beta/2)X_N} 1_{\{\cV_{N,qN^2,+},\ X_N \leq c\log N \}}}=M_{N, q}+E_{N, q}
\end{equation}
where
\begin{equation}
\begin{aligned}
	M_{N,q} := \EsperanceMarcheAleatoiresansparenthese \Big( 
	e^{(\delta - \beta/2)X_N}
	1 \Big\{ \mathcal{V}_{N, q N^2,+}
	&\cap
	\{X_N \leq c\log N\}
	\cap
	\left\{X_{\aN} \in \cC_{N}, A_{\aN} \in \mathcal{D}_{N}\right\} \\
	&\cap
	\left\{X_{N-\aN} \in \cC_{N}, A_{N}-A_{N-\aN} \in \mathcal{D}_{N}\right\} \Big\} \Big).
	\label{Nantes 8.9}
\end{aligned}
\end{equation}
{%
is the main term and $E_{N, q}$ is the remaining (or error) term. The proof of Item (1) will be complete once we establish Lemmas~\ref{Nantes Lemma 7.2} and~\ref{Nantes Lemma 7.3} below, which we do in Steps 2 and 3 respectively.
Lemma~\ref{Nantes Lemma 7.2} allows us to estimate the main term uniformly in $q \in K$ for $K$ any compact set of $(0,\infty)$. Recalling the definitions of $\vartheta$ and $\kappa$ in~\eqref{Nantes 7.6} and \eqref{defkappa}, we have:}
\begin{lemma}
\label{Nantes Lemma 7.2}
Let $\gb>\gb_c$.
If $0<q_1<q_2<\infty$ and $\delta < \gd_0(q)$ for every $q \in [q_1,q_2]$, then
\begin{equation}
	M_{N, q}=\kappa^0(\htildeq/{2}) \frac{\vartheta (\htildeq)^{-\frac{1}{2}}}{2 \pi N^2} \left( \frac{1}{1-e^{\tilde u}} - \frac{1-\kappa^0(\htildeq/{2})}{1-e^{\tilde{u}-\htildeq}} \right)e^{N \psi(q,0)}(1+o(1))
\end{equation}
where o(1) is a function that converges to 0 as $N \rightarrow \infty$ uniformly in $q \in\left[q_{1}, q_{2}\right] \cap \frac{\N}{N^{2}}$.
\end{lemma}
Lemma~\ref{Nantes Lemma 7.3} allows us deal with the error term:
\begin{lemma}
\label{Nantes Lemma 7.3}
Under the same assumption as in Lemma~\ref{Nantes Lemma 7.2}, there exists $\varepsilon : \N \longrightarrow \R^+$ such that $\limite{N}{\infty} \varepsilon(N) = 0$ and for every $N \in \N$ and $q \in [q_1,q_2] \cap \frac{1}{N^2}\N$,
\begin{equation}
	E_{N, q} \leq \frac{\varepsilon(N)}{N^{2}} e^{N {\psi}(q,0)}.
\end{equation}
\end{lemma}
Before going to the proof, let us remind the reader that a random walk $X$ with law $\mathbf{P}_{N, h}$ has a time-reversibility property, see Remark~\ref{rmk:time-rev-prop}.
\subsubsection{Step 2 : Proof of Lemma~\ref{Nantes Lemma 7.2}\label{Nantes Section 6.4}}
In the following, we use the notation $\bar{x}=\left(x_{1}, x_{2}\right)$ and $\bar{a}=\left(a_{1}, a_{2}\right)$ for couples. Recall the definitions of $\mathcal{C}_{N}$ and $\mathcal{D}_{N}$ in~\eqref{Nantes 8.8} and define:
\begin{equation}
\mathcal{H}_{N}:=\left\{(\bar{x}, \bar{a}) \in \mathcal{C}_{N}^{2} \times \mathcal{D}_{N}^{2}\right\}.
\end{equation}
We use the Markov property on the walk $X$ at times $\aN$ and $N-\aN$ and apply time-reversibility between times $N-\aN$ and $N$ so as to obtain
\begin{equation}
M_{N,q}=\sum_{(\bar{x}, \bar{a}) \in \mathcal{H}_{N}} R_{N}\left(x_{1}, a_{1}\right) T_{N}(\bar{x}, \bar{a}) \somme{ x=1}{c\log N} e^{(\delta - \beta/2)x}\tilde{R}^x_{N}\left(x_{2}, a_{2}\right),
\end{equation}
with
\begin{equation}
\begin{aligned}
	R_{N}(x, a)&:=\mathbf{P}_{\beta}\left(X_{\left[1, \aN\right]}>0, X_{\aN}=x, A_{\aN}=a\right) \\
	\tilde{R}^y_{N}\left(x, a\right) &:= \mathbf{P}_{\beta}\left(X_{\left[1, \aN\right]}>0, X_{\aN}=x, A_{\aN}=a\, | X_0 = y\right),
\end{aligned}
\end{equation}
and, after setting $N'=N-2 \aN$,
\begin{equation}
T_{N}(\bar{x}, \bar{a}):=\mathbf{P}_{\beta}\left(X_{\left[0, N'\right]}>-x_{1}, X_{N'}=x_{2}-x_{1}, A_{N'-1}=q N^{2}-a_{1}-a_{2}-x_{1}\left(N'-1\right)\right).
\end{equation}
By tilting $X_{i+1}-X_{i}$ for $0\le i < \aN$ according to $\widetilde{\mathbf{P}}_{h}$ with $h:=\htildeq / 2$, see~\eqref{Nantes 4.3}, we obtain
\begin{equation}
\begin{aligned}
	\tilde{R}^y_{N}\left(x, a\right) &= \mathbf{P}_{\beta}\left(X_{\left[1, \aN\right]}>{ -y}, X_{\aN}=x - y, A_{\aN}=a - y \aN \right) \\
	&=e^{-\htildeq (x-y)/2 +\aN \mathcal{L}(\htildeq)} \widetilde{\mathbf{P}}_{\htildeq/2} \left( X_{[1,\aN]}>{ -y}, X_{\aN}=x - y, A_{\aN}=a - y\aN\right).
\end{aligned}
\end{equation}
We deal with $T_{N}(\bar{x}, \bar{a})$ and $R_{N}(x_1,a_1)$ in the same way as it was done in \cite[(5.43) to (5.64)]{Legrand_2022}. More specifically, combining \cite[(5.55) and (5.64)]{Legrand_2022} gives:
\begin{equation}
\begin{aligned}
	M_{N,q} = &(1+o(1))e^{N [ \cG(\htildeq) - q \htildeq ]} \frac{\vartheta(\htildeq)^{-\frac{1}{2}}}{2 \pi N^{2}} \kappa^0(h) \\ &\somme{x=1}{c\log N} e^{\tilde{u} x} \ProbaTiltee{\htildeq/2}{X_{[1,\aN]}>{ -x} , 
		X_{\aN} \in \cC_N - x,
		A_{\aN} \in \cD_N- x\aN}
	\label{Nantes 7.39}
\end{aligned}
\end{equation}
We now have to estimate the probability inside the sum, that we will denote $P_{N,q,x}$. A first computation gives, when $h=\htildeq /2$,
\begin{equation}
\left|P_{N, q,x}-\widetilde{\mathbf{P}}_{h}\left(X_{\left[1, \aN\right]}{> -x}\right)\right| \leq \widetilde{\mathbf{P}}_{h}\left(X_{\aN} \notin \mathcal{C}_{N} - x \right)+\widetilde{\mathbf{P}}_{h}\left(A_{\aN} \notin \mathcal{D}_{N} - x\aN \right).
\label{Nantes 7.40}
\end{equation}
As a consequence of \cite[(5.70)]{Legrand_2022}, letting $\alpha_q = \lmgf'(h)$ with $h = \htildeq/2$,
\begin{equation}
\Big\{X_{\aN} \notin \mathcal{C}_{N} - x \Big\} \subset \Big\{|X_{\aN}-\aN \alpha_{q}| \geq \tfrac{1}{2}(\aN)^{3 / 4} - x \Big\}.
\end{equation}
Remind that $\aN = (\log N)^2$ and $x \leq c\log N$. Hence, for $N$ large enough,
\begin{equation}
\Big\{|X_{\aN}-\aN \alpha_{q}| \geq \tfrac{1}{2}(\aN)^{3 / 4} - x \Big\} \subset \Big\{|X_{\aN}-\aN \alpha_{q}| \geq \tfrac{1}{4}(\aN)^{3 / 4}\Big\}.
\end{equation}
Using Tchebychev's inequality:
\begin{equation}
\widetilde{\mathbf{P}}_{\htildeq/2}\Big(|X_{\aN}-\aN \alpha_{q}| \geq \tfrac{1}{4}\left(\aN\right)^{3 / 4}\Big) \leq \frac{\cst}{\sqrt{\aN}} \operatorname{Var}_{\htildeq / 2}\left(X_{1}\right) \leq \frac{\cst}{\sqrt{\aN}},
\end{equation}
where we have used that $\operatorname{Var}_{h}\left(X_{1}\right)=\mathcal{L}^{\prime \prime}(h)$ for every $|h|<\beta/2$. Therefore, we get that $\lim _{N \rightarrow \infty} \widetilde{\mathbf{P}}_{h}\left(X_{\aN} \notin \mathcal{C}_{N}\right)=0$ and, with similar computations, $\lim _{N \rightarrow \infty} \widetilde{\mathbf{P}}_{h}\left(A_{\aN} \notin \mathcal{D}_{N}\right)=0$. Both convergences hold true uniformly in $q \in\left[q_{1}, q_{2}\right]$, because the variance $\operatorname{Var}_{h}$  is a continuous function of $h$, and is equal to $0$ only if $h=0$.
Coming back to \eqref{Nantes 7.40}, we may now write
\begin{equation}
P_{N, q,x}=\widetilde{\mathbf{P}}_{\htildeq/2}\left(X_{\left[1, \aN\right]}{ >-x}\right)(1+o(1))
\label{eq:Pnqx}
\end{equation}
where $o(1)$ is uniform in $q \in\left[q_{1}, q_{2}\right] \cap \frac{ \N}{N^{2}}$. 
By Lemma \ref{Nantes Lemma 4.5}, we have $\htildeq \in [\tilde h(q_1), \tilde h(q_2)] \subset (0, \beta)$ for every $q \in [q_1, q_2]$.
{ Therefore, we may apply Lemma~\ref{lemrestpos} to~\eqref{eq:Pnqx} with $[h_1, h_2] := [\tilde h(q_1), \tilde h(q_2)]$, combine the outcome with~\eqref{Nantes 7.39} and finally get:}
\begin{equation}
M_{N,q} = (1+o(1))e^{N [ \cG(\htildeq) - q \htildeq ]} \frac{[\vartheta(\htildeq)]^{-\frac{1}{2}}}{2 \pi N^{2}} \kappa^0(h) \somme{x=1}{c\log N} e^{\tilde{u} x}  \Big( 1-e^{-2hx} \frac{1-e^{h-\beta/2}}{1-e^{-h-\beta/2}} \Big).
\end{equation}
Lemma \ref{Nantes Lemma 7.2} follows directly. We continue this section with the proof of Lemma~\ref{lemrestpos} and Step~3.
\subsubsection{Proof of Lemma \ref{lemrestpos}\label{prooflemsrestpos}}
Let us begin with~\eqref{eq:kappa2}.
Using the same idea as in \cite[(5.69)]{Legrand_2022}, we pick $h \in [h_1, h_2]$, $k \geq 1
$, and we set $\varepsilon := {h_1}/{2}>0$. Then, by Chernov's inequality,%
\begin{equation}
0 \leq \ProbaTiltee{h}{X_{[1,k]} > -x} - \kappa^x(h) \leq
\sum_{j=k+1}^{\infty}  \ProbaTiltee{h}{ X_j \leq -x} \leq
e^{-\gep x} \sum_{j=k+1}^{\infty} e^{-(\lmgf(h) - \lmgf(h-\varepsilon))j},
\end{equation}
and from the convexity of $\lmgf$, $\lmgf(h) - \lmgf(h-\varepsilon) \ge \lmgf'(h)\gep>0$, so that~\eqref{eq:kappa2} follows.
Let us now prove \eqref{defkappa}. Pick \(h \in [0, \beta/2)\) and define { the stopping time} \(\rho = \inf\{i \geq 1, X_i \leq -x\}\). Then,
\begin{equation} 
1 - \kappa^x(h) = \ProbaTiltee{h}{\rho < \infty} = \mathbb{E}_\beta[e^{hX_\rho - \lmgf(h)\rho}1_{\{\rho < \infty\}} ] = \mathbb{E}_\beta[e^{hX_\rho - \lmgf(h)\rho}], 
\end{equation}
where we used \eqref{Nantes 4.3} and the fact that $\rho$ is finite \(\mathbf{P}_\beta\)-a.s. It is well known (easily adapting~\cite[Lemma 6.2]{CarPet16B}) that \(\rho\) and \(X_\rho\) are independent, with \(-X_\rho\) distributed as $x$ plus a geometric law on \(\mathbb{N} \cup \{0\}\) with parameter \(1 - e^{-\beta/2}\). Furthermore, { \((e^{-hX_{n\wedge \rho} - \lmgf(h) (n\wedge \rho)})_{n \geq 1}\)} is a martingale under \(\bP_\beta\) that is is bounded from above, hence uniformly integrable. Thus, by Doob's optional stopping theorem,%
\begin{equation} 
\mathbb{E}_\beta[e^{-\lmgf(h)\rho}] = \mathbb{E}_\beta[e^{-hX_\rho}]^{-1}.
\end{equation}
As a  consequence:
\begin{equation} 
1 - \kappa^x(h) = \ProbaTiltee{h}{\rho < \infty} = \frac{\mathbb{E}_\beta[e^{hX_\rho}]}{\mathbb{E}_\beta[e^{-hX_\rho}]} = e^{-2hx} \frac{1-e^{h-\beta/2}}{1-e^{-h-\beta/2}}. 
\end{equation}
%
\subsubsection{Step 3: Proof of Lemma \ref{Nantes Lemma 7.3} \label{Nantes Section 6.5} }
A direct application of \eqref{Nantes 7.32} with $n=N$ and $h=0$ leads to bound the error term from above by:
\begin{equation}
\begin{aligned}
	E_{N, q} \leq \somme{x=1}{c\log N} & e^{(\delta - \beta/2 )x} \LoiMarchealeatoire{\mathcal{V}_{N, q N^{2}}^{x} \cap\left\{X_{\aN} \notin \mathcal{C}_{N}\right\} } \\+
	& e^{(\delta - \beta/2 )x} \LoiMarchealeatoire{\mathcal{V}_{N, q N^{2}}^{x} \cap\left\{A_{\aN} \notin \mathcal{D}_{N}\right\}}.
	\\+
	& e^{(\delta - \beta/2 )x} \LoiMarchealeatoire{  \tilde{\mathcal{V}}_{N, q N^{2} - xN}^{-x} \cap\left\{A_{\aN} \notin \mathcal{C}_{N} - x \right\}}
	\\+
	& e^{(\delta - \beta/2 )x} \LoiMarchealeatoire{\tilde{\mathcal{V}}_{N, q N^{2} - xN}^{-x}, \cap\left\{A_{\aN} \notin \mathcal{D}_{N} - a_N x\right\}},
\end{aligned}
\end{equation}
with $
\tilde \cV_{N,k}^{-x} = \{X_N = -x,\ A_N= k,\ X_i > -x,\ 0<i<N\},
$
We will only bound from above $\LoiMarchealeatoire{\mathcal{V}_{N, q N^{2}}^{x} \cap\left\{X_{\aN} \notin \mathcal{C}_{N}\right\} }$ and $\LoiMarchealeatoire{\mathcal{V}_{N, q N^{2}}^{x} \cap\left\{X_{\aN} \notin \mathcal{D}_{N}\right\} }$ : the two other terms can be bounded from above using the same method.
Tilting the law with \eqref{Nantes 4.24}, we obtain that for $\mathcal{B}=\left\{X_{\aN} \notin \mathcal{C}_{N}\right\}$ or $\mathcal{B}=\left\{A_{\aN} \notin \mathcal{D}_{N}\right\}$ and neglecting the $e^{(\htildeq x)/(2N)}$ term :
\begin{equation}
e^{(\delta - \beta/2 )x} \mathbf{P}_{\beta}\left(\mathcal{V}_{N, q N^{2}}^x \cap \mathcal{B} \right) 
\leq e^{\psi_{N, h_{N}^{q}}\left(q N^{2}, 0\right)} e^{\tilde{u}x}  \mathbf{P}_{N, h_{N}^{q}}\left(\mathcal{V}_{N, q N^{2}}^x \cap \mathcal{B} \right).
\label{Nantes 8.31}
\end{equation}
Using Proposition \ref{propapproxG} and the definition in \eqref{Nantes 4.24} and \eqref{exprepsitilde}, we change $\psi_{N, h_{N}^{q}}\left(q N^{2}, 0\right)$ to  $\psi(q,0)N$ in the exponential of the r.h.s. in \eqref{Nantes 8.31}, paying at most a constant factor. Therefore, the proof of Lemma \ref{Nantes Lemma 7.3} is complete if we prove the following claim, { since $\tilde u<0$}.
\begin{claim}
For $\left[q_{1}, q_{2}\right] \subset(0, \infty)$, there exists $\varepsilon: \mathbb{N} \mapsto \mathbb{R}^{+}$such that $\lim _{N \rightarrow \infty} \varepsilon(N)=0$ and for every $N \in \mathbb{N}$, $q \in\left[q_{1}, q_{2}\right] \cap \frac{1}{N^{2}}$ { and $0< x \le c \log N$},
\begin{equation}
	\mathbf{P}_{N, h_{N}^{q}}\left(\mathcal{V}^x_{N, q N^{2}} \cap\left\{X_{\aN} \notin \mathcal{C}_{N}\right\}\right)+\mathbf{P}_{N, h_{N}^{q}}\left(\mathcal{V}^x_{N, q N^{2}} \cap\left\{A_{\aN} \notin \mathcal{D}_{N}\right\}\right) \leq \frac{\varepsilon(N)}{N^{2}}.
	\label{Nantes 7.52}
\end{equation}
\label{Claim 6.6}
\end{claim}
\begin{proof}[Proof of Claim~\ref{Claim 6.6}]
For the purpose of the proof, let us note
\beq
R_{N, q}^x:=\mathbf{P}_{N, h_{N}^{q}}\left(\mathcal{V}^x_{N, q N^{2}} \cap\left\{X_{\aN} \notin \mathcal{C}_{N}\right\}\right)
\eeq
and
\beq
S_{N, q}^x:=\mathbf{P}_{N, h_{N}^{q}}\left(\mathcal{V}^x_{N, q N^{2}} \cap\left\{A_{\aN} \notin \mathcal{D}_{N}\right\}\right).
\eeq
We decompose $R_{N, q}^x$ according to the values taken by $X_{\aN}$ and $A_{\aN}$. Then, we use the Markov property at time $\aN$, combined with the time reversal property of Remark~\ref{rmk:time-rev-prop} with $n:=N$, $h:=h_{N}^{q}$, $j:=\aN$ and the event
\begin{equation}
	\Big\{x \in \mathbb{N}^{N-j-1}: \sum_{1\le i < N-j} x_{i}=q N^{2}-z\Big\}
\end{equation}
on the time interval $\left[\aN, N\right]$, in order to obtain
\begin{equation}
	\begin{aligned}
		R_{N, q}^x=\sum_{y \in \mathbb{N} \backslash \mathcal{C}_{N}} \sum_{z \in \mathbb{N}} \mathbf{P}_{N, h_{N}^{q}}\left(X_{\aN}=y,\ A_{\aN}=z,\ X_{\left[1, \aN\right]}>0\right) \\
		\times \mathbf{P}_{N, h_{N}^{q}}\left(X_{\left[1, N-\aN\right]}>0,\ X_{N-\aN}=y,\ A_{N-\aN-1}=q N^{2}-z{\, |\, X_0 = x}\right) .
	\end{aligned}
	\label{Nantes 7.54}
\end{equation}
Using Proposition \ref{loccontrlimtheoarepos}, one can see that 
\begin{equation}
	R_{N, q}^{x} \leq \frac{\cst}{N^2} \mathbf{P}_{N, h_{N}^{q}}(X_{a_N} \notin \cC_{N}).
\end{equation}
This was dealt with in \cite[(5.28) to 
(5.30)]{Legrand_2022}. Hence we have the existence of $\varepsilon : \N 
\longrightarrow \R$ such that $\varepsilon(N) \to 0$ as $N\to\infty$ and 
$R_{N,q}^{x} \leq {\varepsilon(N)}/{N^2}$ for all $x \leq c\log N$.
\par Let us now consider $S_{N,q}^x$, which we decompose similarly to \eqref{Nantes 7.54}. The same idea AS for $R_{N,q}^x$ gives:
\begin{equation}
	S_{N,q}^x \leq \frac{\cst}{N^2} \mathbf{P}_{N, h_{N}^{q}}(A_{a_N} \notin \cD_{N}).
\end{equation}
This was dealt in \cite[(5.33)]{Legrand_2022}. Hence we have the existence of $\varepsilon : \N 
\longrightarrow \R$ such that $\varepsilon(N) \to 0$ as $N\to\infty$ and 
$R_{N,q}^{x} \leq {\varepsilon(N)}/{N^2}$ for all $x \leq c\log N$.
\end{proof}
\subsection{Proof of Item (2): the critical regime} 
The aim of this section is to estimate the partition function at the critical point. We will follow the idea set forth~\ref{Nantes Section 6} and use the same symmetric change of measure. We set $\aN := (\log N)^2$ and $(\bN)$ a sequence such that $ (\log N)^2 \leq  \bN = o(\sqrt{N}/(\log N))$.  Recall the sets $\cC_N$ and $\cD_N$ defined in \eqref{Nantes 8.8}. By Lemma \ref{Nantes Lemma 4.5} and Proposition \ref{propapproxG}, $\htildeq$ and $h_{N}^q$ belong to a compact subset of $(0,\beta)$ as $q$ varies $[q_1,q_2]$ and $N \geq N_0$, hence we denote by ${K}/{2}$ the maximum of $\EsperanceTiltee{N, h_{N}^q}{X_1}$ for all $N \geq N_0$. We now split the partition function as follows:
\beq
\EsperanceMarcheAleatoire{e^{(\delta - \beta/2)X_N}1_{\{\cV_{N,(q+c/\sqrt{N})N^2,+}  \}}} = M_{N,q} + E_{N,q}\, ,
\eeq
with
\begin{equation}
M_{N,q} := \EsperanceMarcheAleatoiresansparenthese \Big( 
e^{(\delta - \beta/2)X_N}
1_{\{ \mathcal{V}_{N, (q+c/\sqrt{N}) N^2,+},\,
	K \aN \leq X_N \leq \bN \sqrt{N},\,
	X_{\aN} \in \cC_{N},\, A_{\aN} \in \mathcal{D}_{N}\}} \Big)
\label{Nantes 8.21}
\end{equation}
and $E_{N,q}$ the remainder term.
The proof will come out as a consequence of Lemmas \ref{Lemme main term critical} and \ref{Lemme error term critical}, respectively proven in Steps 1 and 2 below. 
First, recall the definitions of $f$, $\kappa$ and $\psi$ in \eqref{def_f_h}, \eqref{defkappa} and \eqref{deftildepsi}.
\begin{lemma}
Let { $\beta > \beta_c$ and $0< q_1 < q_2 < \infty$. As $N\to \infty$, uniformly in $q \in\left[q_{1}, q_{2}\right] \cap \frac{\bbN}{N^{2}}$ and assuming $\gd = \gd_0(q)$,}
\begin{equation}
	M_{N, q} = \frac{\gamma({ q,c})}{N^{3/2} }  e^{\psi(q,0)N - \htildeq [c \sqrt{N}]}(1+o(1)),
\end{equation}
where
\begin{equation}
	\label{eq:attention}
	\gamma({ q,c}) := \kappa^0(\htildeq) \int_0^\infty f_{ \boldsymbol{\Tilde{h}}(q,0)} (c ,u) \dd u
\end{equation}
{ and $c$ comes from the left-hand side of~\eqref{Nantes 7.5}.}
\label{Lemme main term critical}
\end{lemma}
The next lemma allows us to control the error term.
\begin{lemma}
Let { $\beta > \beta_c$ and $0< q_1 < q_2 < \infty$. As $N\to \infty$, uniformly in $q \in\left[q_{1}, q_{2}\right] \cap \frac{\bbN}{N^{2}}$ and assuming $\gd = \gd_0(q)$,}
\begin{equation}
	E_{N, q} = o(N^{-3/2}) e^{{\psi}(q,0)N - \htildeq [c \sqrt{N}]}.
\end{equation}
\label{Lemme error term critical}
\end{lemma}
\subsubsection{Step 1: proof of Lemma~\ref{Lemme main term critical} (main term)\label{Section main term critical}}
We first change the measure similarly as in Section \ref{Nantes Section 6.4}: 
\begin{equation}
M_{N, q}=\sum_{(x, a) \in \cC_N \times \cD_N} R_{N}\left(x, a\right) 
\somme{z=K \aN}{\bN \sqrt{N}}
T_{N}(x, a, z)  e^{(\delta - \beta/2)z},
\label{Nantes 8.59}
\end{equation}
with
\begin{equation}
R_{N}(x, a):=\mathbf{P}_{\beta}\left(X_{\left[1, \aN\right]}>0, X_{\aN}=x, A_{\aN}=a\right)
\end{equation}
and, setting $N_1 = N - \aN$,
\begin{equation}
T_{N}(x, a,z):=\mathbf{P}_{\beta}\big(X_{N_{1}}=z-x, A_{N_{1}}=q N^{2}+cN^{3/2}-a-x N_{1}, X_{\left[0, N_{1}\right]}> -x \big).
\end{equation}
We first work on $T_N(x,a,z)$. Using the tilting defined in \eqref{Nantes 4.24},
\begin{equation}\label{expoaftertilt}
T_N(x,a,z)\,  e^{(\delta - \beta/2)z}  = G_{N,x,a,z}^q \, e^{m_{N,x,a,z}^q},
\end{equation}
with 
\begin{equation}
G_{N,x,a,z}^q := \Proba{N_1,h_{N_1}^q}{X_{N_1} = z-x, A_{N_1} = qN^2+cN^{3/2}-a-x N_1, X_{[0,N_1]}>-x},
\end{equation}
and
\begin{equation}\label{defmnq}
m_{N,x,a,z}^q :={ -h_{N_1}^q \big[
	\tfrac{qN^2-a}{N_1} - x  
	+ c\, \tfrac{N^{3/2}}{N_1}  
	-
	\tfrac{1}{2}( 1 - \tfrac{1}{N_1})(z - x)\big]
}+ N_1 \cG_{N_1}(h_{N_1}^q)+(\delta - \tfrac{\beta}{2})z.
\end{equation}
At this stage, we aim at simplifying \eqref{defmnq}. We use that
\beq
\frac{1}{N_1}=\frac{1}{N} +O\Big(\frac{\aN}{N^2}\Big)
\eeq
and Proposition~\ref{Nantes approxaux} in order to obtain
\begin{align}\label{defmnxz}
\nonumber m_{N,x,a,z}^q 
&={ -\htildeq \Big[  q N+ c\sqrt{N} + q \aN -\tfrac{x}{2}\Big]} +(N-\aN) \cG(\htildeq)+(\delta-\tfrac{\beta}{2}+\tfrac{\htildeq}{2}) z+{ O(\tfrac{a_N^2}{N})} \\
&={ -\htildeq \Big[  q N+ c\sqrt{N} + q \aN-\tfrac{x}{2}\Big]} +(N-\aN) \cG(\htildeq)+O(\tfrac{a_N^2}{N}),
\end{align}
where we used that $\delta = \delta_0(q) = {\beta}/{2}-{\htildeq}/{2}$ to go to the last line.
%


We now consider $R_N(x,a)$. By using the tilting procedure set forth in \eqref{Nantes 4.3} to the increments $(X_{i+1}-X_{i})_{i=0}^{\aN-1}$ with the value $h := {\htildeq}/{2}$, we obtain
\begin{equation}\label{defofrnim}
R_{N}(x, a )=
e^{-\frac{\htildeq}{2} x +\aN\,  \lmgf(\frac{\htildeq}{2})}\  \widetilde{\mathbf{P}}_{\frac{\htildeq}{2}}\left(X_{[1,\aN]}>0, X_{\aN}= x, A_{\aN}=a \right).
\end{equation}
At this stage, we recall from Lemma~\ref{lem:technic_IPP} that $ \cG(\htildeq)+ \htildeq q- \mathcal{L}(\frac{\htildeq}{2}) = 0$. We also recall that $\psi (q,0)= \cG(\htildeq)  -\htildeq q$ and that $a_N=o(\sqrt{N})$. Thus, by combining  \eqref{expoaftertilt}, \eqref{defmnxz} and \eqref{defofrnim} we obtain
\begin{align}
R_{N}&(x, a)\   T_N(x,a,z)\   e^{(\delta - \beta/2)z} \\
\nonumber &=(1+o(1))\  G_{N,x,a,z}^q \ \widetilde{\mathbf{P}}_{\frac{\htildeq}{2}}\left(X_{[1,\aN]}>0, X_{\aN}= x, A_{\aN}=a \right)\,e^{N \psi(q,0)-\htildeq [c\, \sqrt{N}]} .
\end{align}
Let us now estimate $G_{N,x,a,z}^q$ with the help of Proposition~\ref{loccontrlimtheoarepos}. To that aim, we first state a lemma that allows us to drop the constraint $X_{[0,N_1]}>-x$ in the definition of $G_{N,x,a,z}^q$.
To that aim, we set 
\begin{equation}\label{deftildeg}
\tilde G_{N,x,a,z}^q = \Proba{N_1,h_{N_1}^q}{A_{N_1} = qN^2+cN^{3/2}-a-x N_1,\, X_{N_1} = z-x},
\end{equation}
and we use Lemma $\ref{Nantes Lemma 8.9}$ below, proven in \ref{Appendix B.3}:
\begin{lemma}
With $\bN = o(\sqrt{N})$, $\aN = (\log N)^2$, $K$ defined before \eqref{Nantes 8.21} 
and $z\in [K \aN, \bN\sqrt{N}]\cap \N$, we have
\begin{equation}
	\begin{aligned}
		\suptwo{(x,a) \in \cC_N\times \cD_N}{q \in [q_1,q_2]} | \tilde G_{N,x,a,z}^q - G_{N,x,a,z}^q | = o(N^{-3}).
	\end{aligned}
\end{equation}
\label{Nantes Lemma 8.9}
\end{lemma}
At this stage, we have:
\begin{equation}\label{Nantes 8.32}
\begin{aligned}
	&M_{N,q}\,e^{-N \psi(q,0)+\htildeq [c\, \sqrt{N}]}\\ =  &(1+o(1)) \sum_{(x, a) \in \cC_N \times \cD_N}  \somme{z=K \aN}{\bN \sqrt{N}}
	\widetilde{\mathbf{P}}_{\frac{\htildeq}{2}}\left(X_{[1,\aN]}>0, X_{\aN}= x, A_{\aN}=a \right) 
	\tilde G_{N,x,a,z}^q  + O(N^{-2})
\end{aligned}
\end{equation}
and Proposition~\ref{loccontrlimtheoarepos} allows us to write that 
\begin{align}\label{applitclloc}
\nonumber \tilde G_{N,x,a,z}^q&=f_{\boldsymbol{\Tilde{h}}(q,0)}\Big(\tfrac{c N^{3/2}-a-x N_1}{N_1^{3/2}},\tfrac{z-x}{N_1^{1/2}}\Big) \frac{1}{N_1^2}
+O\Big(\frac{\log N}{N^{5/2}}\Big) \\
&=f_{\boldsymbol{\Tilde{h}}(q,0)}\Big(c ,\tfrac{z}{\sqrt{N}}\Big) \frac{1}{N^2}+ O\Big(\frac{\aN}{N^{5/2}}\Big) .
\end{align}
The terms $O(1/N^2)$ in \eqref{Nantes 8.32} and $O(\aN/N^{5/2})$ in \eqref{applitclloc} turn out to be negligible, hence we do not write them in what follows:
\begin{equation}\label{frontmn}
\begin{aligned}
	&\bar M_{N,q}:=M_{N,q}\, e^{-N \psi(q,0)+\htildeq[c\, \sqrt{N}]}\\
	&\sim \frac{1}{N^2}\ \sum_{(x, a) \in \cC_N \times \cD_N}  
	\widetilde{\mathbf{P}}_{\frac{\htildeq}{2}}\left(X_{[1,\aN]}>0, X_{\aN}= x, A_{\aN}=a \right) 
	\somme{z=K \aN}{\bN \sqrt{N}}  f_{\boldsymbol{\Tilde{h}}(q,0)}\Big(c ,\tfrac{z}{\sqrt{N}}\Big) \\
	&\sim \frac{1}{N^{3/2}}\ 
	\widetilde{\mathbf{P}}_{\frac{\htildeq}{2}}\left(X_{[1,\aN]}>0, X_{\aN}\in \cC_N, A_{\aN}\in \cD_N \right) 
	\left[\frac{1}{\sqrt{N}}\somme{z=K \aN}{\bN \sqrt{N}}  f_{\boldsymbol{\Tilde{h}}(q,0)}\Big(c ,\tfrac{z}{\sqrt{N}}\Big) \right]\\
	&\sim \frac{1}{N^{3/2}}\ 
	\kappa^0(\tfrac{\htildeq}{2})  \int_0^\infty f_{\boldsymbol{\Tilde{h}}(q,0)}(c ,u) \dd u,
\end{aligned}
\end{equation}
with the help of \eqref{eq:Pnqx} and a Riemman sum approximation to go to the last line.

%
%
\subsubsection{Step 2: proof of Lemma~\ref{Lemme error term critical} (error term)\label{Section error term critical}}
We split the error term in four parts, namely
\begin{equation}
\begin{aligned}
	E_1 &:= \EsperanceMarcheAleatoiresansparenthese( 
	e^{(\delta - \beta/2)X_N}
	1 
	\{X_N \geq \bN \sqrt{N}\}),  \\
	E_2 &:=\EsperanceMarcheAleatoiresansparenthese( 
	e^{(\delta - \beta/2)X_N}
	1 
	\{0 \leq X_N \leq K \aN, (X_{\aN},A_{\aN}) \in \cC_N \times \cD_N, A_N = qN^2\}), \\
	E_3 &:=\EsperanceMarcheAleatoiresansparenthese( 
	e^{(\delta - \beta/2)X_N}
	1\{ A_{\aN} \notin \mathcal{D}_{N}, A_N = qN^2 \}),\\
	E_4 &:=\EsperanceMarcheAleatoiresansparenthese( 
	e^{(\delta - \beta/2)X_N}
	1  
	\left\{X_{\aN} \notin \cC_{N}, A_N = qN^2 \right\}),
\end{aligned}
\end{equation}
so that $E_{N,q}\le E_1 + E_2 + E_3 + E_4$.
Let us first focus on $E_1$. Using the change of variable defined in \eqref{Nantes 4.24} and computational ideas displayed in \eqref{Nantes 8.59} and below, we get, for $N$ large enough,
\begin{equation}
\begin{aligned}
	E_1 \leq 2e^{N \psi(q,0) - \htildeq [c \sqrt{N}]} \Proba{N,h}{X_N \geq \bN \sqrt{N}}.
\end{aligned}
\label{Nantes 8.39}
\end{equation}
\begin{lemma}For every $0<q_1 < q_2 < \infty$,  there exist $C_1, \cstexpo>0$ such that, for every sequence $(\bN)_{N \in \N}$ diverging to $+\infty$ and $q \in [q_1,q_2]$,
\begin{equation}
	\Proba{N,h_{N}^q}{ X_N \geq \bN \sqrt{N} } \leq C_1 e^{- \cstexpo \bN^2 }.
\end{equation}
\label{Nantes Lemma 8.14}
\end{lemma}
This lemma is proven in Appendix \ref{Nantes Appendix B.1}.
Using Lemma \ref{Nantes Lemma 8.14} and the fact that $\bN \geq
(\log N)^2$, we get
\beq
E_1 = o(N^{-3/2}) e^{N \psi(q,0) - \htildeq [ \cstexpo \sqrt{N}]}.
\eeq
To work on $E_2$, we perform the same change of variable. Copying 
\eqref{Nantes 8.32}, it comes:
\begin{equation}
\begin{aligned}
	E_2 \le 2 e^{N \psi(q,0) - \cstexpo \htildeq  \sqrt{N}} &\sum_{(x, a) \in \cC_N \times \cD_N} \widetilde{\mathbf{P}}_{\htildeq}\left(X_{[1,\aN]}>0, X_{\aN}= x, A_{\aN}=a \right) \somme{z=0}{K \aN} G_{N,x,a,z}^q.
\end{aligned}
\end{equation}
Using the local limit theorem in Proposition \ref{Nantes Proposition 8.10}, we get that, uniformly in $z \in \Z$ and $q \in [q_1,q_2]$, $ G_{N,x,a,z}^q$ can be bounded from above by ${c}/{N^2}$, with $c$ being a function of $[q_1,q_2]$. Hence,
\begin{equation}
\begin{aligned}
	E_2 &\leq \frac{K \aN{(c + o(1))}}{N^2} e^{N \psi(q,0) - \cstexpo \htildeq  \sqrt{N}} \sum_{(x, a) \in \cC_N \times \cD_N} \widetilde{\mathbf{P}}_{\htildeq}\left(X_{[1,\aN]}>0, X_{\aN}= x, A_{\aN}=a \right) \\
	&\leq \frac{K \aN {(c + o(1))}}{N^2} e^{N \psi(q,0) - \cstexpo \htildeq  \sqrt{N}} = \frac{\varepsilon_2(N)}{N^{3/2}}e^{N \psi(q,0) - \cstexpo \htildeq  \sqrt{N}}.
\end{aligned}
\end{equation}
To deal with $E_3$, we use the same decomposition and change of variable as displayed in~\eqref{Nantes 8.59} and below. It gives:
\begin{equation}
\begin{aligned}
	E_3 = 
	(1+o(1)) e^{N \psi(q,0) }\sum_{(x, a) \in \Z \times \Z \backslash \cD_N} &\widetilde{\mathbf{P}}_{\htildeq}\left(X_{\aN}= x, A_{\aN}=a \right) \\ &\times\Proba{N_1,h_{N_1}^q}{A_{N_1} = qN^2 - N_1 x - a}.
\end{aligned}
\end{equation}
Using the local limit theorem in Lemma \ref{Nantes Lemma 8.15} and the fact that $g_{\htildeq}$ is uniformly bounded from above when $q \in [q_1,q_2]$ by a constant $c$, we get:
\begin{equation}
\begin{aligned}
	E_3 &\leq
	\frac{c+o(1)}{N^{3/2}} e^{N \psi(q,0) - \cstexpo \htildeq  \sqrt{N}}\sum_{(x, a) \in \Z \times \Z \backslash \cD_N} \widetilde{\mathbf{P}}_{\htildeq}\left(X_{\aN}= x, A_{\aN}=a \right) \\
	&\leq \frac{c+o(1)}{N^{3/2}} e^{N \psi(q,0) - \cstexpo \htildeq  \sqrt{N}} \widetilde{\mathbf{P}}_{\htildeq}\Big( \Big| A_{\aN} - \EsperanceTiltHomogen{\htildeq}{A_{\aN}} \Big| \geq (\aN)^{7/4}  \Big) \\
	&\leq \frac{c+o(1)}{N^{3/2}} e^{N \psi(q,0) - \cstexpo \htildeq  \sqrt{N}} \frac{ \tilde{\Var}_{\htildeq}(A_N)}{ \aN^3 \sqrt{\aN}} \leq \frac{c+o(1)}{N^{3/2}} e^{N \psi(q,0) - \cstexpo \htildeq  \sqrt{N}} \frac{ \tilde{\Var}_{\htildeq}(X_1)}{ \sqrt{\aN}}.
\end{aligned}
\end{equation}
We now use that $\tilde{\Var}_{\htildeq}(X_1)$ is a continuous function over $q \in [q_1,q_2]$ (for instance, using that $\lmgf$ is $\mathscr{C}^{\infty}$), hence has a maximum over this compact set to conclude the proof. Dealing with $E_4$ uses the same kind of ideas, so we do not repeat the proof there.
\subsection{Proof of Item (3): the supercritical regime}
\label{Nantes Section 8.3}
We divide the proof into three steps. In Step 1, we decompose the partition function in a way that is suitable for computations. In Step 2, we compute the main term, and in Step 3, we handle the error term.
\subsubsection{Step 1: decomposition of the auxiliary partition function}
Recall \eqref{Nantes 8.2}. We seek to estimate $D_N(q,\gd)$, defined in \eqref{eq:D_Nqdelta}. We set
\beq
\label{eq:set_value_h}
h := \sqdelta+ \gd -\gb/2
\eeq
and note that $h>0$ since we assumed that $q > \qinf$, see \eqref{eq:def_qstar}. Recall the definitions of $\cC_{N}$ and $\mathcal{D}_{N}$ in \eqref{Nantes 8.8}, which will be applied throughout this section with the newly defined $h$. As done in Section \ref{Section proof of critical}, we write
\beq
D_N(q,\gd)=M_{N, q}+E_{N, q},
\eeq
where
\begin{equation}
M_{N,q} :=\EsperanceMarcheAleatoiresansparenthese \Big( 
e^{ \left( \delta - \frac{\beta}{2} \right)X_N}
1 \Big\{ \mathcal{V}_{N, q N^2,+}
\cap
\left\{X_{\aN} \in \cC_{N}, A_{\aN} \in \mathcal{D}_{N}\right\}\Big)
\label{Nantes 8.101}
\end{equation}
is the main term and $E_{N, q}$ is the remainder or ``error term''. The proof of Item (3) is a straightforward consequence of Lemmas \ref{Nantes Lemma 8.23} and \ref{Nantes Lemma 8.24} below. Those lemmas are proven in Steps 2 and 3 respectively.
\begin{lemma} For $\gd>0$ and $[q_1,q_2]\subset(q^*_\gd, +\infty)$ such that $\delta > \delta_0(q_1)$ then
\begin{equation}
	M_{N,q}  =\kappa^0(h) \frac{\xi(q,\gd)}{N^{3/2}} e^{N \psi(q,\delta)}(1+o(1)) ,
\end{equation}
where the $o(1)$ is uniform in $q \in [q_1,q_2]$, $\kappa^0$ is defined in \eqref{defkappa}, and
\begin{equation}
	\xi(q,\gd) = \frac{e^{\lmgf (\delta - \frac{\beta}{2} + \sqdelta ) - \lmgf (\delta - \frac{\beta}{2})} }{\sqrt{2\pi \int_0^1 x^2 \lmgf''(\delta - \frac{\beta}{2} + \sqdelta x) \dd x }}.
\end{equation}
\label{Nantes Lemma 8.23}
\end{lemma}
\begin{lemma} For $\gd>0$ and $[q_1,q_2]\subset(q^*_\gd, +\infty)$ such that $\delta > \delta_0(q_1)$ then,
\begin{equation}
	E_{N,q}  = o(N^{-3/2}) e^{N \psi(q,\delta)},
\end{equation}
where the $o(1)$ is uniform in $q \in [q_1,q_2]$.
\label{Nantes Lemma 8.24}
\end{lemma}
\subsubsection{Step 2: Proof of Lemma \ref{Nantes Lemma 8.23}}
We split $M_{N,q}$ as in the proof of Item (1): 
\begin{equation}
M_{N,q}=\sum_{(x,a) \in \cC_N \times \cD_N} R_N \left(x, a \right) T_N(x,a),
\label{Nantes 8.106}
\end{equation}
with 
\begin{equation}
R_N(x,a) := \LoiMarchealeatoire{X_{\aN} = x, A_{\aN} = a,\, X_{[1,\aN]}>0}
\end{equation}
and
\beq
T_N(x,a)
:= \EsperanceMarcheAleatoire{e^{(\delta - \beta/2)(X_{N-\aN}+x)} 1\{A_{N-\aN} = qN^2 - a - (N-\aN)x,\, X_{[1,N-\aN]}>-x\}}.
\eeq
Setting $N_1 := N - \aN$, we rewrite the latter quantity as
\begin{equation}
\begin{aligned}
	&T_N(x,a) \\
	&= \EsperanceMarcheAleatoire{e^{(\delta - \beta/2)({ X_{N_1}}+x)} 1\{A_{N_1} = qN_1^2 + 2 q \aN N_1 + q \aN^2 -a-N_1x, X_{[1,N_1]}>-x\}}.
\end{aligned}
\end{equation}
Recall the value of $h$ set in~\eqref{eq:set_value_h}. Using the tilting in \eqref{Nantes 4.3}, it comes:
\begin{equation}
R_N(x,a) = e^{\aN \lmgf(h)-xh} \ProbaTiltee{h}{X_{\aN} = x, A_{\aN} = a, X_{[1,\aN]}>0}.
\label{Nantes 8.107}
\end{equation}
We now use Lemma \ref{Nantes Lemma 8.25}, whose proof is postponed after the proof of Lemma \ref{Nantes Lemma 8.23}. Recall the expression of $\psi(q,\gd)$ in~\eqref{deftildepsi}:
\begin{lemma} Let $\ga_0>0$ and $\qinf \le q_1 < q_2 < \infty$. For $\delta > \gd_0(q_1)$ we have, uniformly in $q \in [q_1,q_2]$, $\mathfrak{a} \in \Z$ and $\alpha \geq \ga_0$,
\begin{equation}
	\begin{aligned}
		&\EsperanceMarcheAleatoire{e^{(\delta - \beta/2)X_N}1\{ A_N = qN^2+\mathfrak{a},X_{[1,N]}\geq -\alpha \aN  \}  } \\
		&= \frac{\xi(q,\gd)}{N^{3/2}}   \exp\Big(N \psi(q,\delta) - \frac{\mathfrak{a} \sqdelta }{N}\Big) \left(1+o(1)+ O\left( \frac{\mathfrak{a}}{N^{3/2}}\right) \right) ,
	\end{aligned}
	\label{Nantes 9.2}
\end{equation}
with $\xi(q,\gd)$ as in~\eqref{Nantes Lemma 8.23}.
\label{Nantes Lemma 8.25}
\end{lemma}
Recall the definitions of $\cC_N$ and $\cD_N$ in~\eqref{Nantes 8.8}. The condition $x \in \cC_N$ gives that $x/\aN \geq \nu$ for a certain $\nu>0$ . We can therefore apply Lemma~\ref{Nantes Lemma 8.25}, substituting $N-\aN$ for $N$ and $-xN_1 + 2q \aN N_1  + q \aN^2 - a $ for $\mathfrak{a}$:
\begin{equation}
T_N(x,a) \stackrel{N\to\infty}{\sim} \frac{\xi(q,\gd)}{N^{3/2}} \exp\Big\{(N-\aN) \psi(q,\delta) - \frac{(-xN_1 + 2q \aN N_1  + q \aN^2 - a ) \sqdelta }{N_1} + x\Big( \delta - \frac{\beta}{2}\Big)\Big\}.
\label{Nantes 8.109}
\end{equation}
We first notice that $q\aN - a = o(N_1)$. Hence, using \eqref{Nantes 8.107} and \eqref{Nantes 8.109}, it comes:
\begin{equation}
\begin{aligned}
	M_{N,q} \stackrel{N\to\infty}{\sim} \frac{\xi(q,\gd)}{N^{3/2}} e^{N \psi(q,\delta)}
	\sum_{(x,a) \in \cC_N \times \cD_N}
	&
	e^{\aN \left(-2 q \sqdelta + \lmgf(h) -  \psi(q,\delta) \right) + x \left( \delta - \frac{\beta}{2}  + \sqdelta - h\right)} \\
	&\times\ProbaTiltee{h}{X_{\aN} = x, A_{\aN} = a, X_{[1,\aN]}>0}.
\end{aligned}
\label{Nantes 8.111}
\end{equation}
Using Lemma~\ref{lem:technic_IPP}, we remark that
\begin{equation}
-2 q \sqdelta + \lmgf\Big(  \delta - \frac{\beta}{2}  + \sqdelta \Big) { -} \psi(q,\delta) = 0
\end{equation}
Recalling~\eqref{eq:set_value_h}, \eqref{Nantes 8.111} gives:
\begin{equation}
M_{N,q} \stackrel{N\to\infty}{\sim} \frac{\xi(q,\gd)}{N^{3/2}} e^{N \psi(q,\delta)}
\sum_{(x,a) \in \cC_N \times \cD_N}
\ProbaTiltee{h}{X_{\aN} = x, A_{\aN} = a, X_{[1,\aN]}>0}.
\label{Nantes 8.114}
\end{equation}
Using~\cite[(5.67)]{Legrand_2022}, we have
\begin{equation}
\lim_{N\to\infty} 
\sum_{(x,a) \in \cC_N \times \cD_N}
\ProbaTiltee{h}{X_{\aN} = x, A_{\aN} = a, X_{[1,\aN]}>0} = \kappa^0(h)
\label{Nantes 1.115}
\end{equation}
uniformly in $q \in (q_1,q_2)$.
Hence, Lemma \ref{Nantes Lemma 8.23} is proven. It remains to prove Lemma~\ref{Nantes Lemma 8.25} to conclude this step. 
For this purpose, recall \eqref{Nantes 9.5}.
\begin{proof}[Proof of Lemma~\ref{Nantes Lemma 8.25}] Using Lemma~\ref{Nantes Lemma 4.7} and \eqref{deftildepsi}, it comes:
\begin{equation}
	\label{eq:calculdeproba2-REFTO}
	\begin{aligned}
		&\EsperanceMarcheAleatoire{e^{(\delta - \beta/2)X_N}1\{ A_N = qN^2+\mathfrak{a},X_{[1,N]}\geq -\alpha \aN  \}  }
		\\
		&\stackrel{N\to\infty}{\sim} \tilde c \left(\ProbaSurCritique{A_N = qN^2 + \mathfrak{a}, X_{[1,N]} \geq - \alpha a_N} + O(N^{-3})\right) e^{N \psi(q + \mathfrak{a}/N^2,\delta) } \\
		&\stackrel{N\to\infty}{\sim} \tilde c \left(\ProbaSurCritique{A_N = qN^2 + \mathfrak{a}, X_{[1,N]} \geq - \alpha a_N} + O(N^{-3})\right) e^{N \psi(q, \delta) }.
	\end{aligned}
\end{equation}
The last term we have to deal with is $\ProbaSurCritique{A_N = qN^2 + \mathfrak{a}, X_{[1,N]} \geq -\alpha \aN}$. We first deal with the condition $\{X_{[1,N]}>-\alpha \aN\}$, that will reveal to be useless. 
\begin{lemma} For $\gd>0$ and $[q_1,q_2]\subset(q^*_\gd, +\infty)$ such that $\delta > \delta_0(q_1)$, there exists $C>0$ such that, uniformly in $q \in [q_1,q_2] \subset (\qinf, \infty)$:
	\begin{equation}
		\left| \ProbaSurCritique{A_N = qN^2 + \mathfrak{a}, X_{[1,N]} \geq -\alpha \aN} - 
		\ProbaSurCritique{A_N = qN^2 + \mathfrak{a}} \right| \leq O \left({N^{-3}}\right).
	\end{equation}
	\label{Nantes Lemma 8.28}
\end{lemma}
The proof of Lemma \ref{Nantes Lemma 8.28}, that is done in Appendix \ref{Nantes Appendix A.4}, actually makes use of the assumption $q>\qinf$. Using Lemma \ref{Nantes Lemma 8.28} and local limit theorems (see Lemma~\ref{Nantes Lemma 8.29}), one has:
\begin{equation}
	\ProbaSurCritique{A_N = qN^2+\mathfrak{a}, X_{[1,N]} \geq -\alpha \aN} = 
	\frac{l_{c(\sqdelta)}(0)}{N^{3/2}}\Big(1 + O\Big( \frac{\mathfrak{a}}{N^{3/2}} \Big)\Big). 
\end{equation}
\end{proof}
\subsubsection{Step 3: proof of Lemma \ref{Nantes Lemma 8.24}}
We split the error term in two parts:
\begin{equation}
\begin{aligned}
	&E_{N,q} \leq \\
	&
	\EsperanceMarcheAleatoiresansparenthese \Big( 
	e^{(\delta - \beta/2)X_N}
	1  
	\left\{X_{\aN} \notin \cC_{N}, A_N = qN^2 \right\}\Big)
	+ \EsperanceMarcheAleatoiresansparenthese \Big( 
	e^{(\delta - \beta/2)X_N}
	1\{ A_{\aN} \notin \mathcal{D}_{N}, A_N = qN^2 \} \Big)\\
	&= E_1 + E_2.
\end{aligned}
\end{equation}
We first work on $E_1$. We use the same decomposition and change of variable as displayed in Equations~\eqref{Nantes 8.106} to \eqref{Nantes 8.114}. It gives:
\begin{equation}
\begin{aligned}
	E_1 = 
	(1+o(1)) e^{N \psi(q,\delta) }\sum_{(x, a) \in (\Z \backslash \cC_N) \times \Z } &\widetilde{\mathbf{P}}_{h}\left(X_{\aN}= x, A_{\aN}=a \right) \\ &\times\Proba{N_1,h_{N_1}^q}{A_{N_1} = qN^2 - N_1 x - a},
\end{aligned}
\end{equation}
with $h$ as in~\eqref{eq:set_value_h}. Using a local limit theorem, see Lemma \ref{Nantes Lemma 8.29}, and the fact that $l_{c(\sqdelta)}$ is uniformly bounded from above by some constant when $q \in [q_1,q_2]$, we get:
\begin{equation}
\begin{aligned}
	E_1 &\leq
	\frac{\cst}{N^{3/2}} e^{N \psi(q,\delta) }\sum_{(x, a) \in (\Z \backslash \cC_N) \times \Z } \widetilde{\mathbf{P}}_{h}\left(X_{\aN}= x, A_{\aN}=a \right) \\
	&\leq \frac{\cst}{N^{3/2}} e^{N \psi(q,\delta) } \widetilde{\mathbf{P}}_{h}\left( \Big|  X_{\aN} -  \EsperanceTiltHomogen{h}{X_1}\aN \Big| \geq (\aN)^{3/4}  \right) \\
	&\leq \frac{\cst}{N^{3/2}} e^{N \psi(q,\delta) } \frac{ \tilde{\Var}_{h}(X_1)}{\sqrt{\aN}}.
\end{aligned}
\end{equation}
We now use that when $h$ is as in~\eqref{eq:set_value_h}, $\tilde{\Var}_{h}(X_1)$ is a continuous function of $q \in [q_1,q_2]$ (for instance, using that $\lmgf$ is $\mathscr{C}^{\infty}$), hence has a maximum over this { compact set}, in order to conclude the proof. Dealing with $E_2$ uses the same idea, hence we do not repeat the proof there.
\appendix
\section{On auxiliary functions}
\subsection{Proof of Lemma~\ref{concavity} \label{Nantes Appendix fct 1}}
\begin{proof} The regularity of $\penalite$ is clear from~\eqref{eq:def_penalite} and Lemma~\ref{regulpsitilde} so we focus on concavity and compute the first and second derivatives. We split between two cases:
\begin{itemize}
	\item If $q < q_\delta$, i.e. $a> 1/\sqrt{q_\delta}$, then \eqref{eq:explicit-qdelta0} and \eqref{exprepsitilde} give $\penalite(a) = a \log\Gamma_\beta + a \cG(\Tilde{h}(\frac{1}{a^2})) - \frac{1}{a}\Tilde{h}(\frac{1}{a^2}) $. Recalling that $\Tilde{h}(q)$ is the solution of $\cG'(\Tilde{h}(q))=q$, we obtain:
	\begin{equation} \label{Nantes 4.26}
		\begin{aligned}
			&\penalite'(a) = \log\Gamma_\beta + \cG(\Tilde{h}(\tfrac{1}{a^2})) + \tfrac{1}{a^2}\Tilde{h}(\tfrac{1}{a^2}) ; \\
			&\penalite''(a) = -\tfrac{2}{a^5} (2 \Tilde{h}'( \tfrac{1}{a^2}) + a^2\Tilde{h} ( \tfrac{1}{a^2})).
		\end{aligned}
	\end{equation}
	By Lemma~\ref{Nantes Lemma 4.5}, $\Tilde{h}'(x) = 1/{\cG''(h(x))} > 0$, and $\lmgf'$ being odd, it comes that $\Tilde{h}(q) \geq 0$. Hence, $T_\delta''(a)<0$.
	\item If $q > q_\delta$, i.e. $a< 1/\sqrt{q_\delta}$, similar computations give:
	\beq
	\label{eq:calcul_derivees_T}
	\begin{aligned}
		\penalite'(a) &= \log \Gamma_\beta + \cH_\gd(s_\gd(\tfrac{1}{a^2})) + \tfrac{1}{a^2}s_\gd(\tfrac{1}{a^2})\\
		\penalite''(a) &=-\tfrac{2}{a^5} (2 s_\gd'( \tfrac{1}{a^2}) + a^2s_\gd (\tfrac{1}{a^2})).
	\end{aligned}
	\eeq
	We obtain thereof, letting $q=1/a^2$, that $\sign(\penalite''(1/\sqrt{q})) = - \sign(s(q) + 2q s'(q))$. Using that $s'(q)=1/\cH''_\gd(s(q))>0$, we get:
	\beq
	\sign(\penalite''(1/\sqrt{q})) = - \sign\Big(\frac{s(q)}{s'(q)} + 2q\Big).
	\eeq
	Recalling that $\cH_{\delta}''(s) = \int_0^1 x^2 \lmgf''(\delta-\gb/2 + sx ) \dd x$ and integrating by part,
	\begin{equation}
		\begin{aligned}
			\frac{s(q)}{s'(q)} + 2q &= s(q) \int_0^1 x^2 \lmgf''(\delta-\gb/2 + s(q)x)\dd x\\
			&\qquad \qquad  + 
			\int_0^1 2x \lmgf'(\delta-\gb/2 + s(q)x)\dd x \\
			&= \lmgf'(\delta-\gb/2 + s(q)),
		\end{aligned}
	\end{equation}
	which leads to
	\beq
	\sign(\penalite''(1/\sqrt{q})) = -\sign(\gd-\gb/2 + s(q)),
	\eeq
	since $\lmgf'$ is odd.
	We may now conclude from the definition of $\qstar$ in~\eqref{eq:def_qstar} and the comments below it.
\end{itemize}
\end{proof}

\subsection{Proof of Lemma~\ref{lem:limits} \label{Nantes Appendix fct 2}}
\begin{proof}
(i) {\it Small-$a$ limit}. As $a\rightarrow 0$, $\Tilde{h}(1/a^2)$ converges to $\beta$ by Lemma~\ref{Nantes Lemma 4.5}, and eventually $a< 1/\sqrt{q_\gd}$, so that,
\begin{equation}
	\penalite(a) = a \Big[\log \Gamma_\beta + 
	\cH_\delta(s(\tfrac 1 {a^2}))\Big] - \tfrac{1}{a} s(\tfrac 1 {a^2}).
\end{equation}
Since $s(1/a^2)$ converges to $\beta-\delta > 0$ and 
$\cH_\delta$ is bounded from above on its domain of definition, by Lemma~\ref{Nantes Lemma 9.5}, we get our claim.\\
(ii) {\it Large-$a$ limit}. (a) Let us first assume that $\delta < \beta/2$. Then, eventually $a> 1/\sqrt{q_{\gd}}$, so that
\begin{equation}
	\penalite(a) = a\Big[\log \Gamma_\beta + \cG(\Tilde{h}(\tfrac{1}{a^2}))\Big] - \tfrac{1}{a}\Tilde{h}(\tfrac{1}{a^2}).
\end{equation}
Using that $\Tilde{h}(1/a^2)$ converges to $0$, $\cG(0)=0$ and $\log \Gamma_\beta < 0$ (since $\gb>\gb_c$) one has that $\log(\Gamma_\beta) + \cG(\title{h}(1/a^2)) < 0$ when $a$ is large enough, hence the limit is $-\infty$. 
\par (b) Let us now assume that $\gd \ge \gb/2$. In that case, we remind that $q_\gd = 0$ and we investigate the limit of $\penalite'(a)$ when $a\to\infty$. Using~\eqref{eq:calcul_derivees_T} and Lemma~\ref{lem:technic_IPP}, we get:
\beq\label{Derive penalite}
\begin{aligned}
	\penalite'(a) &= \log \Gamma_\gb  + \cH_\gd(s_\gd(1/a^2)) + (1/a^2)s_\gd(1/a^2)\\
	&= \log \Gamma_\gb + \lmgf(s_\gd(1/a^2) + \gd - \gb/2),
\end{aligned}
\eeq
hence
\beq
\lim_{a\to \infty} \penalite'(a) = \log \Gamma_\gb + \lmgf(s_\gd(0)+\gd- \gb/2).
\label{Poitiers 11}
\eeq
With the help of Lemma~\ref{concavity}, we see that this limit is positive if and only if the derivative of $\penalite$ has two roots on its domain of definition, that is equivalent to (recall the definition of $x_\gb$ slightly above~\eqref{eq:def-bar-gd}):
\beq
s_\gd(0)+ \gd - \gb/2 < -x_\gb,\quad \text{ i.e.} \quad
\cH_\gd'(\gb/2- \gd - x_\gb) >0.
\eeq
We may now conclude thanks to the definition of $\bar \gd(\gb)$ in~\eqref{eq:def-bar-gd}.
\end{proof}
\subsection{Proof of Lemma~\ref{lem:bar-gd}}
\label{Nantes Appendix fct 3}
The proof of Lemma~\ref{lem:bar-gd} requires two preparatory lemmas that we state below and prove at the end of this section.
%
%
\begin{lemma}\label{lem:deltabar-plus-s}
If $\gd > \gb/2$ then:
\begin{equation}
	\gd - \gb/2 + s_\gd(0) < 0,
	\label{Nantes 9.31}
\end{equation}
\begin{equation}
	\limite{q}{\infty} \gd - \gb/2 + s_\gd(q) = \gb - \gd > 0.
	\label{Nantes 9.32}
\end{equation}
\end{lemma}
%
%
\begin{lemma} 
For every $\gb>\gb_c$ large enough, we have
\begin{equation}
	\log\Gamma_\gb + \int_0^1 \lmgf\Big( \frac{\beta}{2}x \Big) \dd x < 0.
	\label{Nantes 9.35BIS}
\end{equation}
\label{Nantes Lemma 9.13BIS}
\end{lemma}
\begin{proof}[Proof of Lemma~\ref{lem:bar-gd}]
We first show that in any case $\gb/2 \le \bar \gd(\gb)$. Indeed, plugging $\gd=\gb/2$ into the right-hand side of~\eqref{eq:def-bar-gd} and using Lemma~\ref{Nantes Lemma 9.5}, we get
\beq
\cH'_{\gb/2}(-x_\gb) < \cH'_{\gb/2}(0) = 0.
\eeq
We now turn to the first part of our statement. In view of Lemma~\ref{lem:limits} and its proof, we only have to prove that when $\gb$ is close enough to $\gb_c$, then for $\gd$ close enough to (but smaller than) $\gb$, the limit of $\penalite'(a)$ as $a\to \infty$ is positive. Recall that 
\beq
\lim_{a\to \infty} \penalite'(a) = \log \Gamma_\gb+ \cH_{\delta}(s_\gd(0)) = \log \Gamma_\gb + \int_0^1 \lmgf(s_\gd(0)x+\gd-\gb/2) \dd x.
\eeq
Plugging $\gd=\gb$ in the integral above, we get (note that $-\gb < s_\gb(0)< -\gb/2$)
\beq
\begin{aligned}
	\int_0^1 \lmgf(s_\gb(0)x+\gb/2) \dd x &\ge 
	\int_0^{-\frac{\gb}{2s_\gb(0)}}\lmgf(s_\gb(0)x+\gb/2) \dd x\\
	& = \frac{1}{|s_\gb(0)|} \int_0^{\gb/2} \lmgf(x)\dd x
	\ge \frac{1}{\gb}\int_0^{\gb/2} \lmgf(x)\dd x.
\end{aligned}
\eeq
The last integral converges to a certain positive value when $\gb$ converges to $\gb_c$, while $\log \Gamma_\gb$ converges to $0$, hence, for every $\gb$ close enough to $\gb_c$
\beq
\lim_{\gd\to \gb}\lim_{a\to \infty} \penalite'(a) >0,
\eeq
which completes this step. 
\par Let us now turn to the last part of the statement. We abbreviate $\deltabar= \gd - \gb/2$ and observe that $\ga:= s_\gd(0) + \deltabar < 0$, by Lemma~\ref{lem:deltabar-plus-s}. This time, we write:
\begin{equation}
	\begin{aligned}
		\cH_{\delta}(s_\gd(0)) &= \int_0^1 \lmgf(s_\gd(0)x+\deltabar)\dd x = \int_0^{-\frac{\deltabar}{s_\gd(0)}} \lmgf(s_\gd(0)x+\deltabar)\dd x + \int_{-\frac{\deltabar}{s_\gd(0)}}^1 \lmgf( s_\gd(0)x+\deltabar)\dd x \\
		&=-\frac{\deltabar}{s_\gd(0)} \int_0^1 \lmgf(\deltabar(1-x))\dd x + 
		\left(1+\frac{\deltabar}{s_\gd(0)} \right)\int_0^1\lmgf(\alpha t) \dd t.
	\end{aligned}
\end{equation}
We now use that $a \in (0,\beta/2) \longrightarrow  \int_0^1 \lmgf(at) \dd t$ is a non-decreasing function, because $\lmgf'(x)\geq 0$ for $x \in (0,\beta/2)$. Hence, the parity of $\lmgf$ gives that:
\begin{equation}
	\cH_{\delta}(s_\gd(0)) \leq -\frac{\deltabar}{s_\gd(0)} \int_0^1 \lmgf\Big(\frac{\beta}{2} x \Big)\dd x + 
	\Big(1+\frac{\deltabar}{s_\gd(0)} \Big)\int_0^1\lmgf\Big(\frac{\beta}{2} t \Big) \dd t = \int_0^1 \lmgf\Big(\frac{\beta}{2} t \Big) \dd t.
\end{equation}
We may now conclude thanks to Lemma~\ref{Nantes Lemma 9.13BIS}.
\end{proof}
\begin{proof}[Proof of Lemma \ref{lem:deltabar-plus-s}]
(i) Since $s_\gd(0)$ is the only solution of $\cH_{\delta}'(s) = 0$ and $\cH_{\delta}'$ is increasing, \eqref{Nantes 9.31} will be proven if $\cH_{\delta}'(\gb/2 - \gd) >0$. By~\eqref{Nantes 44.14},
\beq
\cH_{\delta}'(\gb/2 - \gd) = \int_0^1 x \lmgf'((\gd-\gb/2)(1-x)) \dd x
\eeq
is indeed positive, since $\lmgf'(t)>0$ for every $t\in(0,\gb/2)$.\\ 
(ii) This is a straightforward consequence of Item (2) in Lemma~\ref{lem:prop-Hgd-Hngd}.
\end{proof}
\begin{proof}[Proof of Lemma~\ref{Nantes Lemma 9.13BIS}]
By expanding the logarithm in \eqref{Nantes A.1}, it comes:
\begin{equation}
	\begin{aligned}
		& \int_0^1 \lmgf\left(\tfrac{\beta}{2} x \right) \dd x \\
		&= -\log c_\beta +
		\log(1-e^{-\beta}) - \int_0^1\log\Big(1-e^{\frac{\beta}{2}(x-1)  }\Big)\dd x - \int_0^1\log\Big(1-e^{\frac{\beta}{2}(-x-1)  }\Big) \dd x \\
		&= -\log c_\beta + \log ( 1-e^{-\beta})  -\int_0^1 \log\Big(1-e^{-\frac{\beta}{2}x  }\Big)\dd x - \int_1^2 \log\Big(1-e^{-\frac{\beta}{2}x  }\Big) \dd x. \\
		&= -\log c_\beta + \log ( 1-e^{-\beta})  -\int_0^2 \log\Big(1-e^{-\frac{\beta}{2}x  }\Big)\dd x. 
	\end{aligned}
\end{equation}
Since $\Gamma_\beta = e^{-\beta}c_\beta$, the proof is complete if we manage to establish that
\begin{equation}
	g(\beta) := \beta - \log(1-e^{-\beta}) + \int_0^2 \log\Big(1-e^{-\frac{\beta}{2}x  }\Big)\dd x > 0.
\end{equation}
By noticing that the above integral is increasing in $\gb$ and computing the derivative of the remaining part, we observe that $g$ is increasing on $[\log 2, +\infty)$. Also, note that $\gb_c \ge \log 2$, which can be proven by recalling that $z_c = e^{\gb_c/2}$ is the only positive solution of $z^3-z^2-z- 1=0$. 
To compute the above integral, we set $y = e^{-\beta x /2 }$ and find:
\begin{equation}
	\begin{aligned}
		\int_0^2 \log\left(1-e^{-\frac{\beta}{2}x  }\right)\dd x &= \frac{2}{\beta} \int_{e^{-\beta}}^1 \frac{\log(1-y)}{y} \dd y = -\frac{2}{\beta} \sum_{k\ge1}\int_{e^{-\beta}}^1\frac{y^{k-1}}{k} \dd y \\
		&= -\frac{2}{\beta} \sum_{k\ge1} \frac{1}{k^2} \Big( 1 - e^{-\beta k}\Big).
	\end{aligned}
\end{equation}
Recalling that $\sum_{k\ge 1}\frac{1}{k^2} = \frac{\pi^2}{6}$, we conclude with the following lower bound:
\beq
g(\gb) \ge \gb - \frac{\pi^2}{3\gb}\, , 
\eeq
which is positive as soon as $\gb > \pi/\sqrt{3} \approx 1,8138$.
\end{proof}
\subsection{Proof of Proposition~\ref{Proposition concave}}
\label{proof concave}
Recall the definition of $\bar a_{\gb,\gd}$ in~\eqref{defbara} and let $\bar q_{\gb,\gd} := \bar a_{\gb,\gd}^{-2}$. Since $\lmgf$ is convex, it follows from~\eqref{eq:Wulff_beta_delta} that
\beq
\sign(\mathsf{W}_{\gb,\gd}''(t)) = - \sign(s_\gd(\bar q_{\gb,\gd})), \qquad t\in [0,1]. 
\eeq
%
By Lemma~\ref{Nantes Lemma 9.5} and~\eqref{def sqdelta}, there exists a unique $q_0:= \cH_\gd'(0)$ such that $s_\gd(q_0)=0$, and
\beq
\sign(s_\gd(\bar q_{\gb,\gd})) = \sign(\bar q_{\gb,\gd} - q_0).
\eeq
By the variation tabular below~\eqref{Above var tabula}, recalling that $\lim \penalite(a) < 0$ as $a\to\infty$, when $(\gb,\gd)\in\cC_{\rm good}$, we get that
\beq
\sign(\bar q_{\gb,\gd} - q_0) = -\sign(\penalite'(1/\sqrt{q_0})).
\eeq
Using~\eqref{Derive penalite}, we obtain
\begin{equation}
\penalite'(1/\sqrt{q_0}) =  \log \Gamma_\beta + \cL(s_\delta(q_0) + \delta - \beta/2) = \log \Gamma_\beta + \cL(\delta - \beta/2).
\end{equation}
Because $\lmgf$ is increasing on $[0,\beta/2)$, $\lmgf(0)=0$, and $\lim \cL(\delta - \beta/2) = +\infty$ as $\gd \to \gb$, there exists indeed $\deltaconcave(\gb) \in (\beta/2,\beta)$ defined as the unique solution of $\cL(\gd - \beta/2) = -\log \Gamma_\beta$ such that
\beq
\sign(\penalite'(1/\sqrt{q_0})) = \sign(\gd - \deltaconcave(\gb)).
\eeq
This completes the proof, as we finally obtain
\beq
\sign(\mathsf{W}_{\gb,\gd}''(t)) = \sign(\gd - \deltaconcave(\gb)).
\eeq

\section{Technical estimates in the supercritical regime and more}
Remind that $\lmgf$ is defined in \eqref{def:cL}.
\begin{lemma} For every $|h|<\gb/2$,
\begin{equation}
	\lmgf(h) =\log \left( \frac{1}{c_\beta} \left( \frac{1}{1-e^{h - \beta/2}} + \frac{1}{1-e^{-h-\beta/2}} - 1 \right) \right) .
	\label{Nantes A.1}
\end{equation}
\begin{equation}
	\lmgf'(h) =  \frac{ \frac{e^{h - \beta/2}}{\left(1-e^{h - \beta/2} \right)^2} - \frac{e^{-h - \beta/2}}{\left(1-e^{-h - \beta/2} \right)^2}}{ \frac{1}{1-e^{h - \beta/2}} + \frac{1}{1-e^{-h-\beta/2}} - 1}  ,
	\label{Nantes A.2}
\end{equation}
with $c_\beta$ and $\EsperanceMarcheAleatoiresansparenthese$ defined in \eqref{Nantes 3.3}.
\end{lemma}
\begin{proof} 
\eqref{Nantes A.1} is straightforward.
For \eqref{Nantes A.2}, we use that $\sum_{k\ge 1} k x^k = \frac{x}{(1-x)^2} $, hence:
\begin{equation}
	\begin{aligned}
		\EsperanceMarcheAleatoire{ X_1 e^{h X_1} } = 
		\frac{1}{c_\beta} \sum_{k\ge1} (ke^{k h - k\beta/2} - ke^{-k h - k\beta/2}) = \frac{1}{c_\beta} \left( \frac{e^{h - \beta/2}}{\left(1-e^{h - \beta/2} \right)^2} - \frac{e^{-h - \beta/2}}{\left(1-e^{-h - \beta/2} \right)^2}  \right).
		\label{Nantes A.3}
	\end{aligned}
\end{equation}
Using that $\lmgf'(h) = {\EsperanceMarcheAleatoire{ X e^{h X} }}/{\EsperanceMarcheAleatoire{ e^{h X} }}$, \eqref{Nantes A.2} is proven.
\end{proof}
%
\subsection{Proof of Proposition~\ref{Nantes Proposition 9.3}}
\label{Nantes Appendix B.11}
Recall the definitions of $\cH_\delta$ and $\cH_{N,\delta}$ in \eqref{defHH} and \eqref{defHHcont}. The proof relies on the following lemma, whose proof is postponed after the proof of Proposition \ref{Nantes Proposition 9.3}:
	\begin{lemma}\label{aprrochbyhcont}
		For every $K \in (0, \beta/2)$, there exists $C_K > 0$ and $n_K \in \N$ such that, for every $N \geq n_K$ and $s \in [\gb/2 - \gd - K, \gb/2 -\gd + K]$ and $j \in \{0,1\}$: 
		\begin{equation} \label{Nantes 7.75}
			\left| \cH_{N,\delta}^{(j)}(s) - \cH_{\delta}^{(j)}(s) \right| \leq \frac{C_K}{N^2}.
		\end{equation}
		\label{Nantes Lemma 9.3}
	\end{lemma}
	This lemma corresponds to the analogue of \cite[Proposition 5.4]{Legrand_2022} in the supercritical regime, that is for the function $\cH_\gd$ instead of $\cG$. Beforehand, we make the following remark:
	\begin{remark} Adapting the proof of Lemma \ref{Nantes Lemma 9.5}, one can see that there exists 
		$R>0$ such that $\cH_{\delta}''(t) \geq R$ for every $t \in ( -\delta,  \beta - \delta)$. Moreover, 
		for every $M>0$, there exists $K\in(0,\gb/2)$ and $N_0 \in \N$ such that, for every $N \geq N_0$, $
		\cH_{\delta}'(\gb/2 - \gd + K) > M$ and $|\cH_{\delta}'(\gb/2 - \gd - K)|>M$, as well as $\cH_{N,
			\delta}'(\gb/2 - \gd + K) > M$ and $|\cH_{N,\delta}'(\gb/2 - \gd - K)|>M$, the 
		uniformity over $N \geq N_0$ stemming from the convergence of $\cH_{N,\delta}'$ to $\cH_{\delta}'$ over 
		all compact subsets.
		A straightforward consequence of this remark is that, for every $[q_1,q_2] \subseteq (0,\infty)$, there exists $K \in (0,\beta/2)$ and $N_0 \in \N$ such that, for every $q \in [q_1,q_2]$ and $N \geq N_0$, $\sqdeltaN,\sqdelta \in [\gb/2 - \gd - K, \gb/2 - \gd + K]$. 
		\label{Nantes Remark 9.8}
	\end{remark}
	\begin{proof}[Proof of Proposition \ref{Nantes Proposition 9.3}] 
		By Lemma~\ref{aprrochbyhcont} and Remark~\ref{Nantes Remark 9.8}, we get:
		\begin{equation}\label{diffbetweencontanddisc0}
			\begin{aligned}
				|\sqdeltaN - \sqdelta| \leq \frac{1}{R}\left| \int_{\sqdeltaN}^{\sqdelta} \cH_{\delta}''(x)\dd x \right| &= \frac{1}{R} \Big| \cH_{\delta}'(\sqdeltaN) - \cH_{\delta}'(\sqdelta) \Big| \\
				&= \frac{1}{R} \Big| \cH_{\delta}'(\sqdeltaN) - \cH_{N,\delta}'(\sqdeltaN) \Big| \leq \frac{C}{RN^2},
			\end{aligned}
		\end{equation}
		where we have used that $q = \cH_{\delta}'(\sqdelta) = \cH_{N,\delta}'(\sqdeltaN)$. Hence, \eqref{Nantes 9.11} is proven. It remains to prove \eqref{Nantes 9.12}. Again, by Remark \ref{Nantes Remark 9.8}, there exists $K>0$ and $n_0 \in \N$ such that, for every $q \in [q_1,q_2]$ and $N \geq n_0$, both $\sqdeltaN$ and $\sqdelta$ belong to the compact set $I:= [\gb/2 - \gd - K, \gb/2 - \gd + K]$. Letting $C' := \max \{ \cH_{\delta}'(x): x \in I\}$, Lemma~\ref{aprrochbyhcont} and~\eqref{diffbetweencontanddisc0} yield:
		\begin{equation}\label{diffbetweencontanddisc}
			\begin{aligned}
				\Big|\cH_{N,\delta}(\sqdeltaN) - \cH_{\delta}(\sqdelta)\Big| &= 
				\Big|\cH_{N,\delta}(\sqdeltaN)- \cH_{\delta}(\sqdeltaN) \Big| + |\cH_{\delta}(\sqdeltaN) - \cH_{\delta}(\sqdelta)|\\
				&\leq \frac{C}{N^2} + C' |\sqdeltaN - \sqdelta| \leq \frac{\cst}{N^2}.
			\end{aligned}
		\end{equation}
		This completes the proof of Proposition \ref{Nantes Proposition 9.3}.
	\end{proof}
	\begin{proof}[Proof of Lemma \ref{Nantes Lemma 9.3}]
		We start with $j=0$ and take inspiration from \cite[Section A.2]{Legrand_2022} for this proof. We set
		\begin{equation}\label{defhNsx}
			h_{N,s}(x) := \lmgf \Big(\gd - \frac{\gb}{2} + s \frac{(x-1/2)}{N}\Big),
		\end{equation}
		so that, by~\eqref{defHH} and~\eqref{eq:calculLambda_n}, $\cH_{N,\delta}(s) = \frac{1}{N} \somme{k=1}{N}h_{N,s}(k) $. Using the Euler-MacLaurin summation formula (see e.g.\ \cite[Theorem 0.7]{Tenenbaum}), we get:
		\begin{equation}
			\begin{aligned}
				N\cH_{N,\delta}(s) &= A(N,s) + B(N,s)
			\end{aligned}
		\end{equation}
		with 
		\begin{equation}
			A(N,s) := \frac{h_{N,s}(1) + h_{N,s} (N) }{2} + \int_1^N h_{N,s}(t) \dd t
		\end{equation}
		and 
		\begin{equation}
			B(N,s) := -\frac{1}{2} \somme{k=1}{N-1} \int_0^1 h_{N,s}''(x+k)(x^2-x)\dd x.
		\end{equation}
		Let us start with $A(N,s)$. A change of variable gives:
		\begin{equation}
			\begin{aligned}
				&\int_1^N h_{N,s}(t)\dd t = N \int_{1/2N}^{1-1/2N} \lmgf \Big(\gd - \frac{\gb}{2} + st \Big) \dd t \\
				&= N \cH_{\delta}(s) - N\int_0^{1/2N} \lmgf \Big(\gd - \frac{\gb}{2} + st \Big) \dd t - 
				N\int_{1-1/2N}^{1} \lmgf \Big(\gd - \frac{\gb}{2} + st \Big) \dd t.
			\end{aligned}
		\end{equation}
		It remains to see that, for $s \in R_K := [\frac{\gb}{2}-\gd -K, \frac{\gb}{2}-\gd + K]$, $\lmgf$ being $\sC^1$ on $R_K$, we can denote $C_K' := \max \{|\lmgf'(x)|, x \in R_K\}$. Hence, comparing the two terms in the absolute values below to $\lmgf(\gd-\gb/2)/2$ on the first line and $\lmgf(\gd-\gb/2+s)/2$ on the second line, we obtain by the triangular inequality:
		\begin{equation}
			\left| \frac{h_{N,s}(1)}{2} - N\int_0^{1/2N} \lmgf \Big( \gd - \frac{\gb}{2} + st \Big) \dd t \right| \leq |s|\frac{C_K'}{N} \leq \beta \frac{C_K'}{N},
		\end{equation}
		\begin{equation}
			\left| \frac{h_{N,s}(N)}{2} - N\int_{1-1/2N}^{1} \lmgf \Big(\gd - \frac{\gb}{2} + st \Big) \dd t  \right| \leq |s|\frac{C_K'}{N}\leq \beta \frac{C_K'}{N}.
		\end{equation}
		Hence, for $N$ large enough:
		\begin{equation}\label{eq:ANs}
			|A(N,s) - N\cH_{\delta}(s)| \leq \frac{2\gb C_K'}{N} \qquad s \in R_K.
		\end{equation}
		Let us now deal with $B(N,h)$. From~\eqref{defhNsx}, we readily obtain that for every $x \in [1,N]$, $h_{N,s}''(x) \leq C_K''\frac{s^2}{N^2}$, with $C_K'' := \max \{|\lmgf''(t)|\colon t \in R_K\}$. Consequently,
		\beq\label{eq:BNs}
		|B(N,s)| \leq C_K''\frac{\beta^2}{N}, \qquad s \in R_K.
		\eeq
		We conclude the proof of the case $j=0$ by collecting~\eqref{eq:ANs} and~\eqref{eq:BNs}.
		\par The proof of the case $j=1$ follows the same line, replacing~\eqref{defhNsx} by \beq
		h_{N,s}(x) :=\frac{x-1/2}{N} \lmgf '\Big(\gd - \frac{\gb}{2}+ s \frac{(x-1/2)}{N} \Big).
		\eeq
		We leave the details to the reader, for the sake of conciseness.
	\end{proof}
	\subsection{Proof of Lemma~\ref{Nantes Lemma 4.7}}
	\label{Proof 4.7}
	Using the tilted measure in~\eqref{Nantes 9.5} for the next-to-last line and Proposition~\ref{approxhn} for the last line, we may write:
	\begin{equation}\label{expl}
\begin{aligned}
	&\EsperanceMarcheAleatoire{e^{(\delta - \beta/2)X_N}1_{\{ A_N = qN^2, X \in B \} } }\\
	&= e^{- \sqdeltaN qN }
	\EsperanceMarcheAleatoire{e^{
			(\delta - \frac{\beta}{2}) X_N +\sqdeltaN { \frac{A_N}{N}} }1_{\{ A_N = qN^2,X \in B \}}  }
	\\
	&= e^{- \sqdeltaN qN } 
	\EsperanceMarcheAleatoire{ e^{
			\somme{k=1}{N} [ \sqdeltaN \frac{2N+1-2k}{2N} + \frac{\sqdeltaN}{2N} + \delta - \frac{\beta}{2} ] U_k}
		1_{\{ A_N = qN^2, X \in B  \} }} \\
	&= e^{N[\cH_{N,\delta}(\sqdeltaN) -  \sqdeltaN q]}\,  \EsperanceSurCritique{e^{\frac{\sqdeltaN}{2N} X_N}  1_{\{ A_N = qN^2, X \in B \}}} \\
	&\sim_{N} e^{N[\cH_{\delta}(\sqdelta) -  \sqdelta q]}\,  \EsperanceSurCritique{e^{\frac{\sqdeltaN}{2N} X_N}  1_{\{ A_N = qN^2, X \in B \}}}.
\end{aligned}
\end{equation}
We now set $\bN := (\log N)^2$ and we split the expectation in the r.h.s. in \eqref{expl} according to the value of $X_N$, i.e., 
\begin{equation}
\EsperanceSurCritique{e^{\frac{\sqdeltaN}{2N} X_N}  1_{\{ A_N = qN^2, X \in B \}}}:= E_{1,N}+E_{2,N}
\end{equation}
with 
\begin{align}
E_{1,N}&:=\EsperanceSurCritique{e^{\frac{\sqdeltaN}{2N} X_N} 1_{\left\{\left|X_N - 
		\EsperanceSurCritique{X_N}\right| > \bN\sqrt{N}\right\}}\ 1_{\left\{A_N = qN^2, X \in B  \right\}} }, \\
E_{2,N}&:=\EsperanceSurCritique{e^{\frac{\sqdeltaN}{2N} X_N} 1_{\left\{\left|X_N - 
		\EsperanceSurCritique{X_N}\right| \leq \bN\sqrt{N}\right\}}\ 1_{\left\{A_N = qN^2, X \in B  \right\}} }.
		\end{align}
		Let us now bound $E_{1,N}$ from above  by getting rid of its second indicator. Next, we decompose the upper bound depending on the value of $X_N$, and we use Lemma \ref{Nantes Lemma 9.8} combined with  \eqref{Nantes A.20} to obtain
		\begin{equation}\label{boundint}
\begin{aligned}
	E_{1,N}\leq\ &\EsperanceSurCritique{e^{\frac{\sqdeltaN}{2N} X_N} 1_{\left\{\left|X_N - 
			\EsperanceSurCritique{X_N}\right| > \bN\sqrt{N} \right\}} } \\ 
	\leq & \cst \somme{k \in \Z \backslash [-\bN \sqrt{N}, \bN \sqrt{N}]}{} e^{\frac{\sqdeltaN k}{2N}} \ProbaSurCritique{X_N - 
		\EsperanceSurCritique{X_N} \in [k,k+1)} \\
	\leq & \cst \somme{k \in \Z \backslash [-\bN \sqrt{N}, \bN \sqrt{N}]}{} e^{\frac{\sqdeltaN k - 2 ck^2}{2N}} 
	\\
	\leq & \cst  \sum_{k > b_N\sqrt{N}} e^{\frac{\sqdeltaN k - 2ck^2}{2N}} = O(e^{-\frac{c}{4}\bN^2}),
\end{aligned}
\end{equation}
where $O$ in \eqref{boundint} is uniform in $q\in [q_1,q_2]$. At this stage it remains to consider $E_{2,N}$, that equals
\begin{equation}
\begin{aligned}
	(1+o(1)) e^{\sqdeltaN \frac{\EsperanceSurCritique{X_N}}{2N}} \ProbaSurCritique{A_N = qN^2,X \in B, \big |X_N - 
		\EsperanceSurCritique{X_N}\big| \leq \bN\sqrt{N}} .
\end{aligned}
\end{equation}
Using \eqref{Nantes 9.12} and \eqref{Nantes A.20}, it comes that, uniformly in $q \in [q_1,q_2]$:
\begin{equation}
e^{\sqdeltaN \frac{\EsperanceSurCritique{X_N}}{2N}} =(1+o(1)) e^{\frac{\sqdelta}{2} \int_0^1 \lmgf' \left( 
	\delta - \frac{\beta}{2} + \sqdelta t
	\right) \dd t} = (1+o(1)) e^{ \frac{1}{2} \Big[\lmgf \left( \delta - \frac{\beta}{2} + \sqdelta \right) - \lmgf \left( \delta - \frac{\beta}{2} \right) \Big] }.	
	\end{equation}
	Finally, we use Lemma \ref{Nantes Lemma 9.8} again, from which we deduce that 	uniformly in $q\in [q_1,q_2]$
	\begin{align*}
\ProbaSurCritique{A_N = qN^2,X \in B, \big |X_N - 
	\EsperanceSurCritique{X_N}\big| \leq \bN\sqrt{N}} = & \ProbaSurCritique{A_N = qN^2,X \in B}\\
&+O(e^{-c\,  b_N^2})
\end{align*}
and this completes the proof.
\subsection{Proof of Lemma \ref{Nantes Lemma 9.8} and Lemma~\ref{Nantes Lemma 4.24}}
\label{Nantes Proof of Lemma 9.8}
The proofs of these two lemmas are very similar, hence we write only the one of Lemma \ref{Nantes Lemma 9.8}. We are going to split the proof of Lemma \ref{Nantes Lemma 9.8} in two parts:
\begin{equation}
\ProbaSurCritique{X_N - \EsperanceSurCritique{X_N} \geq b \sqrt{N}} \leq C e^{-c b^2},
\label{Nantes A.15}
\end{equation}
\begin{equation}
\ProbaSurCritique{-X_N + \EsperanceSurCritique{X_N} \geq b \sqrt{N}} \leq C e^{-c b^2}.
\label{Nantes A.16}
\end{equation}
We start with \eqref{Nantes A.15} and recall Remark~\ref{Nantes 9.5}. Let $\nu > 0$ (to be chosen small enough in the sequel). By Chernov's inequality,
\begin{equation}
\ProbaSurCritique{X_N - \EsperanceSurCritique{X_N} \geq b \sqrt{N}} \leq
\frac{\EsperanceSurCritique{e^{\nu X_N}}}{e^{\nu \EsperanceSurCritique{X_N} + \nu b \sqrt{N}}}
\label{Nantes A.17}
\end{equation}
We first compute the (logarithm of the) numerator:
\begin{equation}
\begin{aligned}
	\log\, &  \EsperanceSurCritique{e^{\nu X_N}}\\ 
	\qquad &=\somme{k=1}{N} \lmgf \Big( 
	\nu + \delta - \frac{\beta}{2} + \sqdeltaN \frac{2N+1-2k}{2N}
	\Big)
	-
	\lmgf \Big( 
	\delta - \frac{\beta}{2} + \sqdeltaN \frac{2N+1-2k}{2N}
	\Big).
\end{aligned}
\end{equation}
Using Remark \ref{Nantes Remark 9.8}, there exists $K \in (0,\beta/2)$ such that both $\sqdeltaN$ and $\sqdelta$ belong to $[\beta/2 - \delta -K,\beta/2 - \delta +K]$ for all $q \in [q_1,q_2]$. Taking $0<\nu <(\beta/2-K)/{2}$, it comes that for all $q \in [q_1,q_2]$, both $\nu + \delta - \frac{\beta}{2} + \sqdeltaN t$ and $ \delta - \frac{\beta}{2} + \sqdeltaN t$ belong to
\beq
R_K := [-K/2 - \gb/4, K/2 + \gb/4],
\eeq
that is a compact subset of $(-\beta/2,\beta/2)$ for all $t \in [0,1]$. Hence, as $N\to \infty$,
\begin{equation}
\begin{aligned}
	&  
	\somme{k=1}{N} \lmgf \Big( 
	\nu + \delta - \frac{\beta}{2} + \sqdeltaN \frac{2N+1-2k}{2N}
	\Big)
	-
	\lmgf \Big( 
	\delta - \frac{\beta}{2} + \sqdeltaN \frac{2N+1-2k}{2N}
	\Big)
	\\
	&\qquad = O(1) + N\int_0^1 \Big[ \lmgf \Big( 
	\nu + \delta - \frac{\beta}{2} + \sqdeltaN t
	\Big)
	-
	\lmgf \Big( 
	\delta - \frac{\beta}{2} + \sqdeltaN t
	\Big)\Big] \dd t,
\end{aligned}
\label{eq:B28}
\end{equation}
the $O(1)$ being smaller than $\sup\{|\lmgf'(t)|, t\in R_K\}$ and uniform in $q \in [q_1,q_2]$ .
We now compute $\EsperanceSurCritique{X_N}$. Using the same kind of arguments for the last equality:
\begin{equation}
\begin{aligned}
	\EsperanceSurCritique{X_N} &= \somme{k=1}{N} \EsperanceTiltee{\delta - \frac{\beta}{2} + \sqdeltaN \frac{2N+1-2k}{2N}}{X_1} = \somme{k=1}{N} \lmgf' \Big( 
	\delta - \frac{\beta}{2} + \sqdeltaN \frac{2N+1-2k}{2N}
	\Big) \\
	&= O(1) + N\int_0^1 \lmgf' \Big( 
	\delta - \frac{\beta}{2} + \sqdeltaN t
	\Big) \dd t.
\end{aligned}
\label{Nantes AA.20}
\end{equation}
Using \eqref{Nantes 9.11}, we can safely substitute $\sqdeltaN$ by $\sqdelta$ with a cost $O({1}/{N^2})$ at most. By \eqref{eq:B28} and \eqref{Nantes AA.20},
\begin{equation}
\begin{aligned}
	\log\, &\left[\frac{\EsperanceSurCritique{e^{\nu X_N}}}{e^{\nu \EsperanceSurCritique{X_N}}}\right] \\&= 
	O(1) + N \int_0^1 \Big[\lmgf \Big( 
	\nu + \delta - \frac{\beta}{2} + \sqdelta t
	\Big)
	-
	\lmgf \Big( 
	\delta - \frac{\beta}{2} + \sqdelta t
	\Big) -\nu\lmgf' \Big( 
	\delta - \frac{\beta}{2} + \sqdelta t
	\Big)\Big] \dd t \\
	&= O(1) + [1+o(1)]N\nu^2\int_0^1 \lmgf'' \Big( 
	\delta - \frac{\beta}{2} + \sqdelta t
	\Big) \dd t  
	\leq O(1) + [C_K + o(1)]N\nu^2 ,
\end{aligned}
\end{equation}
where the $O(1)$ holds as $N\to \infty$, the $o(1)$ as $\nu\to 0$ and
$C_K := \sup\{|\lmgf''(t)|, t \in R_K\}$. Hence, \eqref{Nantes A.17} becomes:
\begin{equation}
\log \ProbaSurCritique{X_N - \EsperanceSurCritique{X_N} \geq b \sqrt{N}} \leq
O(1) + [C_K + o(1)]N\nu^2C_K -\nu b \sqrt{N}.
\end{equation}
Taking $\nu = {b}/[\sqrt{N} (C_K+1)]$ gives that, for $N$ large enough:
\begin{equation}
\ProbaSurCritique{X_N - \EsperanceSurCritique{X_N} \geq b \sqrt{N}} \leq
C e^{-c b^2}.
\end{equation}
The exact same method can be applied to prove \eqref{Nantes A.16}.
\subsection{Proof of Lemma~\ref{Nantes Lemma 8.28}}
\label{Nantes Appendix A.4}
We first prove the following lemma:
\begin{lemma}
For $\gd>0$ and $[q_1,q_2]\subset(q^*_\gd, +\infty)$ such that $\delta > \delta_0(q_1)$, there exists three positive constants, $c$, $C$ and $\lambda$, such that for every integer $j \in \{1,...,N\}$ and every $q\in[q_1, q_2] \subset (q_\gd^*, +\infty)$:
	\begin{equation}
		\EsperanceSurCritique{e^{-\lambda X_j}} \leq  C e^{-cj}.
	\end{equation}
	\label{Nantes Lemma A.3}
\end{lemma}
To prove this lemma, we will use the following equality that will be proven afterwards:
\begin{lemma}For $q >0$, $\beta/2 \geq \delta > 0$ such that $\delta>\delta_0(q)$, one has:
	\begin{equation}
		\sqdelta - \beta/2 + \delta >  \beta/2 - \delta.
	\end{equation}
\label{Nantes Lemma Poitiers 321}
\end{lemma}
\begin{proof}[Proof of Lemma \ref{Nantes Lemma A.3}]
		We distinguish between three cases.\\
		\noindent \textit{Case 1 : $\sqdelta \ge 0$ and $\beta/2 \ge \delta$.} Using Lemma \ref{Nantes Lemma Poitiers 321}, we set
		\begin{equation}
			\begin{aligned}
				&\lambda := \frac12 \sqdelta - \beta/2 + \delta > 0.
				\label{Nantes A.26}
			\end{aligned}
		\end{equation}
		Under $\ProbaSurCritiqueSansParenthese$, the increments of $X$ are independent but not identically distributed. It comes:
		\begin{equation}
			\quad u_j := \log \EsperanceSurCritique{e^{-\lambda X_j}} = \somme{k=1}{j}
			\lmgf(s_{k,N} - \lambda) - \lmgf(s_{k,N}),
			\label{Nantes A.25}
		\end{equation}
		\beq
		\text{where}\quad s_{k,N} := \delta - \beta/2 + \sqdeltaN \frac{2N+1-2k}{2N}.
		\eeq
		Using the convexity of $\lmgf$, one can see that the sequence $(u_j)$ is non-increasing until a certain integer, defined as $k_0^N$, and then is non-decreasing. One can see that, for $N$ large enough and for $j \leq \frac{N}{4}$, $s_{j,N} - \lambda >  \sqdelta/4$ (using \eqref{Nantes 9.12} to deal with $\sqdeltaN - \sqdelta$). Hence, using the fact that $\lmgf'$ is increasing and the Mean Value Theorem gives, for $j \leq N/4$:
		\begin{equation}
			u_j \leq -j \lambda \lmgf'(\sqdelta/4) := -j c_0,
			\label{Nantes A.27}
		\end{equation}
		setting $c_0 := \lambda \lmgf'(\sqdelta/4)>0$. By \eqref{Nantes A.27}, we now see that, for $N/4 \leq j \leq k_0^N$:
		\begin{equation}
			u_j \leq u_{N/4} \leq -\tfrac14 {c_0 N}\leq -\tfrac14 {c_0 j}.
			\label{Nantes A.28}
		\end{equation}
		Finally, for $j \geq k_0^N$, using the parity of $\lmgf$ for the third line and \eqref{Nantes A.26} for the last line:
		\begin{equation}
			\begin{aligned}
				u_j \leq u_N &= O(1) + N \int_{0}^1 \Big[\lmgf(  \delta - \beta/2 + t \sqdelta - \lambda) - \lmgf(  \delta - \beta/2 + t \sqdelta)\Big] \dd t  \\
				& = O(1) - \frac{N}{\sqdelta} \int_{\delta - \beta/2 }^{\delta - \beta/2 + \sqdelta} \lmgf(t) \dd t + \frac{N}{\sqdelta} \int_{\delta - \beta/2 - \lambda }^{\delta - \beta/2 + \sqdelta - \lambda}  \lmgf( t) \dd t \\
				&= O(1) - \frac{N}{\sqdelta} \int_{\delta - \beta/2 + \sqdelta - \lambda}^{\delta - \beta/2 + \sqdelta} \lmgf(t) \dd t + \frac{N}{\sqdelta} \int_{\beta/2 - \delta}^{ \beta/2 - \delta + \lambda}  \lmgf( t) \dd t \\
				&= O(1) - \frac{N}{\sqdelta} \int_{\beta/2 - \delta + \lambda}^{\delta - \beta/2 + \sqdelta} \lmgf(t) \dd t + \frac{N}{\sqdelta} \int_{\beta/2 - \delta}^{ \beta/2 - \delta + \lambda}  \lmgf( t) \dd t.
			\end{aligned}
		\end{equation}
		Using that $\cL$ is increasing on $(0,\beta/2)$ ends the proof in this case.\\
		\noindent \textit{Case 2 : $\sqdelta \ge 0$ and $\beta/2 < \delta$.} In this case, we set $\lambda = (\delta - \beta/2)/2$ in \eqref{Nantes A.25} and observe that the sequence $(u_j)$ defined therein is always decreasing. Dealing with $j \leq N/4$ is done as above \eqref{Nantes A.27}, and dealing with $j \in [N/4,...,N]$ is done as in \eqref{Nantes A.28}.\\
\noindent \textit{Case 3 : $\sqdelta<0$.} Recall~\eqref{eq:def_qstar}. Since $q>q_\gd^*$, we note that $\delta - \frac{\beta}{2} + \sqdelta >0$.
The proof in this case is easier: taking \eqref{Nantes A.25}, one has:
\begin{equation}
	u_j \leq j \left(- \lmgf(\sqdelta + \delta - \beta/2) + \lmgf(\sqdelta + \delta - \beta/2 - \lambda) \right).
\end{equation}
Setting $\lambda = \frac{1}{2}(\sqdelta + \delta - \beta/2)$ gives $\eqref{Nantes A.27}$ for all $j \leq N$.
\end{proof}
\begin{proof}[Proof of Lemma \ref{Nantes Lemma 8.28}.] A rough upper bound gives:
\begin{equation}
\begin{aligned}
	&\left| \ProbaSurCritique{A_N = qN^2 + \mathfrak{a}, X_{[1,N]} \geq -\alpha \aN} - 
	\ProbaSurCritique{A_N = qN^2+ \mathfrak{a}} \right| \\
	& \leq \somme{k=1}{N} \ProbaSurCritique{X_k \leq -\alpha \aN} \leq \somme{k=1}{N}\frac{\EsperanceSurCritique{e^{-\lambda X_k}}}{e^{\lambda k \alpha \aN}} \leq \frac{\cst}{e^{- \lambda \alpha \aN }} = O\Big( \frac{1}{N^3}\Big),
\end{aligned}
\end{equation}
having used Lemma \ref{Nantes Lemma A.3} for the last line.
\end{proof}
\begin{proof}[Proof of Lemma \ref{Nantes Lemma Poitiers 321}]
Let us proceed by contradiction, and suppose that $\sqdelta - \beta/2 + \delta \leq  \beta/2 - \delta$, i.e. $\sqdelta \leq  \beta - 2\delta$. Then:
\begin{equation}
\begin{aligned}
q &= \int_0^1 t \cL'(\delta - \beta/2 + t \sqdelta) \dd t, & \text{by \eqref{Nantes 3.10},} \\
& \leq \int_0^1 t \cL'(\delta - \beta/2 + t (\beta - 2 \delta)) \dd t, & \text{because $\cL'$ is an increasing function,} \\
& = \int_0^1 (1-x) \cL'(\beta/2 - \delta - x (\beta - 2 \delta)) \dd x, & \text{letting $x = 1-t$,} \\
& = \int_0^1 x \cL'(x (\beta - 2 \delta) - (\beta/2 - \delta)) \dd x, & \text{because $\cL'$ is odd.}
\label{Nantes Poitiers 1}
\end{aligned}
\end{equation}
We now set $f: \delta \rightarrow \int_0^1 x \cL'(x (\beta - 2 \delta) - (\beta/2 - \delta)) \dd x $. We prove that $f$ is decreasing in $\delta$:
\begin{equation}
f'(\delta) = \int_0^1 x(1-2x) \cL''(x (\beta - 2 \delta) - (\beta/2 - \delta)) \dd x < 0,
\end{equation}
thanks to the following lemma:
\begin{lemma} For every even positive function $g$,
	\begin{equation}
\int_0^1 x(1-2x) g(x -1/2 ) \dd x < 0.
	\end{equation}
\label{Nantes Lemma Poitiers 322}
\end{lemma}
Using Lemma \ref{Nantes Lemma Poitiers 322}, we have that $f$ is a decreasing function, because $\cL''$ is positive and even. Therefore, coming back to \eqref{Nantes Poitiers 1} and using that $\delta > \delta_0(q)$ along with \eqref{Nantes 3.7} and \eqref{eq:delta0q} for the last equality:
\begin{equation}
q \leq f(\delta) < f(\delta_0(q)) = q, 
\end{equation}
hence the contradiction.
\end{proof}
\begin{proof}[Proof of Lemma \ref{Nantes Lemma Poitiers 322}]
A computation gives:
\begin{equation}
\begin{aligned}
\int_0^1 x (1-2x) g(x-1/2) \dd x &=-2 \int_{-1/2}^{1/2} x(x+1/2)g(x) \dd x \\
&= -2 \int_{-1/2}^{0} x(x+1/2)g(x) \dd x -2 \int_{0}^{1/2} x(x+1/2)g(x) \dd x \\
&= 2 \int_{0}^{1/2} x(1/2-x)g(x) \dd x  -2 \int_{0}^{1/2} x(x+1/2)g(x) \dd x \\
&= -4 \int_0^{1/2} x^2 g(x) \dd x <0.
\end{aligned}
\end{equation}
\end{proof}
\subsection{Proof of Lemma~\ref{Nantes Lemma 8.29}}
\label{Nantes Appendix A.3}
As Carmona, Nguyen and Pétrélis~\cite[Section 6.1]{CNP16}, we check that the results originally proven in~\cite{DH96} hold uniformly in $q \in [q_1,q_2]$. From Lemma~\ref{lem:prop-Hgd-Hngd}, there exists $\nu>0$ such that both $\sqdelta$ and $\sqdeltaN$ are in $[-\gd+\nu,\beta-\gd-\nu]$ for all $q \in [q_1,q_2]$ and $N$ large enough. We let $\textfrak{E}$ be the holomorphic function defined on $\{z \in \C : \text{Re}(z) \in (-\beta/2,\beta/2) \}$ by $\textfrak{E}(z) = \bE_\beta(e^{z X_1})$. For any $h \in (-\beta/2,\beta/2)$ and $t \in \R$, we set:
\begin{equation}
\phi_h(t) := \frac{\textfrak{E}(h+it)}{\textfrak{E}(h)}.
\end{equation}
We will use some properties of $\phi_h$ that were proven in~\cite{DH96}, namely:
\begin{enumerate}
\item For all $h \in [-\beta/2+\nu,\beta/2-\nu]$ and $t \in \R$,
\begin{equation}
	|\phi_h(t)| \leq \phi_h(0)=1.
	\label{Nantes A.24}
\end{equation}
\item There exists a constant $\alpha = \alpha(\nu)>0$ such that, for all $h \in [-\beta/2+\nu,\beta/2-\nu]$ and $|t|<\pi$,
\begin{equation}
	|\phi_h(t)| \leq \exp(-\alpha^2 t^2 \lmgf''(h)).
	\label{Nantes A.38}
\end{equation}
\item For all $\vartheta \in (0,\pi)$, there exists a positive constant $C = C(\nu,\vartheta)$ such that, for all $h \in [-\beta/2+\nu,\beta/2-\nu]$ and $t \in [\vartheta, 2\pi-\vartheta]$, we have:
\begin{equation}
	\label{Nantes A.39}
	|\phi_h(t)|\leq e^{-C}.
\end{equation}
\end{enumerate}
For all $t \in \R$, we define
\begin{equation}
\Phi_{N}^{\gd, q}(t) := \EsperanceSurCritique{e^{itA_N/N}} = \produit{j=1}{N} \phi_{h_{j,N}}(t_{j,N}),
\label{Nantes A.37}
\end{equation}
where
\begin{equation}
h_{j,N} := \delta - \frac{\beta}{2} + \sqdeltaN \frac{2N+1-2j}{2N} \quad \text{and } t_{j,N} := \Big( 1 - \frac{j-1}{N}\Big) t.
\label{Nantes A.41}
\end{equation}
Note that 
\begin{equation}
\hat \Phi_{N}^{\gd, q}(t) := \Phi_{N}^{\gd, q}(t/\sqrt{N}) \exp \left( -\frac{it}{\sqrt{N}}\EsperanceSurCritique{\frac{A_N}{N}} \right)
\end{equation}
is the characteristic function of the random variable $\frac{A_N}{N} - \EsperanceSurCritique{\frac{A_N}{N}}$ evaluated at $t/\sqrt{N}$. We define 
\begin{equation}
\bar{\Phi}_s(t) = \exp\Big(-\tfrac{1}{2}c(s)t^2\Big),
\end{equation}
that is the characteristic function associated to the density $l_{c(s)}$. Using the well-known inversion formula for the Fourier transform, we rewrite the left-hand side of~\eqref{Nantes 9.23}, that is
\begin{equation}
R_N := N^{3/2} \ProbaSurCritique{A_N=qN^2+x} - l_{c(\sqdelta)} \left( \frac{x}{N^{3/2}} \right)
\end{equation}
as
\begin{equation}
R_N = \frac{1}{2\pi} \int_\sA \hat \Phi_{N}^{\gd, q}(t) e^{-it \frac{x}{N^{3/2}}} \dd t - \frac{1}{2\pi} \int_\R \bar{\Phi}_{\sqdelta}(t) e^{-it \frac{x}{N^{3/2}}} \dd t,
\label{Nantes A.32}
\end{equation}
where $\sA = [-\pi N^{3/2},\pi N^{3/2}]$. The attentive reader could object that 
\beq
q \neq \EsperanceSurCritique{A_N/N^2}.
\eeq
Indeed, by differentiating~\eqref{defHH} with respect to $s$ and evaluating at $s = s_{N,\gd}(q)$, one obtains
\beq
q - \EsperanceSurCritique{A_N/N^2} = \EsperanceSurCritique{X_N}/(2N^2), 
\eeq
and we find that the error term is thus at most $O(1/N)$, uniformly in $q\in[q_1, q_2]$. Following the proof in~\cite[Section 6.1]{CNP16}, we bound \eqref{Nantes A.32} by the sum of four terms,
\begin{equation}
|R_N| \leq \frac{1}{2\pi} \Big( J_1^{(q)} + J_2^{(q)}+ J_3^{(q)}+J_4^{(q)} \Big),
\label{Nantes A.33}
\end{equation}
where, for some positive constant $\Delta$ and setting $B_N := \log N$,
\begin{equation}
\begin{aligned}
	J_1^{(q)} &= \int_{-B_N}^{B_N} | \hat \Phi_{N}^{\gd, q}(t) - \bar{\Phi}_{\sqdelta}(t) | \dd t,\\
	J_2^{(q)} &= \int_{\R \backslash [-B_N,B_N] } |\bar{\Phi}_{\sqdelta}(t)| \dd t,\\
	J_3^{(q)} &= \int_{[- \Delta \sqrt{N}, \Delta \sqrt{N} ] \backslash [-B_N,B_N] } 
	| \hat \Phi_{N}^{\gd, q}(t) |
	\dd t,\\
	J_4^{(q)} &= \int_{\sA \backslash [- \Delta \sqrt{N}, \Delta \sqrt{N} ] } 
	| \hat \Phi_{N}^{\gd, q}(t)  |
	\dd t.
\end{aligned}
\end{equation}
\par (i) First, we bound $J_1^{(q)}$. For $s \in (- \delta ,\beta-\delta)$, we define 
\begin{equation}
c_N(s) := \frac{1}{N}\somme{j=1}{N} \Big(\frac{2N+1-2j}{2N} \Big) ^2 \cL''\Big( \delta - \frac{\beta}{2} + s \frac{2N+1-2j}{2N}\Big).
\end{equation}
Since $\lmgf'$ and $\lmgf''$ are bounded over any compact interval of $(-\gb/2,\gb/2)$, one can check that
\begin{equation}
\sup_{q \in [q_1,q_2]} | c_N(\sqdeltaN) - c(\sqdelta) |=O(1/N).
\label{Nantes B.55}
\end{equation}
Using that \textfrak{E}~is holomorphic on $\{z \in \C : \text{Re}(z) \in (-\beta/2,\beta/2)\}$ and following the argument in \cite[between (6.23) and (6.24)]{CNP16}, we may extend $\lmgf$ to $[-\beta/2+\nu,\beta/2-\nu] + i[-A',A']$ for some $A'>0$ as a branch of the complex logarithm of \textfrak{E}. Hence, for $|t|\le A'\sqrt{N}/2$ and $\sqdeltaN \in [-\delta + \nu, \beta - \delta - \nu]$, we have:
\begin{equation}
\hat \Phi_{N}^{\gd, q}(t) = 
\exp \Big( 
\somme{j=1}{N} \lmgf\Big( i \frac{t_{j,N}}{\sqrt{N}} + h_{j,N}  \Big)
-\lmgf(h_{j,N})
-\frac{i t_{j,N}}{\sqrt{N}} \lmgf'(h_{j,N})
\Big).
\label{Nantes B.24a}
\end{equation}
For $N$ large enough, $\sqdeltaN \in [-\delta + \nu, \beta- \delta - \nu]$ for $q \in [q_1,q_2]$. Since $|t|\le B_N$ implies that $|t|\le A'\sqrt{N}/2$, a Taylor-Lagrange expansion applied to~\eqref{Nantes B.24a} yields
\begin{equation}
\begin{aligned}
	\hat \Phi_{N}^{\gd, q}(t) &= \exp \Big( - \frac{t^2}{2N} \somme{j=1}{N}\Big( 1- \frac{j-1}{N} \Big)^2 \cL''(h_{j,N}) + \frac{O(\log(N)^3)}{\sqrt{N}} \Big) \\
	&= \exp \Big( - \frac{t^2}{2} c_N(\sqdeltaN) + \frac{O(\log(N)^3)}{\sqrt{N}} \Big), \qquad {\text as\ } N\to\infty,
\end{aligned}
\label{Nantes B.57}
\end{equation}
uniformly in $|t|\le B_N$ and $q \in [q_1,q_2]$. Combining \eqref{Nantes B.55} and \eqref{Nantes B.57}, it comes:
\begin{equation}
\begin{aligned}
	\underset{q \in [q_1,q_2],\, |t|\le B_N}{\sup} 
	| \hat \Phi_{N}^{\gd, q}(t) - \bar{\Phi}_{\sqdelta}(t) |
	= O\Big(\frac{B_N^3}{\sqrt{N}}\Big).
\end{aligned}
\end{equation}
Hence, $J_1^{(q)} = O(B_N^4/\sqrt{N})$ uniformly in $q \in [q_1,q_2]$.\\ 
\par (ii) With $\underbar{c} := \inf\{c(\sqdelta) \colon q\in[q_1, q_2]\}$, that is positive, one may write
\begin{equation}
	\sup_{q\in[q_1, q_2]}J_2^{(q)} \le  
	\int_{ \R \backslash[-B_N,B_N]} e^{-\underbar{c}t^2/2} \dd t = O \left( \frac{1}{\sqrt{N}} \right).
\end{equation}
\par (iii) To estimate $J_3^{(q)}$, we fix $t$ such that $B_N < |t| \le \Delta \sqrt{N}$ and put $\Delta = {\pi}/{2}$. Then, all the numbers $t_{j,N}$ in \eqref{Nantes A.41} satisfy $|t_{j,N}| \leq \pi \sqrt{N}$, and evaluating each factor in \eqref{Nantes A.37} with the help of \eqref{Nantes A.38}, we get, denoting $m:=\inf \{\lmgf''(h) : h \in (-\beta/2,\beta/2)\}$(that is positive by the strict convexity of $\lmgf$):
\begin{equation}
| \hat \Phi_{N}^{\gd, q}(t) | \leq \exp(- m\alpha^2 t^2 /3).
\end{equation}
It easily comes that
\begin{equation}
\underset{q \in [q_1,q_2]}{\sup}  J_3^{(q)} \leq \cst \int_{B_N}^{\infty} \exp(-m\alpha^2t^2 /3)\dd t 
= O\Big(\frac{1}{\sqrt{N}}\Big).
\end{equation}
\par (iv) Finally, in order to evaluate $J_4^{(q)}$, we put $\vartheta = {1}/68$ (that is \cite[(4.43)]{DH96} when $k=1$), and for every $|t|>\Delta \sqrt{N}$, we set 
$\mathfrak{N}_N(t) := \#\{1  \leq j \leq N : \frac{1}{2\pi \sqrt{N}}t_{j,N} \notin \bbZ + [-\vartheta,\vartheta]  \}$. With~\eqref{Nantes A.24},~\eqref{Nantes A.39} and~\eqref{Nantes A.37}, it comes:
\begin{equation}
| \hat \Phi_{N}^{\gd, q}(t) | = \produit{j=1}{N} |\phi_{h_{j,N}}(t_{j,N}/\sqrt{N})| \leq e^{-\cst \mathfrak{N}_N(t)}.
\end{equation}
Moreover, there exists $\kappa>0$ such that $\mathfrak{N}_N(t) \geq \kappa N$ for all $|t| > \Delta \sqrt{N}$ and $N$ large enough, by \cite[(4.45)]{DH96}. Therefore,
\begin{equation}
\sup_{q \in [q_1,q_2]} J_4^{(q)} \leq 4\pi^2 N^{3/2}e^{-\cst \kappa N} = O\Big(\frac{1}{\sqrt{N}}\Big).
\end{equation}

\section{Technical estimates at the critical point}
\subsection{Proof of Lemma~\ref{Nantes Lemma 8.14}}
\label{Nantes Appendix B.1}
Let $\delta>0$. We denote by $(U_i)_{i\ge 1}$ the increments of the random walk $X$. Recalling Remark~\ref{rmk:time-rev-prop}, it comes:
\begin{equation}
\begin{aligned}
	\Proba{N,h}{ X_N \geq \bN \sqrt{N} } \leq 
	\frac{\bE_{N,h}\Big(e^{\delta(U_1+...+U_N)}\Big)}{e^{\delta \bN \sqrt{N}}}.
\end{aligned}
\label{Nantes B.1}
\end{equation}
By \eqref{Nantes 4.3}, we see that for every $1\le i \le N$:
\begin{equation}
\bE_{N,h}(e^{\delta U_i}) = 
\tilde{\mathbf{E}}_{\frac{h}{2} \left( 1 - \frac{2i-1}{n} \right) }(e^{\delta U_1}) = \exp\Big[\lmgf\Big(\delta + \frac{h}{2} \Big( 1 - \frac{2i-1}{n} \Big)\Big) - 
\lmgf\Big(\frac{h}{2} \Big( 1 - \frac{2i-1}{n} \Big)\Big)\Big].
\end{equation}
Hence, using that $\lmgf$ is $\sC^2$ and even,
\begin{equation}
\begin{aligned}
	\log \bE_{N,h}\Big(e^{\delta(U_1+...+U_N)}\Big) &= 
	\somme{i=0}{N-1} \lmgf\Big(\delta + \frac{h}{2} \Big( 1 - \frac{2i-1}{n} \Big)\Big) - 
	\somme{i=0}{N-1} \lmgf\Big(\frac{h}{2} \Big( 1 - \frac{2i-1}{n} \Big)\Big)\\
	&=
	O(1) + N \delta \int_0^1\Big[ \lmgf \Big( \frac{h}{2} + t\delta \Big) - \lmgf \Big( \frac{h}{2} - t\delta \Big) \Big]\dd t\\
	&= O(1) + N \delta^2 (\lmgf' (h/2) + o_\delta(1)),
\end{aligned}
\label{Nantes B.3}
\end{equation}
where the $O(1)$ depends on $\sup\{|\lmgf'(x)|, x\in [0,\delta + h/2]\}$ and holds as $N\to \infty$, and the $o_\delta(1)$ holds as $\gd \to 0$. Picking $\delta = (c_0 \bN)/{\sqrt{N}}$ with $c_0= [2 \lmgf'(h/2)]^{-1}$ and combining \eqref{Nantes B.1} and \eqref{Nantes B.3} gives the existence of $\cstexpo>0$ such that
\begin{equation}
\Proba{N,h}{ X_N \geq \bN \sqrt{N} } \leq e^{- \cstexpo \bN^2 - O_N(1) },
\end{equation}  
the $\cstexpo$ being uniform over $q \in [q_1,q_2]$ and the $O_N(1)$ too.
\subsection{Proof of Lemma~\ref{Nantes Lemma 8.9}}
\label{Appendix B.3}
Using the time-reversal property from \eqref{Nantes 7.32}, we get, for $1\le i < N_1$,
\begin{equation}
\begin{aligned}
	&\Proba{N_1, h_{N_1}^q}{  X_{N_1} \in \left[ K \aN - x, \bN \sqrt{N} - x \right], X_i \leq -x }  \\
	&=
	\somme{k=Ka_n-x}{\bN \sqrt{N} - x}  \Proba{N_1, h_{N_1}^q} { X_{N_1} = k, X_i \leq -x } 
	\\
	&=
	\somme{k=Ka_n-x}{\bN \sqrt{N} - x} \Proba{N_1, h_{N_1}^q}{ X_{N_1} = -k, X_{N_1-i} \leq -k - x } 
	\\
	& \leq 
	\somme{k=Ka_n-x}{\bN \sqrt{N} - x} \Proba{N_1, h_{N_1}^q}{ X_{N_1-i} \leq -x }
	\leq b_N \sqrt{N} \Proba{N_1, h_{N_1}^q}{  X_{N_1-i} \leq -x }
\end{aligned}
\end{equation}
Since $x \in \cC_N$, there exists $c_1 > 0$ such that $x > c_1 \aN$ for all $q \in [q_1, q_2]$ and $N > N_0$ a certain integer. It comes:
\begin{equation}
\begin{aligned}
	&\Proba{N_1, h_{N_1}^q}{ X_{N_1} \in \left[ K \aN - x, \bN \sqrt{N} - x \right], \exists i\colon X_i \leq -x } \\
	&\leq
	\somme{i=0}{N_1/2} \Proba{N_1, h_{N_1}^q}{ X_i \leq -x } + b_N \sqrt{N} \somme{i=0}{N_1/2} \Proba{N_1, h_{N_1}^q}{ X_i \leq -x }  \\
	&\leq N^2 e^{-\lambda c_1 \log(N)^2} = o(1/N^2),
\end{aligned}
\end{equation}
using that $\bN \leq \sqrt{N}$. The uniformity in $a$ stems from not imposing the condition on $A_{N_1}$, while the uniformity in $x$ and $q$ is ensured by the uniform lower bound $x > c_1 \aN$, that is true for all $x \in \cC_N$ and $q \in [q_1,q_2]$.
\section{Second-order expansion of the excess free energy at the critical point}
\subsection{Proof of Lemma~\ref{inegcentral}}
\label{Nantes Appendix C.1}
Recall the expression of $\psi(q,\delta)$ in~\eqref{exprepsitilde}. Using \eqref{Nantes 4.29} and \eqref{Nantes 4.30} at the second line, it comes:
\begin{equation} \label{Nantes C.1}
\begin{aligned}
	\psi(q,\delta) - \psi(q,0)  &=  
	[\cH_\gd(\sqdelta) - q \sqdelta] - [\cG(\htildeq) - q\htildeq]\\
	&= \lmgf \Big(\sqdelta + \gd - \frac{\gb}{2} \Big) - \lmgf \Big( \frac{\htildeq}{2} \Big) - 2q \Big( \sqdelta - \htildeq \Big). 
\end{aligned}
\end{equation}
We now work with $\delta= \delta(\gep) := \gd_0(q) + \gep = \frac{\beta}{2} - \frac{\htildeq}{2} + \varepsilon$ and $\varepsilon > 0$ small. Then,
\begin{equation} \label{Nantes C.1,5}
\psi(q,\delta(\gep)) - \psi(q,0) > 0 \Leftrightarrow \lmgf \Big(\sqdelta + \varepsilon - \frac{\htildeq}{2} \Big) - \lmgf \Big( \frac{\htildeq}{2} \Big) > 2q \Big( \sqdelta - \htildeq \Big). 
\end{equation}
By definition of $\htildeq$ and $\sqdelta$,
\begin{equation} \label{Nantes C.2}
\begin{aligned}
	\int_0^1 t \lmgf' \Big( \htildeq t - \frac{\htildeq}{2} \Big) \dd t &= q \\
	\int_0^1 t \lmgf' \Big( \sqdelta t - \frac{\htildeq}{2} + \varepsilon \Big) \dd t &= q. 
\end{aligned}
\end{equation}
Since $\lmgf'$ is an increasing function over $(-\beta/2,\beta/2)$, it comes that $\sqdelta = \htildeq - a \varepsilon - b \varepsilon^2 + o(\varepsilon^2)$, with $a>0$. Let us first determine $a$. From \eqref{Nantes C.2}, we get
\begin{equation} \label{Nantes C.2BIS}
\int_0^1 t \Big( \lmgf' \Big( \htildeq t - \frac{\htildeq}{2} \Big) -  \lmgf' \Big( \htildeq t - \frac{\htildeq}{2} + \varepsilon(1-at) \Big) \Big) \dd t = 0.
\end{equation}
Expanding the line above at first order, we obtain that
\begin{equation}
\int_0^1 t(1-at) \lmgf'' ( \htildeq (t-1/2)) \dd t= 0.
\end{equation}
Integrating par parts (see Lemma~\ref{Nantes Lemma C.1} for the denominator), we obtain:
\begin{equation} \label{Nantes C.5}
a = 
\frac
{\int_0^1 t \lmgf''(\htildeq(t-1/2)) (\htildeq \dd t)}
{\int_0^1 t^2 \lmgf''(\htildeq(t-1/2)) (\htildeq \dd t)}
=\frac{\lmgf' \left( \frac{\htildeq}{2} \right)}{\lmgf' \left( \frac{\htildeq}{2} \right) - 2q}.
\end{equation}
A straightforward computation gives, at first order:
\begin{equation}
\label{Nantes D.7}
\psi(q,\delta) - \psi(q,0) = \varepsilon\Big[(1-a)\lmgf'(\htildeq/2) + 2qa \Big] +
O(\varepsilon^2) = O(\varepsilon^2).
\end{equation}
Hence, the transition is at least of order 2. We now determine $b$. Coming back to \eqref{Nantes C.2} and performing a second-order Taylor expansion, we get:
\begin{equation}
b = \frac{1}{2}\times \frac{\int_0^1 {t(1-at)^2} \lmgf'''(\htildeq (t-1/2)) \dd t  }{ \int_0^1 t^2 \lmgf''(\htildeq (t-1/2)) \dd t}.
\end{equation}
To compute $b$, we use the following lemma, the proof of which is left to the reader (integrate by parts).
\begin{lemma} \label{Nantes Lemma C.1}
\begin{equation}
	\begin{aligned}
		\int_0^1 {t(1-at)^2} \lmgf'''(\htildeq (t-1/2)) ({\htildeq} \dd t) = 
		&{(1-a)^2} \lmgf''(\htildeq/2) + \frac{2(2a-1)}{\htildeq}\lmgf'(\htildeq/2) \\ &- {3a^2} \int_0^1 t^2 \lmgf''(\htildeq (t-1/2)) \dd t\ ,
	\end{aligned}
\end{equation}
\begin{equation}
	\int_0^1 t^2 \lmgf''(\htildeq (t-1/2)) (\htildeq \dd t) = 
	\lmgf'(\htildeq/2) - 2q.
\end{equation}
\end{lemma}
We finally obtain:
\begin{equation}
b =  \frac{\frac{(1-a)^2}{2}\lmgf''(\htildeq/2) + \frac{(2a-1)\lmgf'(\htildeq/2)}{\htildeq}}{\lmgf'(\htildeq/2) - 2q} -\frac{3a^2}{2 \htildeq}.
\end{equation}
Expanding the left-hand side of~\eqref{Nantes C.1,5} to the second order, we get that the coefficient in front of $\gep^2$ equals
\beq
C:=\frac{3a^2}{2} {[\lmgf'(\htildeq/2) - 2q]} - (2a-1)\lmgf'(\htildeq/2) = \frac{\lmgf'(\htildeq/2)[\lmgf'(\htildeq/2) - 4q]}{2[\lmgf'(\htildeq/2)-2q]}.
\eeq
%
To conclude, it remains to observe that $C>0$, i.e.\ that $\lmgf'(\htildeq/2)> 4q$. Indeed, by definition of $\htildeq$ and strict convexity of $\lmgf$,
\begin{equation}
\begin{aligned}
	q &= \int_{-1/2}^{1/2}\left(t + \frac{1}{2} \right) \lmgf'(\htildeq t) \dd t= 
	\int_{0}^{1/2}\left(t + \frac{1}{2} \right) \lmgf'(\htildeq t) \dd t - 
	\int_{0}^{1/2}\left( \frac{1}{2} - t \right) \lmgf'(\htildeq t) \dd t \\
	&= \int_0^{1/2} 2t \lmgf'(\htildeq t) \dd t <  \lmgf'(\htildeq/2) \int_0^{1/2}2t \dd t = (1/4)\lmgf'(\htildeq/2). 
\end{aligned}
\end{equation}
\subsection{\texorpdfstring{Proof of~\eqref{secorder}}{Proof of the second order}}
\label{Nantes Appendix C.2}
To ease notation, we first set 
\beq
\mathsf{q}_\gep := \argmax\Big\{q>0\colon\,  \frac{1}{\sqrt{q}} \log \Gamma_\beta+ \frac{1}{\sqrt{q}} \psi(q,\delta_c(\beta) + \varepsilon)\Big\}, \qquad \gep \ge 0,
\eeq
and recall that $\psi(q,\gd)= \psi(q,0)$ for every $0\le \gd \le \gd_c(\gb)$.
In this section, we write $s(\delta,q) := \sqdelta$ to emphasize on the variable $q$ and $h(q) := \htildeq$. By Remark \ref{Nantes Remark 4.13}, we have on the one hand,
\beq \label{eq:-logG-expr1}
-\log \gG_\gb = \lmgf(h(\sfq_0)/2)
\eeq
and on the other hand,
\beq \label{eq:-logG-expr2}
\begin{aligned}
-\log \gG_\gb &= \lmgf\Big(s(\delta_c(\gb) + \varepsilon, \mathsf{q}_\varepsilon) + \gd_c(\gb)+\gep - \gb/2\Big)\\
&\stackrel{\eqref{closedexprdeltac}}{=} \lmgf\Big(s(\delta_c(\gb) + \varepsilon, \mathsf{q}_\varepsilon) + \gep - h(\sfq_0)/2\Big).
\end{aligned}
\eeq
Equating \eqref{eq:-logG-expr1} with \eqref{eq:-logG-expr2} and using that $\lmgf$ is  increasing on $(0,\gb/2)$, we obtain:
\begin{equation} \label{Nantes C.17}
s(\delta_c(\gb) + \varepsilon, \mathsf{q}_\varepsilon) + \varepsilon = h(\mathsf{q}_0).
\end{equation}%
By Lemma~\ref{lem:prop-Hgd-Hngd} and the Implicit Function Theorem,
there exist $\mathsf{a}$ and $\mathsf{b}$ (to be determined) such that
\begin{equation}
\label{eq:expansion-sfq}
\mathsf{q}_\gep = \mathsf{q}_0 + \mathsf{a} \varepsilon + \mathsf{b} \varepsilon^2 + o(\varepsilon^2), \qquad \text{as } \gep \to 0.
\end{equation}
We first compute $\mathsf{a}$ and $\mathsf{b}$ in~\eqref{eq:expansion-sfq}. Using \eqref{Nantes 3.7}, \eqref{Nantes 3.10} and \eqref{Nantes C.17}, one has:
\begin{equation}
	\sfq_\varepsilon = \int_0^1 t \lmgf'(h(\sfq_0)(t-1/2) + \varepsilon(1-t)) \dd t, \qquad \gep\ge0.
\end{equation}
Using a second-order Taylor expansion on $\sfq_\gep - \sfq_0$ gives:
\begin{equation} \label{Nantes C.19}
\begin{aligned}
	&\mathsf{a} = \int_0^1 t(1-t) \lmgf''(h(\sfq_0)(t-1/2))\dd t, \\
	&\mathsf{b} = \int_0^1 \frac{t(1-t)^2}{2} \lmgf'''(h(\sfq_0)(t-1/2))\dd t. 
\end{aligned}
\end{equation}
%
Using~\eqref{Nantes C.5} and Lemma~\ref{Nantes Lemma C.1} with the choice $a=1$ yields
\begin{equation}
\label{eq:expr-sfa-sfb}
\mathsf{a} = \frac{2\sfq_0}{h(\sfq_0)}, \qquad
\mathsf{b} = \frac{6\sfq_0-\lmgf'(h(\sfq_0)/2)}{2h(\sfq_0)^2}.
\end{equation}
We now have all the tools to establish~\eqref{secorder}. It comes:
\begin{equation}
\label{Nantes C.22}
\begin{aligned}
	&g(\beta,\delta_c(\beta)+\varepsilon)-g(\beta,\delta_c(\beta))  \\
&\stackrel{\eqref{eq:gbetadelta-varfor}}{=} \frac{\log \gG_\gb + \psi(\sfq_0, \gd_c(\gb))}{\sqrt{\sfq_0}}
-\frac{\log \gG_\gb + \psi(\sfq_\gep, \gd_c(\gb)+\gep)}{\sqrt{\sfq_\gep}}	\\
	&\stackrel{\eqref{exprepsitilde}+\text{Lemma~\ref{lem:technic_IPP}}}{=} \frac{ \log \Gamma_\beta +\lmgf\Big(s(\delta_c(\gb) + \gep, \sfq_\gep) + \delta_c(\gb) + \gep -\gb/2\Big) - 2\sfq_\gep s(\delta_c(\gb) + \gep,\sfq_\gep)}{\sqrt{\sfq_\gep}} \\
	&\qquad \qquad \qquad- \frac{\log \Gamma_\beta +\lmgf( h(\sfq_0)/2 ) - 2\sfq_0 h(\sfq_0)}{\sqrt{\sfq_0}} \\
	&\stackrel{\eqref{eq:-logG-expr1}-\eqref{eq:-logG-expr2}}{=} 2\sqrt{\sfq_0}h(\sfq_0) - 2\sqrt{\sfq_\gep} s(\gd_c(\gb)+\gep,\sfq_\gep)\\
	&\stackrel{\eqref{Nantes C.17}}{=} 2\sqrt{\sfq_0}h(\sfq_0) - 2\sqrt{\sfq_\gep}(h(\sfq_0) - \gep) \\
	&\stackrel{\eqref{eq:expansion-sfq}-\eqref{eq:expr-sfa-sfb}}{\sim} 
	\frac{\lmgf'(h(\sfq_0)/2)}{2h(\sfq_0)\sqrt{\sfq_0}} \times \gep^2,
\end{aligned}
\end{equation}
which completes the proof.
%
%
\section{Proof of Item (2) in Proposition~\ref{c1diffeo}}
\label{sec:corr-proof-diffeo}
In this section we correct a flaw in the proof of \cite[Lemma 5.3]{CNP16}. The flaw lies in the proof of the bijective nature of the map $\nabla\lmgf_\Lambda\colon \cD_\gb \to \bbR^2$ and is due to the fact that $\text{dist}(\mathbf{h},\partial\cD_\gb) \to 0$ does not necessarily imply that $\|\lmgf_\Lambda(\mathbf{h})\|\to \infty$, as it is claimed therein. Let us recall that for every $(h_0,h_1)\in \cD_\gb$, i.e.\ such that $|h_0|<\gb/2$ and $|h_0+h_1|<\gb/2$,
\beq
\nabla\lmgf_\Lambda(h_0,h_1) = \Big( \int_0^1 x\lmgf'(h_0x+h_1)\dd x, \int_0^1 \lmgf'(h_0x+h_1)\dd x \Big).
\eeq
Integrating by parts, we notice that, provided $h_0\neq 0$,
\beq
\nabla\lmgf_\Lambda(h_0,h_1) = \Big(
\frac{\lmgf(h_0+h_1)-\int_0^1\lmgf(h_0x+h_1)\dd x}{h_0},
\frac{\lmgf(h_0+h_1)-\lmgf(h_1)}{h_0}
\Big)
\eeq
Using the bijective change of variable $(h_0,h_1)\in \cD_\gb \mapsto (u,v)=(h_0+h_1,h_1)\in(-\gb/2, \gb/2)^2 $, we rather consider the function:
\beq
\mathsf{L}(u,v) := \nabla\lmgf_\Lambda(u-v, v), \qquad (u,v)\in (-\gb/2, \gb/2)^2,
\eeq
and notice that, when $u\neq v$,
\beq\label{eq:diffeo1}
\mathsf{L}(u,v) = \Big(
\frac{\lmgf(u)-\int_0^1\lmgf((u-v)x+v)\dd x}{u-v},
\frac{\lmgf(u)-\lmgf(v)}{u-v}
\Big).
\eeq
We start with a technical lemma:
\begin{lemma}\label{lem:control-lmgf}
For every $(h_0,h_1)\in \cD_\gb$ such that $h_0\neq 0$,
\beq\label{eq:control-lmgf}
\int_0^1 \lmgf(h_0 x + h_1) \dd x \le \frac{1}{|h_0|} \int_{-\gb/2}^{\gb/2} \lmgf(x)\dd x.
\eeq
\end{lemma}
Note that the integral on the right-hand side of~\eqref{eq:control-lmgf} is finite.
\begin{proof}[Proof of Lemma~\ref{lem:control-lmgf}]
A straightforward change of variable yields:
\beq
\int_0^1 \lmgf(h_0 x + h_1) \dd x = \frac{1}{|h_0|}\int_{h_0 \wedge (h_0+h_1)}^{h_0 \vee (h_0+h_1)} \lmgf(x) \dd x \le \frac{1}{|h_0|} \int_{-\gb/2}^{\gb/2} \lmgf(x)\dd x.
\eeq
\end{proof}
We now investigate the properties of the first coordinate of the function $\mathsf{L}$ when the second variable $v$ is fixed:
\begin{lemma}\label{lem:control-lmgf1}
For every $|v|<\gb/2$, the mapping $u\mapsto \mathsf{L}_1(u,v)$ is bijective from $(-\gb/2, \gb/2)$ to $\bbR$.
\end{lemma}
\begin{proof}[Proof of Lemma~\ref{lem:control-lmgf1}]
Since $\lmgf$ is strictly convex, we first notice that
\beq
(\partial_u\mathsf{L}_1)(u,v) = \int_0^1 x^2 \lmgf''((u-v)x+v) \dd x >0.
\eeq
Thus it remains to prove the limits $\mathsf{L}_1(\pm\gb/2,v)= \pm\infty$. Using~\eqref{eq:diffeo1} and Lemma~\ref{lem:control-lmgf}, we may write:
\beq
\mathsf{L}_1(u,v) \ge \frac{\lmgf(u)}{u-v} - \frac{\cst}{|u-v|^2}.
\eeq
As $u\to \gb/2$, the term $u-v$ eventually becomes positive, bounded away from $0$ and infinity, and $\lmgf(u)$ converges to $+\infty$, hence the result. The proof for the case $u\to -\gb/2$ follows the same idea, writing instead:
\beq\mathsf{L}_1(u,v) \le \frac{\lmgf(u)}{u-v} + \frac{\cst}{|u-v|^2}.
\eeq
\end{proof}
From Lemma~\ref{lem:control-lmgf1}, for every $q\in\bbR$ we may define $u(v,q)$ as the unique solution in $(-\gb/2, \gb/2)$ of
\beq\label{eq:diffeo2}
\mathsf{L}_1(u(v,q),v) = q. 
\eeq
The proof will be complete once we prove that for every (fixed) $q\in\bbR$, $v\mapsto \mathsf{L}_2(u(v,q),v)$ is a bijective map from $(-\gb/2,\gb/2)$ to $\bbR$. We first argue that this (continuous) map is injective. Indeed, suppose by contradiction that there exists $v_1\neq v_2$ such that $\mathsf{L}_2(u(v_1,q),v_1) = \mathsf{L}_2(u(v_2,q),v_2)$. Coming back to the old variables, this would imply that the map $(h_0, h_1) \mapsto \int_0^1 \lmgf'(h_0 x + h_1)$ itself is not injective, which contradicts the fact that $(h_0, h_1) \mapsto \int_0^1 \lmgf(h_0 x + h_1)$ is strictly convex. It now remains to prove, in order to conclude, that
\beq
\lim_{v\to \pm \gb/2} \mathsf{L}_2(u(v,q),v)  = \pm \infty.
\eeq
To this end, we first differentiate~\eqref{eq:diffeo2} with respect to $v$ and get
\beq
(\partial_vu)(v,q) = -\frac{\int_0^1 x(1-x)\lmgf''(u(v,q)x+v(1-x))\dd x}{\int_0^1 x^2\lmgf''(u(v,q)x+v(1-x))\dd x} <0,
\eeq
(by strict convexity of $\lmgf$), from which we infer that $v\mapsto u(v,q)$ is decreasing. Consequently, $|u(v,q)-v|$ remains bounded away from zero as $v\to \pm \gb/2$. Using~\eqref{eq:diffeo1} and~\eqref{eq:diffeo2}, we finally get:
\beq
\mathsf{L}_2(u(v,q),v) = \frac{\lmgf(v)}{u(v,q)-v} - q - \frac{1}{u(v,q)-v}\int_0^1\lmgf((u(v,q)-v)x+v)\dd x,
\eeq
and, using Lemma~\ref{lem:control-lmgf},
\beq
|\mathsf{L}_2(u(v,q),v)| \ge \frac{\lmgf(v)}{|u(v,q)-v|} - |q| - \frac{\cst}{|u(v,q)-v|^2},
\eeq
which converges to $+\infty$ as $v\to\pm \gb/2$.
\bibliographystyle{plain}
\bibliography{biblio.bib}
\end{document}